\documentclass[oneside,a4paper,limits]{amsart}
\usepackage{graphicx}
\usepackage{euscript,mathrsfs}
\usepackage{dsfont}
\usepackage{xcolor}
\usepackage{empheq}
\usepackage{geometry}
\usepackage{color}
\catcode`\@=11 \@addtoreset{equation}{section}

\catcode`\@=12

\usepackage{braket,amsfonts}

\usepackage{tikz}
\usetikzlibrary{patterns}
\usepackage[normalem]{ulem}
\usepackage{cancel}

\usepackage{scalerel,stackengine}
\stackMath
\newcommand\reallywidehat[1]{%
\savestack{\tmpbox}{\stretchto{%
  \scaleto{%
    \scalerel*[\widthof{\ensuremath{#1}}]{\kern-.6pt\bigwedge\kern-.6pt}%
    {\rule[-\textheight/2]{1ex}{\textheight}}
  }{\textheight}%
}{0.5ex}}%
\stackon[1pt]{#1}{\tmpbox}%
}

\newtheorem{Theorem}{Theorem}[section]

\newtheorem{Lemma}[Theorem]{Lemma}
\newtheorem{Corollary}[Theorem]{Corollary}

\newtheorem{Definition}[Theorem]{Definition}

\newtheorem{Remark}[Theorem]{Remark}

\newcommand{\hheps}{h^\varepsilon}

\newcommand{\eps}{\TS}

\newcommand{\rh}{r_\eps}
\newcommand{\rht}{r_{\heps,\eps}}
\newcommand{\gh}{g_\eps}
\newcommand{\bfh}{\bff_\eps}

\newcommand{\hrh}{\widehat{r}_\eps}

\newcommand{\vh}{v_\eps}
\newcommand{\vrh}{\vr_\eps}
\newcommand{\vuh}{\vu_\eps}
\newcommand{\vvh}{\vv_\eps}
\newcommand{\vwh}{\bfw_\eps}
\newcommand{\etah}{\eta_\eps}

\newcommand{\zh}{z_\eps}

\newcommand{\hv}{\widehat{v}}
\newcommand{\hvr}{\widehat{\vr}}
\newcommand{\hvu}{\widehat{\vu}}
\newcommand{\hvv}{\widehat{\vv}}
\newcommand{\heta}{\widehat{\eta}}
\newcommand{\hvw}{\widehat{\vw}}

\newcommand{\hbf}{\widehat{\bff}}

\newcommand{\hvh}{\hv_{\TS}}

\newcommand{\hvuh}{\hvu_{\TS}}

\newcommand{\hvwh}{\hvw_{\TS}}
\newcommand{\hbfh}{\hbf_{\TS}}

\newcommand{\hvht}{\hv_{h,\TS}}
\newcommand{\hvrht}{\hvr_{h,\TS}}
\newcommand{\hvuht}{\hvu_{h,\TS}}
\newcommand{\hvvht}{\hvv_{h,\TS}}
\newcommand{\hvwht}{\hvw_{h,\TS}}

\newcommand{\bfPsih}{\bfPsi_\eps}

\newcommand{\phih}{\phi_\eps}
\newcommand{\Phih}{\Phi_\eps}
\newcommand{\varphih}{\varphi_\eps}

\newcommand{\Psie}{\Psi_{\epsilon}}
\newcommand{\Psieh}{\Psi_{\epsilon,\heps}}
\newcommand{\psih}{\psi_\eps}

\newcommand{\heps}{{h}}
\newcommand{\pht}{p_{\heps,\eps}}
\newcommand{\ght}{g_{\heps,\eps}}
\newcommand{\bfht}{\bff_{\heps,\eps}}

\newcommand{\vht}{v_{\heps,\eps}}
\newcommand{\vrht}{\vr_{\heps,\eps}}
\newcommand{\vuht}{\vu_{\heps,\eps}}
\newcommand{\vvht}{\vv_{\heps,\eps}}
\newcommand{\vwht}{\bfw_{\heps,\eps}}
\newcommand{\etaht}{\eta_{\heps,\eps}}

\newcommand{\zht}{z_{\heps,\eps}}

\newcommand{\Uht}{U_{\heps,\eps}}
\newcommand{\Oht}{\Omega_{\heps,\eps}}

\newcommand{\bfw}{\mathbf{w}}
\newcommand{\bfu}{\mathbf{u}}

\newcommand{\bfq}{\mathbf{q}}

\newcommand{\bfe}{\mathbf{e}}

\newcommand{\bfy}{\mathbf{y}}

\newcommand{\mJ}{\mathcal{F}}
\newcommand{\Jacob}{\mathbb{J}}
\newcommand{\mJkm}{\mJ_k^{k-1}}
\newcommand{\mJik}{\mJ_0^{k}}
\newcommand{\mJikm}{\mJ_0^{k-1}}

\newcommand{\bfx}{ x }
\newcommand{\bfX}{ X}
\newcommand{\bfr}{r}

\newcommand{\bfn}{\mathbf{n}}
\newcommand{\bff}{\mathbf{f}}
\newcommand{\xref}{\widehat{x}}
\newcommand{\bfxref}{\widehat{\bfx}}
\newcommand{\bfXkm}{\bfX_k^{k-1}}

\newcommand{\bfF}{\mathbf{F}}

\newcommand{\bfD}{\mathbf{D}}
\newcommand{\bfS}{\mathbf{S}}
\newcommand{\bfI}{\mathbf{I}}

\newcommand{\bfPhi}{\boldsymbol{\Phi}}

\newcommand{\bfPsi}{\boldsymbol{\Psi}}

\newcommand{\bftau}{\boldsymbol{\tau}}

\newcommand{\hvarphi}{\widehat{\varphi }}

\newcommand{\hbfPsi}{\widehat{\bfPsi }}

\newcommand{\hpsi}{\widehat{\psi }}

\newcommand{\bfPhiht}{{\bfPhi }_{h,\eps}}

\newcommand{\varphiht}{{\varphi }_{h,\eps}}

\newcommand{\bfPsiht}{{\bfPsi }_{h,\eps}}
\newcommand{\psiht}{{\psi }_{h,\eps}}

\newcommand{\hvarphiht}{\widehat{\varphi }_{h,\eps}}

\newcommand{\hbfPsiht}{\widehat{\bfPsi }_{h,\eps}}
\newcommand{\hpsiht}{\widehat{\psi }_{h,\eps}}

\newcommand{\Div}{{\rm div}}
\newcommand{\Grad}{\nabla}
\newcommand{\Lap}{ \Delta}
\newcommand{\Divh}{\Div}
\newcommand{\Gradh}{\Grad}

\newcommand{\Divup}{\Div^{\rm up}_\eps}

\newcommand{\Du}{\bfD(\bfu)}

\newcommand{\avg}[1]{ \left< #1 \right>_\sigma }
\newcommand{\avc}{\Pim}

\newcommand{\jump}[1]{\left[\! \left[#1\right] \! \right]}
\newcommand{\Abs}[1]{ \left| #1 \right|}
\newcommand{\abs}[1]{ | #1 |}
\newcommand{\norm}[1]{\left\lVert#1\right\rVert}
\newcommand{\co}[2]{{\rm co}\{ #1 , #2 \}}
\newcommand{\Ov}[1]{\overline{#1}}
\newcommand{\Ovg}[1]{\overline{#1}_\sigma}

\newcommand{\seb}[1]{{#1}}

\newcommand{\Pim}[1]{\Pi_\grid [#1]}
\newcommand{\Pie}[1]{\Pi_\edges [#1]}
\newcommand{\Pit}[1]{\Pi_t [#1]}
\newcommand{\Pip}[1]{\Pi_p [#1]}

\newcommand{\aleq}{\stackrel{<}{\sim}}

\newcommand{\vr}{\varrho}
\newcommand{\vu}{\bfu}
\newcommand{\vw}{\bfw}

\newcommand{\vc}[1]{{\bf #1}}

\newcommand{\vv}{\vc{v}}

\newcommand{\Q}{\mathbb{Q}}
\newcommand{\R}{\mathbb{R}}

\definecolor{Cgrey}{rgb}{0.85,0.85,0.85}
\definecolor{Cblue}{rgb}{0.50,0.85,0.85}
\definecolor{Cred}{rgb}{1,0,0}
\definecolor{fancy}{rgb}{0.10,0.85,0.10}
\definecolor{forestgreen}{rgb}{0.13, 0.55, 0.13}
\newcommand{\cred}{\color{Cred}}

\newcommand{\ds}{{\,d\bfr}}
\newcommand{\dSx}{{\rm d}S(x)}
\newcommand{\dSxref}{{\rm d}S(\widehat{x})}

\newcommand{\TS}{\tau}

\newcommand{\dx}{\,{\rm d} {x}}
\newcommand{\dxref}{\,{\rm d} \widehat{x}}

\newcommand{\dt}{\,{\rm d} t }
\newcommand{\dxdt}{\dx \ \dt}
\newcommand{\pd}{\partial}
\newcommand{\pdt}{\pd _t}
\newcommand{\Hc}{\mathcal{H}}
\newcommand{\faces}{\mathcal{E}}

\newcommand{\facesK}{\faces(K)}

\newcommand{\facesint}{\edgesint}

\newcommand{\edges}{\faces}
\newcommand{\edgesint}{\edges_{\rm I}}
\newcommand{\edgesext}{\edges_{\rm E}}

\newcommand{\edgesextS}{\edges_{\rm S}}

\newcommand{\Sh}{\Sigma_h}
\newcommand{\Th}{\mathcal{T}_h}
\newcommand{\Thk}{\mathcal{T}_h^{k}}

\newcommand{\Tht}{\mathcal{T}_{h,\eps}}
\newcommand{\Xspace}{Q_h}
\newcommand{\Yspace}{V_h}
\newcommand{\Wspace}{W_h}
\newcommand{\hXspace}{\widehat{Q}_h}
\newcommand{\hYspace}{\widehat{V}_h}
\newcommand{\hWspace}{\widehat{W}_h}

\newcommand{\Thz}{\mathcal{T}_h^0}

\newcommand{\Qh}{Q}
\newcommand{\Vh}{V}
\newcommand{\Wh}{W}

\newcommand{\ALE}{\mathcal{A}}
\newcommand{\ALEh}{\ALE_{\eps}}
\newcommand{\ALEhk}{\ALEh^{k}}
\newcommand{\ALEhkm}{\ALEh^{k-1}}

\newcommand{\ALEht}{\ALE_{h,\eps}}
\newcommand{\ALEhtk}{\ALE_{h,\eps}^{k}}
\newcommand{\ALEhtkm}{\ALE_{h,\eps}^{k-1}}

\newcommand{\DT}{\mathcal{ J}}          

\newcommand{\order}{\mathcal{O}}
\newcommand{\grid}{\mathcal{T}}

\newcommand{\Ohk}{\Oh^k}
\newcommand{\Oh}{\Omega_\eps}
\newcommand{\Oref}{\widehat{\Omega}}
\newcommand{\Ohref}{\widehat{\Omega}_\eps}

\newcommand{\ints}[1] {\int_{\sigma} #1 \dSx }
\newcommand{\intsh}[1] {\int_{\sigma} #1 \dSx }
\newcommand{\intS}[1] {\int_{\Sigma} #1 \ds }
\newcommand{\intSh}[1] {\int_{\Sh} #1 \ds }
\newcommand{\intO}[1]{\int_{\Oh} #1 \dx}
\newcommand{\intOt}[2]{\int_{\Oh(#1)} #2 \dx}
\newcommand{\intOtau}[2]{\int_{\Oh(#1)} #2 \dx}

\newcommand{\intshref}[1] {\int_{\sigma} #1 \dSxref }

\newcommand{\intOref}[1]{\int_{\Oref} #1 \dxref}
\newcommand{\intOk}[1]{\int_{\Oh^k} #1 \dx}
\newcommand{\intOkm}[1]{\int_{\Oh^{k-1}} #1 \dx}

\newcommand{\intOhk}{\intOk} 
\newcommand{\intOhkm}{\intOkm}

\newcommand{\intEhref}{\sum_{\sigma \in \edgesint}\intshref}

\newcommand{\intEh}{\sum_{\sigma \in \edgesint}\intsh}
\newcommand{\intEhk}{\sum_{\sigma \in \edgesint^{k}}\intsh}
\newcommand{\intEhkKL}{\sum_{\sigma=K|L \in \edgesint^{k}}\intsh}

\newcommand{\intK}[1]{\int_{K} #1 \ \dx}
\newcommand{\sumK}{\sum_{K \in \Th}}
\newcommand{\sumintK}{\sum_{K \in \Thk}\intK}
\newcommand{\sumintKs}{\sum_{K \in \Thk}\sum_{\sigma \in \pd K}\intsh}

\newcommand{\intTO}[1]{\int_0^T \int_{\Oh} #1 \ \dxdt}

\newcommand{\sumEKh}[1]{ \sum_{K \in \Tht} \sum_{\sigma \in \facesK }\int_{\sigma} #1 \dSx }

\def\Xint#1{\mathchoice 
{\XXint\displaystyle\textstyle{#1}}%
{\XXint\textstyle\scriptstyle{#1}}%
{\XXint\scriptstyle\scriptscriptstyle{#1}}%
{\XXint\scriptscriptstyle\scriptscriptstyle{#1}}%
\!\int} 
\def\XXint#1#2#3{{\setbox0=\hbox{$#1{#2#3}{\int}$} 
\vcenter{\hbox{$#2#3$}}\kern-.5\wd0}} 
\def\dashint{\,\Xint-}

\begin{document}


\title{On numerical approximations to fluid-structure interactions involving compressible fluids}
\thanks{Both authors thank the support of the primus research programme of Charles University, PRIMUS/19/SCI/01 and the program GJ17-01694Y of the Czech national grant agency (GA\v{C}R). 
 S.~S also thanks the University Centre UNCE/SCI/023 from Charles University. The institute of Mathematics of the Czech Academy of Sciences is supported by RVO: 67985840. 
}

\author{Sebastian Schwarzacher \and Bangwei She
}
\address{Department of Analysis, Faculty of Mathematics and Physics, Charles University (schwarz, she)@karlin.mff.cuni.cz.}


\maketitle

%
%
%
%
%
%

\begin{abstract}
In this paper we introduce a numerical scheme for fluid-structure interaction problems in two or three space dimensions: A flexible elastic plate is interacting with a viscous, compressible barotropic fluid. Hence the physical domain of definition (the domain of Eulerian coordinates) is changing in time. We introduce a fully discrete scheme that is stable, satisfies geometric conservation, mass conservation and the positivity of the density. We also prove that the scheme is consistent with the definition of continuous weak solutions.

\noindent\textsc{Keywords:}
fluid-structure interaction, compressible Navier--Stokes, stability, consistency, finite elements,  elastic plates, incremental time-stepping


\noindent
\textsc{MSC (2010):
35Q30, 76N99, 74F10, 65M12, 65M60
}
\end{abstract}

\section{Introduction}\label{sec:1}
In the recent decades, there is an increasing attendance of mathematicians on the subject of fluid--structure interaction (FSI) problems due to their numerous applications. This includes blood flow through a vessel, oil flows through an elastic pipe but also oscillations of suspension bridges, lifting of airplanes, bouncing of elastic balls or the rotation of wind turbines, see~\cite{BelGalKye13, FSIforBIO,  Canic1, kamakoti2004fluid} and the references therein. 

We will consider the particular setting where the solid (or the structure) is a shell or a plate. This means that it is modeled as a thin object of one dimension less than the fluid. For related up-to-date {\em modeling and model reductions} on {\em plates} and {\em shells} see \cite{CiaII,CSS2,K1} and the references therein. The fluid will be considered to be governed by the compressible Navier--Stokes equation. We are interested in the development of {\em Galerkin schemes} which are connected to the setting of {\em weak solutions}. 
Most of the mathematical effort in this setting so far was devoted to incompressible fluids for {\em weak solutions} with a {\em fixed prescribed scalar direction of displacement of the shell}. Well posedness results commonly show that a weak solution exists {\em until a self-touching} of the solid is approached. For {\em incompressible Newtonian fluids} we name the following results~\cite{Bou07,Boris1,DE,DEGLT,CG,GraHil16,  Daniel, BorSun13, SunBorMulti,MuhSch19}. On the other hand, the theory for {\em compressible flows} is much less developed.  Only recently the existence of weak solutions in the above setting was shown~\cite{BS}, see also \cite{TW}.

The numerical results of fluid-structure interactions are rich and diverse. The numerical analysis for the incompressible flows is developed in accordance with the existence theory, see the kinematical splitting schemes developed in \cite{Boris1,Boris2,Boris3,Lukacova}, see also \cite{Canic1,FRW, Hundertmark, Ric17} for more simulation results. 
Without a surprise, the numerical theory for compressible fluids interacting with shells or plates is quite sparse. We mention~\cite{BF,DFKV} for the stability analysis with a given variable geometry and~\cite{FKP, Kosik} for some numerical simulations. 
It seems that a numerical strategy for compressible flows interacting with elastic structures  stayed undeveloped due to the high nonlinearity of the problem originating from the fluid and its sensitive coupling to its geometry. 

This paper aims to fill that gap and enrich the theory on fluid-structure interactions by introducing a (fully discrete) numerical approximation scheme which is in coherence with the known continuous existence theory. 
In particular we study numerics for the interaction between a compressible barotropic fluid flow with an elastic shell in the time-space domain $I \times \Omega(t)$, where 
$\Omega=\Omega(t)\subset R^{d}$  ($d\in\{2,3\}$, $t\in I=[0,T])$ is a time dependent domain defined by its unsteady boundary. 
The boundary of $\Omega$ consists of a time dependent elastic shell $\Gamma_S(t)$ on the top surface of the fluid (whose projection in $d^{th}$-direction is $\Sigma$ given below), and fixed solid walls $\Gamma_D=\pd\Omega \backslash\Gamma_S $ for the other parts of the boundary. Throughout the paper we reserve $\bfr=(x_1,\ldots, x_{d-1})$ as the coordinates for the plate displacement $\eta : \Sigma \to \mathbb{R}$, i.e. the distance of the shell above the horizontal plane $x_d=H$. We define $\bfx= (\bfr,x_d) $ as the Eulerian coordinates in the domain 
\[
\Omega(t):=\{(\bfr,x_d)\in \Sigma \times \mathbb{R}\,:\, 0<x_d<H+\eta(\bfr)\}.
\]
We denote by $\Oref = \Omega^0=\Sigma \times[0,H]$ the reference domain, with $\Sigma= [0,L_1] \times \cdots \times[0,L_{d-1}]$ being a $(d-1)$-dimensional time-independent domain. Accordingly we introduce the following the one-to-one invertible mapping  
\begin{equation}\label{ALE_mapping}
\ALE:\Oref(\bfxref) \rightarrow \Omega(\bfx), \quad 
\bfx=\ALE(t,\bfxref)=\ALE(t, \widehat{\bfr},\widehat{x}_d) =  \left(\widehat{\bfr}, \widehat{x}_d\frac{H+\eta(t,\bfr)}{H}\right).
\end{equation}
Here and hereafter, we distinguish the functions on the reference domain by the superscript `` {$\widehat{ } $} '' except the ALE mapping.  
We denote by $\widehat{\nabla}\equiv \nabla_{\bfxref}$ and $\widehat{\Div}\equiv \Div_{\bfxref}$. Furhter, we denote  $\Jacob$ and $\mJ$ as the Jacobian of the mapping $\ALE$ and its determinant: 
\[ \Jacob = \widehat{ \Grad} \ALE, \quad \mJ =\det(\Jacob). 
\]
We present Figure \ref{fig_ALE} for a two dimensional example of the domain and ALE mapping. 
\begin{figure}[!h]\centering
\vspace{-1.5cm}
\begin{tikzpicture}[scale=1.0]
\draw[->] (0,0)--(6.4,0);
\draw[->] (0,0)--(0,2.4);
\draw[very thick](0,0)--(6,0)--(6,2)--(0,2)--(0,0);
\draw[very thick, blue](0,2)--(0,0)--(6,0)--(6,2);
\draw[very thick, red](0,2)--(6,2);
\path node at (3,1) {$\Omega(0) =\Oref$};
\path node at (-0.3,1) {\textcolor{blue}{$\Gamma_D$}};
\path node at (3,2.3) {\textcolor{red}{$\Gamma_S(0)=\Sigma$}};
\path node at (-0.2,-0.2) {$0$};
\path node at (6,-0.3) {$L$};
\path node at (-0.3,2) {$H$};
\draw[very thick,->] (6.3,1)--(7.7,1);
\path node at (7,1.3) {$\ALE$};
\draw[thick,red]   (8, 2) .. controls (9.5, 0.5) and (12.5, 4) .. (14, 2);
\draw[thick,dashed](8,0)--(14,0)--(14,2)--(8,2)--(8,0);
\draw[very thick, blue](8,2)--(8,0)--(14,0)--(14,2);
\path node at (11,1) {$\Omega(t)$};
\path node at (10,2.2) {\textcolor{red}{$\Gamma_S(t)$}};
\draw[<->,cyan ] (12.5,2)--(12.5,2.6);
\path node at (12.7,2.3) {\textcolor{cyan}{$\eta$}};
\end{tikzpicture}
\vspace{-0.5cm}
\caption{Time dependent domain and the ALE mapping}\label{fig_ALE}
\end{figure}
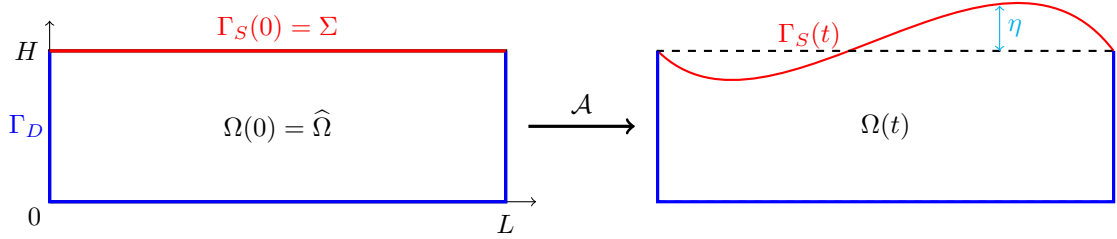

The evolution of the  fluid flow is modeled by the Navier--Stokes system
\begin{subequations}\label{pde}
 \begin{equation}\label{pde_D}
  \partial_t \vr+\Div(\vr\bfu ) =0,\text{ in }  I\times \Omega,
 \end{equation}
  \begin{equation} \partial_t(\vr\bfu )+\Div(\vr\bfu \otimes\bfu ) =\Div \bftau 
  +\vr \bff, \text{ in } I\times \Omega.\label{pde_M}
  \end{equation}
where $\vr = \vr(t,\bfx)$ is the fluid density and $\bfu = \bfu(t, \bfx)$ is the velocity field, $\bftau$ is the Cauchy stress 
\[\bftau=\bfS(\Grad \vu) -p(\vr)\bfI, \quad  \bfS(\Grad \vu) =2 \mu \Du + \lambda\Div\bfu\; \bfI,\quad \Du=\frac{\Grad \bfu + \Grad^T \bfu}{2},\quad p(\vr)=a \vr^{\gamma}
\]
for $a>0$, $\gamma >1$. 
The viscosity coefficients satisfy $\mu>0$ and $ \mu+\lambda \geq 0$. 
The motion of the shell is given by 
\begin{align}\label{pde_S}
\varepsilon_0\vr_S\partial_t z + K'(\eta) &=
  g+\bfe_d \cdot \bfF,\quad z =\pdt \eta, \quad\text{on}\quad I\times \Sigma, 
\end{align}
where $ z$ is the velocity of the shell, $\vr_S>0$ is the density of the shell,  $\bfe_d=(0,0,1)^T$ for $d=3$ ($\bfe_d=(0,1)^T$ for $d=2$), 
$g=g(t,\bfr)$ is a given function, 
$ \bfF= - \big(\bftau \cdot \bfn \big)\circ \ALE \ \DT$, $\bfn$ is outer normal and $\bftau =\bfS-p \bfI$ is the Cauchy stress. 
 For the sake of simplicity, we assume throughout the paper that  
$ \varepsilon_0\vr_S=1$.  
As elastic energy $ K(\eta)$ we use the following linearized  energy
$ K(\eta)= \frac{\alpha|\nabla^2 \eta|^2}{2}+\frac{\beta|\nabla \eta|^2}{2}, \; \alpha >0,\; \beta\geq 0$, 
which leads to the following $L^2$-Gradient:
\[K'(\eta)=\alpha \Delta^2 \eta - \beta\Delta \eta. \]

We refer to Ciarlet and Roquefort~\cite{CiaII} and references therein for the details of the model and also other choices of $K(\eta)$. 
To close the system we propose the following boundary conditions and initial data
\begin{equation}\label{initials}
 \bfu|_{\Gamma_D} =0, \;   \eta|_{\pd \Sigma}=0,\; \nabla \eta|_{\pd \Sigma}=0, \quad 
  \vr(0)=\vr_0,\; (\vr \vu)(0) =\bfq_0 \text{ in } \Omega(0), \quad 
\eta(0,\cdot)=\eta_0,\;  z (0,\cdot)=z_0 \; \text{in}\; \Sigma,
\end{equation}
where $\eta_0(x_1)= 0$ and $z_0(x_1)$ is a given function.  
We also need a compatibility condition between the shell and the fluid
\begin{equation*}
 \bfu(t,\bfx)|_{\Gamma_S}\circ \ALE =   z (t,\bfr) \bfe_d. 
\end{equation*}  
\end{subequations}


The purpose of the present paper is to introduce a fully discrete numerical scheme that is equipped with suitable physical and mathematical properties. By that we mean that it satisfies in particular:
\begin{itemize}
\item[(a)] A fully discretized weak continuity equation that can be renormalized in the sense of DiPerna and Lions, such that the error for convex renormalizations is positive.
\item[(b)] Mass conservation and positivity of the discrete density is preserved.
\item[(c)] A fully coupled momentum equation in the spirit of Definition~\ref{def_ws} on the discrete level.
\item[(d)] A discrete energy inequality for the coupled system (analogous to the continuous energy inequality~\eqref{energy_estimates_continuous}). 
\item[(e)] The scheme is consistent with the continuous weak solutions introduced in~\cite{BS} (See also Definition~\ref{def_ws}). This means in particular, that if the discrete deformation, density and velocity converge (strongly) to some limit triple, this limit triple is indeed a weak solution of the continuous problem. 
\item[(f)] The scheme exists for a minimal time-interval. I.e.\ for every $\delta_0\in (0,H/2)$ there is a minimal time $T_0$, such that a-priori $\inf_{[0,T_0]\time \Sigma}\eta(t,\bfr)\geq \delta_0-H$.
\end{itemize}
The existence of weak solutions for compressible viscous barotropic fluids interacting with an elastic plate is only recent~\cite{BS}. It follows the seminal existence proof for weak solutions of the compressible Navier--Stokes equations~\cite{Li2,feireisl1}. Note that the existence approach introduced in~\cite{BS} can not be adapted to numerical approximations in a straight forward manner since it uses fixed point theorems and regularization operators on the continuous level. Indeed, the introduction of a numerical scheme that satisfies all conditions above turns out to be rather sophisticated. In particular, in order to capture the material time-derivative at the interacting interface, we have to introduce a corrector flow field (the function $\vw$, below) that depends (linearly) on the elastic deformation $\eta$ which allows to approximate the material derivative (the Eulerian time derivative) of the elastic solid. 

We will consider a fully coupled implicit Euler scheme with respect to the time derivative. The spatial discretization of the deformation is done by piecewise polynomials. 
All three quantities are prescribed w.r.t\ a fixed steady reference mesh. This provides a nonlinear system for which we can prove the existence of solutions every time--step via a homotopy argument (See Theorem~\ref{thm:exist}). 

The {\em critical highlight} of the present paper is given in Theorem~\ref{thm_s1} and Theorem~\ref{thm_s2} where it is shown that the introduced scheme satisfies a discrete version of the energy inequality.
It turns out that for compressible fluids only a fully non-linear implicit scheme does satisfy an energy inequality (See Remark~\ref{rem:implicit}). This is in contrast to incompressible fluids, which can be linearized~(see e.g.\ \cite{Boris1}). While the strategy to get energy stable schemes for the compressible barotropic Navier--Stokes system is quite standard if the fluid domain $\Omega$ is fixed, see e.g. \cite{GallouetMAC, HS_MAC, Karper},  it becomes rather difficult when a time dependent domain is considered. We would like to mention here the stability results of \cite{BF,DFKV} where the moving domain is a given function.  
As far as we know, this is the first result on energy stable and mass conservative numerical solutions for the FSI problem with compressible fluids even in two space dimensions. The {\em technical highlight} is the consistency of solutions, see Theorem~\ref{Thm_C1} and Theorem~\ref{Thm_C2}. This is due to the fact that in case of fluid-structure interaction the space of test function is {\em a part of the weak solution} (see Definition~\ref{def_ws}). For that one has to ensure that the space of test functions of the limit weak solution (that depends on the limit geometry) can indeed be approximated. For that reason, the consistency of solutions is sensitive to the regularity of solutions--hence the consistency is the only part of the paper where there are restrictions on the barotropic exponent $\gamma$. In the fully discrete case the restriction is that $\gamma>\frac{6}5$ in the semi-discrete setup there is no restriction on $\gamma$.

The {\bf plan of the paper} is the following. In Section 2 we introduce the necessary analysis for the incremental time-stepping approximation, Section 3 is dedicated to the semi-discrete scheme, Section 4 to the fully discrete scheme.

The {\bf main result} of the present paper is the existence of numerical solutions which satisfy (a)--(e) stated above. We first introduce the troubles related to the variable Eulerian geometry by studying the case of a discrete in time, but continuous in space model for which we prove properties (a)--(e). In the second part of the paper we study the fully discrete case for which (a)--(e). 
For the better readability we state here where the respective results are shown:
\begin{itemize}
\item[(a)] See Lemma~\ref{lem_r0} (semi-discrete) and Lemma~\ref{lem_r1} (fully discrete) for the renormalized equation. 
\item[(b)] See \eqref{MC} and Lemma~\ref{lem_non-negative} for conservation of mass and non-negativity of the density for the semi-discrete case; 
see \eqref{MCFS} and Lemma~\ref{lem_positive} for conservation of mass  and positivity of the density for the fully discrete case. 
\item[(c)] See Definition~\ref{def_SS} (semi-discrete) and Definition~\ref{def_FS} (fully discrete) for the fully coupled momentum problem.
\item[(d)] See Theorem~\ref{thm_s1} (semi-discrete) and Theorem~\ref{thm_s2} (fully discrete) for the energy inequality.
\item[(e)] See Theorem~\ref{Thm_C1} (semi-discrete) and Theorem~\ref{Thm_C2} (fully discrete) for the consistency of the schemes.
\item[(f)] See Theorem~\ref{thm:exist} for the existence of a numerical solution to the fully discrete scheme. See Lemma~\ref{lem:Tsd} (semi-discrete) and Corollary~\ref{cor:maxtimeh} (fully discrete) for the minimal time interval of existence.
\end{itemize}

We wish to point out that the scheme is built in such a way that one may prove that any subsequence of a numerical approximation converges weakly to a continuous solution.\footnote{Please observe, that the lower bound on $\gamma$ is the very same as was requested in~\cite{BS}.} The convergence result for the very same scheme will be the content of an independent paper.


\section{Preliminaries}\label{sec2}
In this section, we introduce the necessary notations, the time discretization and time difference operators. 
\subsection*{Weak solutions}
We begin by introducing the following concept of weak solutions developed in~\cite{BS, TW} where  the existence of weak solutions (until a self-contact of the boundary) under appropriate initial conditions was shown.  Indeed, existence could be shown in the following {\em continuous} spaces
\begin{itemize}
    \item The {\em deformation} is usually assumed to be in the following Bochener space\footnote{Throughout the paper we make use of the standard notation of Bochner spaces, Sobolev spaces and Lebesgue spaces, see for instance~\cite{FeireislNovotny09} for more details.} $\eta \in W^I:=L^2(0,T;W^{2,2}_0(\Sigma)) \cap W^{1,2}(0,T;L^2(\Sigma))$.
    \item The {\em density} $\vr\in Q^I$, were $Q^I:=L^\infty(0,T;L^\gamma(\Omega(t))$. This means that $\vr(t)\in L^\gamma(\Omega(t))$ for almost every $t$ and that the essential supremum over the respective norms is bounded.
    \item The {\em velocity} $\vu\in V^I:=\{\vu\in L^2(0,T;W^{1,2}(\Omega(t))\, :\, \vu(\bfr,H+\eta(\bfr))=\partial_t \eta(\bfr)\bfe_d\text{ for all }\bfr\in \Sigma \text{ and } \vu\equiv 0 \text{ on }\Gamma_D\}
$
    \end{itemize}
\begin{Definition}[Weak solution]\label{def_ws}  
A weak solution to \eqref{pde}--\eqref{initials} is a triple 
$(\eta,\vr, \vu)\in W^I\times Q^I\times V^I$
that satisfies the following for all $\varphi \in C^\infty \left( \bar{I}\times \R^d \right)$ and for all $(\bfPsi, \psi) \in C^\infty_0 ([0,T] \times \R^d) \times C^\infty_0 (\Sigma)$ that 
\[ \int_0^T \frac{d}{dt} \intO{\vr \varphi }\dt - \int_0^T \intO{ \left( \vr \pdt \varphi +\vr \vu \cdot \Grad \varphi \right) } \dt  =0
\]
\begin{align*}
& \int_0^T \frac{d}{dt} \intO{\vr \vu \cdot \bfPsi }\dt - \int_0^T \intO{ \left( \vr \vu \cdot \pdt \bfPsi +\vr \vu \otimes \vu :  \Grad \bfPsi  \right) } \dt 
+ \int_0^T \intO{ \left( \bfS(\Grad \vu) : \Grad \bfPsi  -a \vr^\gamma \Div \bfPsi \right) }\dt 
\\ &
+ \int_0^T \left( \frac{d}{dt} \int_\Sigma \pdt \eta \psi \ds   - \int_\Sigma \Big( \pdt \eta \pdt \psi  +  K'(\eta) \psi \Big) \ds  \right) \dt
 = \int_0^T \intO{ \vr \bff \cdot \bfPsi }\dt + \int_0^T \int_\Sigma g \psi \ds \dt
\end{align*}
 with $ \bfPsi(\bfr,H+\eta)=\psi(\bfr )\bfe_d$ on $\Sigma$  and  $\bfPsi\equiv 0$ on $\Gamma_D$.  
Moreover, the solution satisfies the energy estimates
\begin{equation}\label{energy_estimates_continuous}
\begin{split}
& \sup_{t\in [0,T]}\bigg( \intO{ \left( \frac12 \vr |\vu|^2 + \Hc(\vr)   \right) }  +  
\intS {\big( z^2  + K(\eta) \big) } \bigg)
+
\int_0^T \intO{ \bfS(\Grad \vu) :  \Grad \vu }\dt 
\\& 
\leq
\intO{ \left( \frac12 \vr_0 |\vu_0|^2 + \Hc(\vr_0)   \right)  } +\intS{ \left(|\eta_0|^2 + |z_0|^2 +K(\eta(0))\right)} + \int_0^T \left( \norm{\bff}_{L^2(\Omega)} + \norm{g}_{L^2(\Sigma)} \right) \dt 
 \end{split}
\end{equation}
where $\Hc(\vr) = \frac{a }{\gamma-1}\vr^\gamma$ represents the internal energy of the fluid. 
\end{Definition}

\subsection*{Time discretization} 
We divide the time interval by $N_T$ subintervals and set  $\TS = T/ N_T$ as the size of the time step. For simplicity, we write $t^k=k\TS$ and $I^k=[t^k,t^{k+1})$ for all $k=0,1,\ldots,N_T$. 
Further, we denote $\vh^k$ as the approximation of $v$ at the time $t^k$.  Next, we start the time discretization with the piecewise constant in time approximations of the domain (mesh)
\[
\Oh(t) =\Oref \text{ for } t < \TS; \quad \Oh(t) = \Oh^k \mbox{ for all } t\in I^k, \; k=1,2,\ldots,N_T. 
\]
Note that the deformation of the domain is related to the ALE mapping, that is also approximated as a piecewise constant in time function
\[
\ALEh(t) =\ALEh^0 \text{ for } t < \TS; \quad \ALEh (t) = \ALEh^k \mbox{ for all } t\in I^k,\; k=1,2,\ldots,N_T. 
\] 
Further, we continue the time discretization  of all unknowns,  including the test functions,  by piecewise constant in time functions on the fixed reference domain $\Oref$
\begin{equation}\label{Tspace}
\hvh (t,\bfxref) =\hvh^0(\bfxref) \text{ for } t < \TS; \quad \hvh (t,\cdot) = \hvh^k(\bfxref)  \mbox{ for all } t\in[k\TS, (k+1)\TS),\quad k=1,2,\ldots,N_T, \quad \bfxref \in \Oref
\end{equation}
where $\hv\in\{ \hvr, \hvu, \widehat{p}, \heta, \widehat{z} ,\hvw, \hvarphi, \hbfPsi, \hpsi\}$. 
To recover the functions from the reference domain back to  current domain, we take 
\begin{equation}\label{two_domain}
\vh^k = \hvh^k \circ (\ALEhk)^{-1} \quad \mbox{and} \quad 
\vh = \hvh \circ \ALEh^{-1}
\end{equation}
for $v\in\{ \vr, \vu, p, \eta, z , \vw, \varphi, \bfPsi, \psi\}$. 
Finally, we define a projection operator 
\begin{equation}\label{Pit}
\Pit v = \sum_{k=0}^{N_T}   \int_{I^k} 
\frac{1_{I^k}(t)}{\TS} v \dt, \; \forall \; k \in \{0,1,\ldots, N_T\}, \mbox{ and }
1_{I^k}(t)=\begin{cases} 1 & \mbox{if } t \in I^k\\
0 & \mbox{otherwise }.
\end{cases}
\end{equation}

\subsection*{ALE mapping}
In consistent with \eqref{two_domain}, we define the deformation rate of the fluid domain at time $t^k$ $(k=1,2,\ldots, N_T)$ as  
\begin{equation}\label{mesh_velocity}
\hvwh^{k}= \frac{\ALEhk -\ALEhkm }{\TS} 
= \left( {\bf0}_{d-1}, \frac{\etah^{k} -\etah^{k-1}}{\TS }\frac{\xref_d}{H } \right)^T, \quad 
\vwh^k = \hvwh^k \circ (\ALEh^k)^{-1} = \left( {\bf0}_{d-1}, \frac{\etah^{k} -\etah^{k-1}}{\TS }\frac{x_d}{\etah^{k} +H } \right)^T,
\end{equation}
where ${\bf0}_{d-1}$ is $(d-1)$-dimensional zero vector.  
For convenience, we introduce $\bfX_i^j$ as the mapping from $\Oh(t^i)$ to $\Oh(t^j)$, i.e., 
\begin{equation}\label{mappingij}
\bfX_i^j: \quad \Oh(t^i) \mapsto \Oh(t^j), \qquad  \bfX_i^j(\bfx^i) = \ALEh^j \circ (\ALEh^i)^{-1}(\bfx^i) \quad \mbox{for all } \bfx^i \in \Oh(t^i). 
\end{equation}
Recalling the definition of the ALE mapping \eqref{ALE_mapping}, the Jacobian of $\bfX_i^j$ and its determinant read    
\begin{equation} \label{jacob}
\Jacob_i^j = \frac{\pd \bfX_i^j(\bfx^i)}{\pd \bfx^i}, \quad \mbox{and} \quad 
\mJ_i^j =\det \left(  \Jacob_i^j \right) = \frac{\eta^{j} +H }{\eta^i +H} , 
\end{equation}
respectively. 
From the above notations it is easy to check 
\begin{equation}\label{div_jacob}
\TS \;\Div \vwh^{k} 
=   1- \mJkm. 
\end{equation}
Further we observe, that if $\etah^k(\bfr)\in (\delta_0-H, H_{\max} -H )$ for all $k\in\{1,...,N_T\}$ and all $\bfr\in \Sigma$, then 
\begin{align*}
 0<   \frac{\delta_0}{H_{\max}}\leq \mJ_i^j \leq  \frac{H_{\max}}{\delta_0}, \quad i,j\in \{1,...,N_T\}. 
\end{align*}
In order to transfer between the current domain and the reference domain, 
we recall the chain-rule and properties of the Piola transformation from~\cite{Ciarlet_elas}
\begin{equation}\label{Piola}
\dx = \mJ \dxref, \quad 
\dSx =  \abs{\mJ \Jacob^{-T} \widehat{\bfn} } \dSxref, \quad 
\bfn = \frac{\mJ \Jacob^{-T} \widehat{\bfn} }{ \abs{\mJ \Jacob^{-T} \widehat{\bfn} }}, \quad 
\mJ \Div \bfq = \widehat{\Div}  \left( \mJ \Jacob^{-T}  \widehat{\bfq} \right), 
\quad 
\Jacob^{T} \Grad_{\bfx} r =  \widehat{\Grad}  \widehat{r} , 
\end{equation}
for a scalar function $r$ and a vector filed $\bfq$. 
Finally we denote for simplicity  
\begin{equation}\label{div_twomesh}
\reallywidehat{\Div \bfq} := \Div \bfq  \circ \ALE = \frac{1}{\mJ} \widehat{\Div}  \left(  \mJ \Jacob^{-T}  \widehat{\bfq} \right)
, \quad 
\reallywidehat{\Grad r} := \Grad r  \circ \ALE =  \Jacob^{-T} \widehat{ \Grad} \widehat{r} .
\end{equation}
\subsection*{Time difference operators}
First, let us introduce the discrete derivative in time for the displacement of the shell. As the shell function $\eta(t, {\bfr})$ is defined on the time independent reference domain, we apply the standard backward Euler discretization for the family of functions $r^k:\Sigma\to \mathbb{R}$, $k \in\{0,...,N_T\}$: 
\begin{equation}\label{def_DT1}
\delta_t r^{k}({\bfr}) = \frac{ r^{k}({\bfr})-r^{k-1}({\bfr}) }{\TS}, \; \delta_t^2 r^{k}({\bfr})=    \delta_t(\delta_t r^{k}({\bfr}))=\frac{\delta_t r^{k}({\bfr})- \delta_t r^{k-1}({\bfr})}{\TS}. 
\end{equation}

Next, for the fluid part, it is necessary to use the material time derivative in order to discretize our scheme properly.  Since the domain $\Oh$ is changing in time discretely 
we divide the material derivative in  the bulk part (inside the domain where the deformation of the elastic shell is of minor importance) and the boundary part, where we approximate it by $z^k$. We define in the following the discrete material derivative that reflects the change of the domain as
\begin{equation*}
D_t^\ALE \rh^k = \frac{ \rh^k - \rh^{k-1} \circ \bfXkm  }{\TS}, 
\end{equation*}
where $\bfXkm=\ALEhkm \circ (\ALEhk)^{-1}$ is the mapping from $\Oh(t^k)$ to $\Oh(t^{k-1})$, see \eqref{mappingij}. In the following we deviate the material derivative in a fluid part and a shell part by the following heuristics. 
Let $q$ be some quantity defined on the current domain and $\vu$ be the fluid velocity on the same domain. We deviate 
\[
\lim_{\tau\to 0}\pdt q (t+\tau, \bfx+\tau\vu(t,\bfx))
=\pdt q(t,\bfx)+ \vu \cdot \nabla q(t,\bfx)=\pdt q(t,\bfx) + \vw \cdot \nabla q(t,\bfx) + (\vu-\vw)\cdot \nabla q(t,\bfx) ,
\]
where $\vw$ denotes the velocity of the moving domain, and $(\vu-\vw)$ is the relative velocity of the fluid with respect to the moving domain. 
Note that the first two terms on the right hand side capture the material time-derivative.  Indeed $D_t^\ALE \rh^k$ is approximating that part. 
%
Further, we observe 
\begin{align*}
D_t^\ALE r^k & =\frac{r^k-r^{k-1}\circ  \bfXkm }{\TS} \pm \mJkm \frac{r^{k-1}\circ  \bfXkm }{\TS}
= \frac{r^k-r^{k-1}\circ  \bfXkm \mJkm }{\TS} - \Div \vwh^{k} \;  r^{k-1}\circ  \bfXkm ,
\end{align*}
which, as can be seen below turns out to be the suitable deviation in order to get {\sl a-priori} estimates. In addition, 
the above calculation motivates the use of the following non-standard time difference operator approximating the Eulerian time derivative
\begin{equation}\label{def_DT3}
D_t r^k = \frac{ r^k - r^{k-1} \circ \bfXkm \mJkm   }{\TS}.
\end{equation} 
We summarize the heuristics by the following discrete version of the Reynolds transport theorem.
\begin{Lemma}[Discrete Reynolds transport]\label{lem_DRT} 
For the time difference operator defined in \eqref{def_DT1} and \eqref{def_DT3}, we have the following discrete analogy of the Reynolds transport theorem. 
\begin{equation*}
\delta_t \intOhk{r^k}  = \intOhk{ D_t r^k} = \intOhk{ \left( D_t^\ALE r^k + \Div \vwh^{k} \;  r^{k-1}\circ  \bfXkm \right)}.
\end{equation*}
\end{Lemma}
\begin{proof} From the definition of time difference operators and the determinant of the Jacobian given in \eqref{jacob}, we easily get 
\begin{equation*} 
\begin{aligned}
\delta_t \intOhk{r^k} &= \frac{1}{\TS } \left( \intOhk{r^k} - \intOhkm{r^{k-1}}\right)  
 =
 \frac{1 }{\TS} \left( \intOhk{r^k} - \int_{\Oh^{k}} r^{k-1} \circ \bfXkm   \det\left(\frac{\pd \bfx^{k-1}}{\pd \bfx^{k}}\right) {\rm d} \bfx^{k}   \right)
\\&= \intOhk{ \frac{ r^k - r^{k-1} \circ \bfXkm \mJkm   }{\TS} }  
 = \intOhk{ D_t r^k} = \intOhk{ \left( D_t^\ALE r^k + \Div \vwh^{k} \;  r^{k-1}\circ  \bfXkm \right)}.
\end{aligned}
\end{equation*}
\end{proof}
Note that the discrete Reynolds transport holds also for any $C\subset \Oh$. Thus we obtain the geometric conservation law by taking $r=1$ 
\begin{equation}\label{GC}
\frac{1}{\TS} \left( |C^k|  - |C^{k-1}|  \right)= \int_{C^k} \Div \vwh^{k} \dx = \int_{\pd C^k}  \vwh^{k} \cdot \bfn \ds .
\end{equation}

\section{Semi-discrete scheme}\label{sec:SS}
This section introduces the necessary tools and observations with respect to the time discretization. Due to the overwhelming technical notation in the fully discrete case we decided to include this semi-discrete section. We wish to emphasize that the main objective of this section is to explain the methodology.
Hence we will assume within this section that the discrete in time but continuous solutions in space introduced below exist and are bounded in spaces in such a way that the discrete energy is well defined. We assume further (for this section) that the all needed test functions are admissible without further justification.

\subsection{The scheme}
The analysis is best understood when considering the scheme in the current domain (which is changing in each time step). However, for applications the scheme defined on the reference domain seems more handable to be implemented (see also the next section). Hence we first introduce the semi-discrete ALE scheme on the current domain followed by its equivalent formulation on the fixed reference geometry.
For the spaces of existence we simply assume, that $\Wh(\Sigma)\subset W^{2,2}_0(\Sigma)$, $\Qh(\Oh)\subset L^\gamma(\Oh)$ and $\Vh(\Oh)\subset W^{1,s}(\Oh)$ for all $s<2$.


\begin{Definition}[Semi-discrete scheme on the current domain]\label{def_SS} 
For all $k \in \{ 1,\ldots, N_T \}$  we seek the solution $(\etah^{k} ,\vrh^{k}, \vuh^{k} ) \in (\Wh(\Sigma),\Qh(\Oh^k),\Vh(\Oh^k))$ such that for all (admissible) $(\psih,\varphih, \bfPsih) \in ( \Wh(\Sigma), \Qh(\Oh^k), \Vh(\Oh^k))$ with $\bfPsih|_{\Gamma_S}\circ\ALEh = \psih \bfe_d$ the following hold:
\begin{subequations}\label{SS}
\begin{equation} \label{SS_D}
\intOk{ D_t \vrh^{k} \varphih}  + \intOk{\Div ( \vrh^{k} \vvh^{k} ) \varphih}  
= 0;
\end{equation}
\begin{multline}\label{SS_M}
\intOk{ D_t \left( \vrh^{k}  \vuh^{k} \right) \cdot \bfPsih +\Div (\vrh^{k} \vuh^{k}\otimes \vvh^{k} ) \cdot \bfPsih }
+  \intOk{ \bfS(\Grad \vuh^{k}) : \Grad \bfPsih } 
- \intOk{ p(\vrh^{k}) \Div \bfPsih }\\
+ \intS{ \delta_t \zh^{k}  \psih }  + \alpha \intS{  \Delta \etah^{k} \Delta \psih } 
 + \beta \intS{  \Grad \etah^{k} \cdot \Grad \psih } 
=
  \intOk{\vrh^k \bfh^{k}\cdot \bfPsih}  + \intS{\gh^{k} \psih} ;
\end{multline}
where 
$ \zh^k = \delta_t \etah^k, \quad  \vvh^{k} = \vuh^{k}-\vwh^{k}, \quad 
\gh^k:=\frac{1}{\eps}\int_{I^k} g\, \dt \text{ and }\bfh^k:=\frac{1}{\eps}\int_{I^k} \bff\, \dt.$  
The scheme is supplemented with the initial data  and  boundary conditions 
\begin{equation*}
 \vrh^0 = {\vr_0}, \quad \vuh^0 = {\vu_0}, \quad \etah^0=0, \quad \zh^{0}= 0, \quad 
 \vuh^k|_{\pd \Oh} = \vwh^k|_{\pd \Oh}.
\end{equation*}  
\end{subequations}
\end{Definition}
Concerning the solvability of the scheme, we will discuss later for the fully discrete scheme, see Theorem~\ref{thm:exist}.

\subsection{Stability}\label{sec:stability}
In this section, we aim to show some stability properties for the scheme \eqref{SS}. 
First, we remark that the scheme \eqref{SS} preserves the total mass. 
Indeed, by setting $\varphih \equiv 1$ in \eqref{SS_D} and applying the discrete Reynolds transport Lemma~\ref{lem_DRT}, we derive $\delta_t \left( \intOk{ \vrh^k} \right)=\intOk{D_t \vrh^k}=0$ for all $k=1,\ldots,N_T$, which implies 
\begin{equation}\label{MC}
 \intOk{\vrh^k} = \intOhkm{\vrh^{k-1}} = \cdots = \intOref{\vrh^0} =:M_0, \; \mbox{ for all } k=1,\ldots,N_T.
\end{equation}
Next, we show the renormalization of the discrete density problem. 
\begin{Lemma}[Renormalized continuity equation]\label{lem_r0}  
Let $(\vrh,\vuh)\in  \Qh \times \Vh$ satisfy the discrete continuity equation \eqref{SS_D} with the boundary condition $\vuh|_{\pd \Oh} = \vwh|_{\pd \Oh}$. 
Then for any $B \in C^1(R)$ it holds
\begin{equation*}
\frac{ 1}{\TS} \left( \intOk{  B(\vrh^k)}  - \intOkm{ B(\vrh^{k-1}) }  \right)
+ \intOk{  \left( \vrh^k B'(\vrh^k)- B(\vrh^k) \right) \Div \vuh^{k} } + D_0 =0.
\end{equation*}
where 
$D_0=  \frac{1}{\TS}\intOk{  \mJkm  \left( B(\vrh^{k-1}\circ \bfXkm)  - B(\vrh^k) - B'(\vrh^{k}) \big(  \vrh^{k-1}\circ \bfXkm-\vrh^k  \big)  \right)  }. 
$
Moreover, $D_0\geq 0$ if $B$ is convex. 
\end{Lemma}
\begin{proof}
We set $\varphih =B'(\vrh^k)$ in the discrete density equation \eqref{SS_D} and obtain
\begin{equation*}
\intOk{ D_t \vrh^{k} B'(\vrh^{k}) }  + \intOk{\Div \left( \vrh^{k} \vvh^{k}  \right) B'(\vrh^{k}) } 
 = 0 .
\end{equation*}
First, by applying the Taylor expansion, we know there exist $\xi \in \co{\vrh^{k-1}\circ \bfXkm}{\vrh^{k}}$ that 
\begin{equation}\label{N1}
\begin{split}
& \intOk{ D_t \vrh^{k} B'(\vrh^{k})} =\intOk{ \frac{\vrh^{k} -\vrh^{k-1}\circ \bfXkm \mJkm }{\TS} B'(\vrh^{k}) }  
\\&= 
\frac{1}{\TS}\intOk{  \bigg(  B(\vrh^k) - B(\vrh^{k-1}\circ \bfXkm) \mJkm  +   
\left( \vrh^k B'(\vrh^k)- B(\vrh^k) \right)  
+ \mJkm \Big( B(\vrh^{k-1}\circ \bfXkm)  - \vrh^{k-1}\circ \bfXkm  B'(\vrh^{k}) \Big)  \bigg)} 
\\&= 
\frac{1}{\TS}\intOk{  \left( B(\vrh^k) - B(\vrh^{k-1}\circ \bfXkm) \mJkm  \right) } 
+ \frac{1}{\TS}\intOk{  \left( \vrh^k B'(\vrh^k)- B(\vrh^k) \right)(1-\mJkm) }
\\ & \qquad \qquad    + \frac{1}{\TS}\intOk{  \mJkm  \left( B(\vrh^{k-1}\circ \bfXkm)  - B(\vrh^k) - B'(\vrh^{k}) \big(  \vrh^{k-1}\circ \bfXkm-\vrh^k  \big)  \right)  } 
\\& = 
\frac{ 1}{\TS} \left( \intOk{  B(\vrh^k)}  - \intOkm{ B(\vrh^{k-1}) }  \right)
+  \intOk{ \left( \vrh^k B'(\vrh^k)- B(\vrh^k) \right) \Div \vwh^k} +D_0
\end{split}
\end{equation}
where we have used relation between the Jacobian and the deformation rate of the domain given in \eqref{div_jacob}.  
Next, by applying integration by parts twice, we reformulate the convective term as 
\begin{align*}
& \intOk{  \Div (\vrh^{k} \vvh^{k}) B'(\vrh^{k}) } 
 =  -\intOk{ \vrh^{k}\vvh^{k} \cdot \Grad B'(\vrh^{k})  }
=  -\intOk{ \vvh^{k} \cdot \Grad \left(\vrh^{k} B'(\vrh^{k}) - B(\vrh^{k} )\right)  }
\\ &
=  \intOk{ \Div \vvh^{k} \left(\vrh^{k} B'(\vrh^{k}) - B(\vrh^{k} )\right)  } , 
\end{align*}
where we have used the equality $ \Grad \Big(\vr B'(\vr) -B \Big) =   \vr \Grad B'(\vr)$.
Consequently, summing up the above equations and seeing $\vvh = \vuh -\vwh$, we complete the proof of the identity. Now, if $B\in C^2(\mathbb{R})$ is convex, we use the fact that by Taylor expansion there exists a $\xi(\bfx) \in \co{\vrh^{k-1}\circ \bfXkm(\bfx)}{\vrh^{k}(\bfx)}$ for all $\bfx\in \Oh^k$, such that 
$
D_0=\intOk{  \TS \mJkm  \frac{B''(\xi)}{2} \left| D_t^\ALE \vrh^{k} \right|^2 }\geq 0.
$
The general convex case follows by approximation. 
\end{proof}

With the renormalized continuity equation in hand, we are ready to show non-negativity of the discrete density and the internal energy balance. 
\begin{Lemma}[Non-negativity of density]\label{lem_non-negative}  
Any solution to the scheme~\eqref{SS} preserves non-negativity of the density. It mean $\vrh^{k} \geq 0$ for all $k=1,\ldots, N_T$ provided $\vrh^0 \geq 0$.
\end{Lemma}
\begin{proof}
By setting  $B(\vr) = \max\{0,-\vr\} \geq 0$ in Lemma~\ref{lem_r0} and 
assuming $\vrh^{k-1} \geq 0$, we observe   
\[  B(\vrh^{k-1}) =0,\; \quad \vrh^k B'(\vrh^k)- B(\vrh^k) = 0, \;
\left( B(\vrh^{k-1}\circ \bfXkm)  - B(\vrh^k) - B'(\vrh^{k}) \big(  \vrh^{k-1}\circ \bfXkm-\vrh^k  \big)  \right) \geq 0. \]
Thus we find $\intOhk{B(\vrh^k)} =0$.  
Realizing $B$ is a non-negative function we know that $B(\vrh^k)=0$ holds for all $\bfx \in \Oh^k$ which implies $\vrh^k \geq 0$. As $\vrh^0 \geq 0$ we finish the proof by mathematical induction. 
\end{proof}
Further discussion on the strictly positivity of the discrete density will be shown for the fully discrete scheme in Lemma~\ref{lem_positive} in the next section. 

Next, by setting $B=\Hc(\vr)$ in Lemma \ref{lem_r0} and realizing $p=\vr \Hc'(\vr) -\Hc$, we derive the following relation on the internal energy.
\begin{Corollary}[Internal energy balance]
Let $(\vrh,\vuh, \etah)\in  \Qh \times \Vh \times \Wh$ be the solution of the discrete problem \eqref{SS}.  Then there exists $\xi(\bfx) \in \co{\vrh^{k-1}\circ \bfXkm(\bfx)}{\vrh^{k}(\bfx)}$  such that
\begin{equation}\label{r1H}
\begin{aligned}
& \frac{ 1}{\TS} \left( \intOk{  \Hc(\vrh^k)}  - \intOkm{ \Hc(\vrh^{k-1}) }  \right)
+ \intOk{  p(\vrh^{k}) \Div \vuh^{k} } + D_1 =0,\quad \mbox{with} \\
& D_1  = 
 \frac{1}{\TS}\intOk{  \mJkm  \left( \Hc(\vrh^{k-1}\circ \bfXkm)  - \Hc(\vrh^k) - \Hc'(\vrh^{k}) \big(  \vrh^{k-1}\circ \bfXkm-\vrh^k  \big)  \right)  } 
  =
\intOk{  \TS \mJkm   \frac{\Hc''(\xi)}{2} \left| D_t^\ALE \vrh^{k} \right|^2 }  \geq 0.
\end{aligned}
\end{equation} 

\end{Corollary}
Finally, we proceed to show the energy stability of the scheme \eqref{SS}. 
\begin{Theorem}[Energy estimates]\label{thm_s1}  
Let $\left(\vrh^{k}, \vuh^{k}, \etah^{k}\right)_{k=1}^{N_T}$ be a family of numerical solutions obtained by scheme~\eqref{SS}. Then the following energy estimate holds 
\begin{multline}\label{energy_stability} 
 \delta_t \left(  \intOk{  E_f^{k} }   +  \intS{   E_s^{k}}  \right) 
+ \frac{\TS}2  \intS{ \left( |\delta_t \zh^{k} |^2+ \alpha \left|\Delta \zh^{k} \right|^2 + \beta \left| \Grad \zh^{k}\right|^2  \right) } 
  +  \intOk{ \left(2\mu|\bfD(\vuh^{k}) |^2 +   \lambda |\Div \vuh^{k} |^2 \right)    }
\\ + \intOk{  \TS \mJkm   \frac{\Hc''(\xi)}{2} \left| D_t^\ALE \vrh^{k} \right|^2 }  
 + \intOk{\frac{\TS}{2} \vrh^{k-1} \circ \bfXkm
\left| D_t^\ALE  \vuh^{k} \right|^2  }
  =  \intOk{\vrh^k \bfh^{k}\cdot \vuh^{k}} + \intS{\gh^{k} \zh^{k}}.
\end{multline}
where 
$\quad E_f^{k}=  \frac{1}{2} \vrh^{k} \left| \vuh^{k}\right|^2 + \Hc(\vrh^{k}) , 
\quad E_s^{k} =\frac12 (|\zh^{k}|^2 +  \alpha |\Delta \etah^{k}|^2 +  \beta |\Grad \etah^{k}|^2)$.
\end{Theorem}
\begin{proof}
Setting $\varphih = - \frac{\left| \vuh^{k}\right|^2}{2}$ in \eqref{SS_D}, and $(\bfPsih, \psih) =( \vuh^{k},\zh^k)$  in \eqref{SS_M},  we have 
$\sum_{i=1}^2 I_i=0, \quad \text{and }  \sum_{i=3}^{9} I_i=0$, respectively, where 
\begin{equation*}
\begin{aligned}
& I_1 = -\intOk{ D_t \vrh^{k} \frac{\left| \vuh^{k}\right|^2}{2} } , \quad 
I_2= -\intOk{\Div \left( \vrh^{k} \vvh^{k}  \right) \frac{\left| \vuh^{k}\right|^2}{2} } , 
\quad
I_3 = \intOk{ D_t \left( \vrh^{k} \vuh^{k} \right) \cdot \vuh^{k} },  \quad 
\\& I_4 = \intOk{\Div \left(\vrh^{k} \vuh^{k} \vvh^{k} \right) \cdot \vuh^{k} },\quad
I_5 =- \intOk{ p(\vrh^{k}) \Div \vuh^{k} } , 
\quad
I_6 = 2\mu \intOk{ |\bfD( \vuh^{k})|^2} + \lambda \intOk{ |\Div \vuh^{k}|^2 },  \quad
\\&I_7=  \intOk{\vrh^k \bfh^{k}\cdot \vuh^{k}} + \intS{\gh^{k}  \zh^k} , 
\quad
I_{8}= \intS{ \frac{\zh^k- \zh^{k-1}}{\TS} \zh^k }, \quad
I_{9}= \alpha \intS{  \Delta \etah^{k} \Delta \zh^{k}  }   + \beta \intS{  \Grad \etah^{k} \cdot \Grad \zh^{k} } 
.
\end{aligned}
\end{equation*}
Now we proceed with the summation of all the $I_i$ terms for  $i=1,\ldots,12$. 

{\bf Term $I_1+I_3+I_{8}$.} Applying the equality $a(a-b) =\frac{a^2-b^2}{2}+\frac{(a-b)^2}{2}$ we get 
\begin{equation*}
\begin{aligned}
I_1+I_3 +I_{8} =& \frac{1}{\TS} \left( \intOk{ \frac12\vrh^{k} \left| \vuh^{k}\right|^2}   -  \intOkm{ \frac12\vrh^{k-1} \left| \vuh^{k-1}\right|^2}   \right)
+ \intS{ \delta_t \left( \frac{|\zh^{k} |^2}{2}  \right)  } 
\\ & 
+ \frac{\TS}{2}  \intOk{ \vrh^{k-1}\circ \bfXkm \left| D_t^\ALE \vuh^{k}  \right|^2  } 
+ \frac{\TS}{2} \intS{|\delta_t \zh^{k} |^2  } .
\end{aligned}
\end{equation*}

{\bf Term $I_2+I_4$.} For the convective terms, we have
\begin{equation*}
 I_2+I_4  = \intOk{ \left( -\Div ( \vrh^{k} \vvh^{k} ) \frac{\left| \vuh^{k}\right|^2}{2} +  \Div \left(\vrh^{k}\vuh ^{k} \otimes \vvh^{k} \right) \cdot \vuh^{k} \right) }
= 
 \intOk{  \left( \vrh^{k} \vvh^{k} \cdot \Grad \frac{\left| \vuh^{k}\right|^2}{2} - \vrh^{k}\vuh ^{k} \otimes \vvh^{k}  :\Grad \vuh^{k} \right) } = 0. 
\end{equation*}

{\bf Pressure term $I_{5} $. }
Recalling the discrete internal energy equation \eqref{r1H}, we can rewrite the pressure term as 
\begin{equation*}
I_5= -\intOk{  p(\vrh^{k}) \Div \vuh^{k} } =
\frac{ 1}{\TS} \left( \intOk{  \Hc(\vrh^k)}  - \intOkm{ \Hc(\vrh^{k-1}) }  \right)
+ \intOk{  \TS \mJkm   \frac{\Hc''(\xi)}{2} \left| D_t^\ALE \vrh^{k} \right|^2 }  .
\end{equation*}

{\bf Term $I_6+I_7$.} These terms don't change. 

{\bf Term $I_{9}$. } Applying again $a(a-b) =\frac{a^2-b^2}{2}+\frac{(a-b)^2}{2}$, we deduce 
\begin{align*}
I_{9} & =  \intS{ \frac12  \delta_t \left( \alpha |\Delta \etah^{k}|^2 + \beta |\Grad \etah^{k}|^2 \right)}
+ \intS{ \left( \frac{\TS\alpha}2 \left|\delta_t(\Delta \etah^{k})\right|^2 +\frac{\TS\beta}2 \left|\delta_t(\Grad \etah^{k})\right|^2 \right) }\\
& = \intS{ \frac12  \delta_t \left( \alpha |\Delta \etah^{k}|^2 + \beta |\Grad \etah^{k}|^2 \right)}
+ \frac{\TS}2 \intS{ \left( \alpha \left| \Delta \zh^{k}\right|^2 + \beta \left| \Grad \zh^{k}\right|^2 \right) }
\end{align*}

Collecting all the above terms, we finish the proof, i.e.,
\begin{equation*} 
\begin{aligned}
& \frac{1}{\TS} \left(  \intOk{  E_f^{k} }   - \intOkm{ E_f^{k-1} }  \right) +  \intS{  \delta_t E_s^k }
+ \frac{\TS}2  \intS{ \left( |\delta_t \zh^{k} |^2+ \alpha \left|\Delta \zh^{k} \right|^2 + \beta \left|\Grad \zh^{k} \right|^2  \right) }
\\ & 
+ \intOk{\left( 2\mu |\bfD( \vuh^{k})|^2  +  \lambda  |\Div \vuh^{k}|^2 \right)}
+ \intOk{  \TS \mJkm   \frac{\Hc''(\xi)}{2} \left| D_t^\ALE \vrh^{k} \right|^2 }  
 + \intOk{\frac{\TS}{2} \vrh^{k-1} \circ \bfXkm
\left| D_t^\ALE  \vuh^{k} \right|^2  }
\\ &  =  \intOk{\vrh^k \bfh^{k}\cdot \vuh^{k}} + \intS{\gh^{k} \zh^{k}}.
\end{aligned}
\end{equation*}
\end{proof}

\subsection{Some a-priori estimates}
Let us recall that all unknowns including the domain and the test functions are piecewise constant in time, see \eqref{Tspace}. 
We define $\overline{\etah}(t,\bfr)$ as the affine linear interpolant of $\etah$ meaning that
$\overline{\etah}\in C^0(0,T;\Sigma)$, such that $\overline{\etah}(t^k,\bfr)=\etah^k(\bfr)$ and $\partial_t\overline{\etah}(t,\bfr)=\zh^k(\bfr)$ for $t\in I^k= [t^k,t^{k+1})$.

With a little abuse of notation we use $[0,T]\times \Oh(\cdot)=\bigcup_{k=1}^{N_T}(t^{k-1},t^k]\times \Oh(t^k)$. Accordingly we define for $s\in [0,\infty)$, $q\in [1,\infty]$ and $p\in [1,\infty)$
\begin{equation*}
\norm{f_\tau}_{L^p(0,T;W^{s,q}(\Oh(\cdot))}:=\bigg(\sum_{l=0}^{N_T}\tau\norm{f_\tau^l}_{W^{s,q}(\Oh(t_l))}^p\bigg)^\frac{1}{p}, 
\quad 
\norm{f_\tau}_{L^\infty(0,T;W^{s,q}(\Oh(\cdot))}:=\max_{k}\norm{f_\tau^l}_{W^{s,q}(\Oh(t_l))}.
\end{equation*}
Note that the expressions above bound the respective norms for both the piecewise constant functions in time as well as the piecewise affine linear functions in time.

Then the energy estimate~Theorem~\ref{thm_s1} implies the following a-priori estimates (for the piecewise constant in time functions $\etah,\vrh,\vuh$) that are uniform in $\tau$:
\begin{equation*}
\begin{split}
\norm{ \vrh }_{ L^{\infty}(0,T; L^\gamma(\Oh(\cdot)))} \leq c, \quad \norm{ \vrh |\vuh|^2}_{ L^{\infty}(0,T; L^1(\Oh(\cdot)))} \leq c, \\
\norm{\vuh }_{  L^2(0,T; L^6(\Oh(\cdot)))}\leq c, \quad \norm{\Gradh \vuh }_{  L^2(0,T; L^2(\Oh(\cdot)))}\leq c,  \quad \norm{\Divh \vuh }_{  L^2(0,T; L^2(\Oh(\cdot)))}\leq c, \\
\norm{\zh }_{  L^\infty(0,T; L^2(\Sigma))}\leq c, \quad \alpha \norm{\Lap  \etah }_{  L^\infty(0,T; L^2(\Sigma))} \leq c, \quad \beta \norm{\Gradh  \etah }_{  L^\infty(0,T; L^2(\Sigma))}\leq c, 
 \end{split}
\end{equation*}
 where $c$ depends on the external force $\bff$ and $g$ as well as the initial data.
 Furthermore,  for all $1\leq \beta<\gamma$ such that  $\frac{\beta}{\gamma}+\frac{\beta}{a}=1$ for some $a\in(1,\infty)$ we find
\begin{equation}    \label{eq:aprir}
    \int_{\Ohk}(\abs{\vrh^k}(\abs{\vuh^k}+1))^\beta\, dx=\int_{\Oh^k} \abs{\vrh^k}^{\beta}(\abs{\vuh^k}+1)^\beta\, dx
\leq \norm{\vrh^k}_{L^\gamma(\Oh^k)}^{\beta}\norm{\abs{\vuh^k}+1}_{L^a(\Oh^k)}^{\beta}.
\end{equation}
Please observe, that in case $d=2$ for every $\gamma>1$ one finds an $a$ such that the right hand side will be bounded. In case $d=3$ we are restricted to $\gamma>\frac{6}{5}$.
Indeed, in this case we find by Jensen's inequality that for $\beta\in (1,2]$
\begin{equation}
    \label{eq:aprir2}
    \begin{aligned}
    \sum_k\tau^\frac{\beta}{2}\int_{\Ohk}(\abs{\vrh^k}(\abs{\vuh^k}+1))^\beta\, dx
&\leq \norm{\vrh}_{L^\infty(0,T;L^\gamma(\Oh)}^{\beta}\norm{\abs{\vuh^k}+1}_{L^2(0,T;L^a(\Oh^k))}^{\beta}.
\end{aligned}
\end{equation}

In order to prove the consistency of the above scheme we need some additional a-priori estimates. 
\begin{Lemma}
\label{lem:apri}
For all $s\in [0,\frac12)$ and all $q\in [1,4)$ there is a constant independent of $\tau$ such that
\begin{align*}
& \max_{k}\norm{\delta_t\etah^k}_{L^2(\Sigma)}+\sum_{l=1}^{N_T} \tau \Big(\norm{\delta_t\etah^l}_{W^{s,2}(\Sigma)}^2+\norm{\delta_t\etah^l}_{L^{q}(\Sigma)}^2\Big)\leq C,
\\& \max_k\norm{\vwh^k}_{L^2(\Oh(t^k))}+\sum_{l=1}^{N_T} \tau \Big(\norm{\vwh^l}_{W^{s,2}(\Oh(t^l))}^2+\norm{\vwh^l}_{L^q(\Oh(t^l))}^2\Big)\leq C.
\end{align*}
The constant $C$ depends on the initial values and the bounds of the energy estimates alone.
Moreover, for all $\theta\in [0,\frac13)$ there exists a constant $C$ depending on the energy estimates and $\theta$, such that
\begin{align}
\label{eq:cont}
    \max_{k}\norm{\etah^k(x)-\etah^{k-1}(x)}_{L^\infty(\Sigma)}\leq C\tau^\theta.
\end{align}
\end{Lemma}
\begin{proof}
The energy estimate (Theorem~\ref{thm_s1}) implies that $\norm{z_\tau^k}_{L^2(\Sigma)}$ is uniformly bounded, which implies the same bound for $\partial_t\etah^k$ by the definition of $\zh$. Moreover, since $\zh^k$ is the trace of $\vuh^k$ which is in $W^{1,2}(\Oh^k)$, we find by the trace-theorem 
(see the related estimate in~\cite[Corollary 2.9]{Daniel}) that
\begin{align*}
\sum_{l=1}^{N_T} \tau \norm{\delta_t\etah^l}_{W^{s,2}(\Sigma)}^2&=
\sum_{l=1}^{N_T} \tau \norm{z_\tau^k}_{W^{s,2}(\Sigma)}^2\leq c\sum_{l=1}^{N_T} \tau \norm{\hvuh^l}_{W^{s+\frac{1}{2},2}(\Ohref)}^2
\leq c\sum_{l=1}^{N_T} \tau \norm{\vuh^l}_{W^{1,2}(\Oh(t^l))}^2\norm{\etah^l}_{W^{2,2}(\Sigma)}^2
\end{align*}
which can be bounded by the energy as well. Due to the fact that for any $q\in [1,4)$ there is an $s<2$ such that $W^{s,2}\hookrightarrow L^q$ the first inequality is completed. The second inequality follows by the very definition of $\vwh$. We extend $\etah^k, \etah^{k-1}$ by zero to $\mathbb{R}^2$ and take $r<\tau$. 
We use the notation of $\dashint_{B_r(x)}\etah^k(y)\,dy=\frac{1}{r^2\pi}\int_{B_r(x)}\etah^k(y)\,dy$ for the mean value integral. Then by Sobolev embedding, we find that $\etah^k\in C^\alpha(\Sigma)$, for all $\alpha<1$, and hence that
\begin{align*}
\abs{\etah^k(\bfr)-\etah^{k-1}(\bfr)}&\leq \Big|\etah^k(\bfr)-\dashint_{B_r(\bfr)}\etah^k(\bfy)\,d\bfy\Big|+\Big|\dashint_{B_r(\bfr)}\etah^k(\bfy)-\etah^{k-1}(\bfy)\, d\bfy\Big|+ \Big|\etah^{k-1}(\bfr)-\dashint_{B_r(\bfr)}\etah^{k-1}(\bfy)\,d\bfy\Big|
\\
&\leq Cr^{\alpha} +\Big|\dashint_{B_r(\bfr)}\etah^k(\bfy)-\etah^{k-1}(\bfy)\,d\bfy\Big|
\leq  Cr^{\alpha} + \tau \dashint_{B_r(\bfr)}\abs{\zh^k(\bfy)}\, d\bfy \leq Cr^{\alpha} + C\frac{\tau}{r^2}. 
\end{align*} 
Now the result follows by choosing $r=\tau^\frac{1}{\alpha+2}$.
\end{proof}
The regularity can be used to guarantee a minimal existence interval in time in which the shell is not touching the bottom of the fluid domain. At first we have the following observation which is a direct consequence of \eqref{eq:cont} above. 
\begin{Corollary}[Inductive prolongation principle]\label{cor:maxtime}
Let $\tau^\theta\leq \frac{\delta_0}{C}$ and $\delta_1\geq 2\delta_0$. Then, if for some $k\in \{1,...,N_T\}$ we find that
$\inf_\sigma \etah^k(\bfr)\geq \delta_1-H$, the $\etah^{k+1}$ satisfies
$\inf_\sigma \etah^k(\bfr)\geq \delta_1-\delta_0-H$.
\end{Corollary}

Moreover,  \eqref{eq:cont} implies the following lemma:
\begin{Lemma}
\label{lem:Tsd}
For every $\delta_0\in (0,H/2)$ there exists a $T_0$ just depending on the bounds of the energy inequality and $H$, such that 
$ \inf_{[0,T_0]\time \Sigma}\eta(t,\bfr)\geq \delta_0-H$.
\end{Lemma}
\begin{proof}
The result essentially follows from \eqref{eq:cont} from which we import the constants $C$ and $\theta$. Let $(T_0+\tau)^\theta\leq \frac{H-\delta_0}{C}$. Then we choose $N$ such that $(N-1)\tau<T_0\leq N\tau$, then for $k\in \{1,...,N\}$ we find by the fact that $\etah^0\equiv 0$, by \eqref{eq:cont} and by Jensen's inequality (for the concave functions using $\theta\in (0,1]$) that
\begin{equation*}
\etah^k(\bfr)=\etah^k(\bfr)-\eta_\tau^0(\bfr)\geq -\norm{\etah^k-\eta_\tau^0}_L^\infty(\Sigma)
\geq -C\sum_{i=0}^N\tau^\theta\geq -(T_0+\tau)^\theta C\sum_{i=0}^N\Big(\frac{\tau}{N\tau}\Big)^\theta\geq \delta_0-H, \; \forall \; \bfr\in \Sigma. 
\end{equation*}
\end{proof}

\subsection{Consistency}
\label{ssec:consistency}
In this subsection, we aim to show the consistency of the scheme, meaning the if the numerical solution converges, then it satisfies the weak formulation \eqref{def_ws} in the limit of $\TS \to 0$.

Usually, for that one takes a fixed test function and shows that the error produced by the discretization vanishes in the limit. Due to the fact that the {\em domain of the test function is a part of the solution} we have to approximate the test function space as well. We will do this in the following.
 We recall that $\overline{\eta}_\tau:[0,T]\times \Sigma \to [\delta_0-H,\infty)$ is defined as the affine linear function in time which satisfies $\overline{\eta}_\tau(k \tau)=\etah^k$ for all $k$. 

Now, the a-priori estimates imply the following lemma:
\begin{Lemma}
For any $\alpha \in [0,\frac13)$ and any of the above approximation sequences there exists a sub-sequence, $\{\overline{\eta}_{\tau_j}\}_{j\in \mathbb{N}}\in C^\alpha([0,T]\times \Sigma)$ and a $\eta\in \in C^\alpha([0,T]\times \Sigma)$, such that  
$
\overline{\eta}_{\tau_j}\to \eta\text{ with }j\to \infty
$ 
uniformly in $C^\alpha([0,T]\times \Sigma)$.
\end{Lemma}
\begin{proof}
Sobolev embedding implies that $\overline{\eta}_\tau(t)$ is bounded in $C^{\alpha}(\Sigma)$ for all $\alpha\in (0,1)$ uniformly in $t,\tau$. Combining that with \eqref{eq:cont} implies that $\overline{\eta}_\tau$ is bounded in $C^{\alpha}([0,T]\times \Sigma)$ for all $\alpha\in (0,\frac13)$ uniformly in $\tau$. Hence the theorem of Arzela Ascoli implies the result.
\end{proof}
In the following we may assume that $\overline{\eta}_{\tau}\to \eta$ uniformly (omitting the index $j$).
Now, we take a test function on the limit domain:
\begin{align}
\label{eq:test function}
\begin{aligned}
&(\psi,\Psi)\in C^\infty([0,T],C^\infty_0(\Sigma))\times C^\infty([0,T]\times \Oh(t);\mathbb{R}^d)\text{ such that }\Psi(t)|_{\Gamma_D}=0,
\\
&\Psi(t,\cdot,\eta(t,\cdot)+H)=\psi (t,\cdot )\bfe_d\text{ on }\Sigma \text{ and } \Psi(t)\equiv0\equiv\psi(t)\text{ for all }t\geq  T.
\end{aligned}
\end{align}
In order to satisfy the coupling condition we introduce an approximating sequence introducing the new approximation parameter $\epsilon\in(0,1)$ 
\begin{align}
\label{eq:epsilonlayer}
\begin{aligned}
&\Psie:C^\infty([0,T]\times \mathbb{R}^3;\mathbb{R}^d)\text{ such that }
\Psie(t,\bfr,x_d)=\psi (t,\bfr )\bfe_d\text{ for all }\bfr\in \Sigma
\\
&\text{ and }x_d\in (\eta(t,\cdot)+H-\epsilon,\eta(t,\cdot)+H+\epsilon).
\end{aligned}
\end{align}
Such an approximation can be made precise by taking a cut-off function. We take $\phi_\epsilon\in C^\infty[0,\infty)$, such that $\phi^{(k)}_\epsilon(0)=0$ for all $k\in \mathbb{N}$ and $\phi(x)\equiv 1$ for all $x\in [\epsilon,\infty)$ and $0\leq \phi'_\epsilon\leq \frac{2}{\epsilon}$. Moreover, we take for a function $b:C^\alpha([0,T])$ the notation $(b)_\epsilon$ as the standard convolution function. Recall that since $\eta\in C^\alpha$ uniformly we find  in particular $(\eta)_\epsilon\leq \eta + \epsilon^\alpha$.
Then (for a fixed $t$) we define for $\epsilon<\min\{\frac{1}{3}\delta_0,1\}$
\begin{align*}
\Psie(t,\bfr, x^d):= (1-\phi_\epsilon(H+(\eta)_\epsilon(t)-2\epsilon^\alpha+x_d))\Psi(t,\bfr, x^d)+\phi_\epsilon(H+(\eta)_\epsilon(t)-2\epsilon^\alpha+x_d)\psi(t,\bfr).
\end{align*}
%
%
For $\tau<\frac12\epsilon$ the function $(\psi(t),\Psie(t))$ is now an admissible test function for all $t\in [0,T]$.
For the continuity equation we do not need the extra approximation parameter for the test function since no boundary values are requested.

\begin{Theorem}[Consistency of the semi-discrete scheme \eqref{SS}]\label{Thm_C1}    
Let $(\vrh, \vuh, \etah)$ be a solution of the scheme \eqref{SS}. 
 Then for any 
$\varphi \in C^2([0,T]\times \mathbb{R}^d)$ we have
\begin{equation}\label{consistency_semi_den}
- \intO{\vrh^0 \varphi^0}  - \int_0^T \intOtau{t}{ \left( \vrh  \pdt \varphi +  \vrh \vuh \cdot \Grad \varphi \right) }  = \order(\TS^\vartheta),
\end{equation}
If moreover, $\overline{\eta}_\eps\rightarrow \eta$ in $C^\alpha([0,T]\times \Sigma)$ (for some $\alpha\in (0,1)$, then there exists $\vartheta>0$ for all
pairs $(\Psi,\psi)  \in C^2_0(0,T\times \mathbb{R}^d) \times C^2_0([0,T]\times\Sigma)$ as constructed in \eqref{eq:test function} we have uniformly in $\epsilon$ that for all $\tau\leq \frac12\epsilon$ and $\Psie$ satisfying \eqref{eq:epsilonlayer} that
\begin{multline} \label{consistency_semi_mom}
-\intO{\vrh^0 \vuh^0 \cdot \Psie^0}  - 
\int_0^T \intOtau{t}{ \left( \vrh  \vuh \cdot \pdt \Psie + \vrh \vuh\otimes \vuh : \Grad \Psie \right) } \\
+ \mu \int_0^T \intOtau{t}{ \Grad \vuh^{k} : \Grad \Psie }   
+ \left( \mu + \lambda \right) \int_0^T \intOt{t}{ \Div \vuh^{k} \Div \Psie } 
- \int_0^T \intOtau{t}{ p(\vrh) \Div \Psie }  
\\
-\intS{\pdt \eta(0) \psi^0} - \int_0^T \intS{\delta_t \etah \pdt \psi }\dt + \int_0^T \intS{K'(\etah) \psi} \dt
=
\int_0^T \intS{ \gh \psi }  +  \int_0^T \intOtau{t}{\bfh \cdot \bfPsih} + \order(\TS^\vartheta) .
\end{multline}
\end{Theorem}

\begin{proof}
To prove the consistency, we must test the discrete problem \eqref{SS} with piecewise constant in time test functions. Therefore, we have to apply the piecewise constant projection operator $\Pi_t$ introduced in \eqref{Pit} to the smooth test functions $ \varphi$, $\Psi$  and $\psi$. 
Note that for any $\phih = \Pit \phi$, $\phi \in \{\varphi, \Psie, \psi \}$ and  for any piecewise constant in time function $\rh$ it holds 
\[
\int_0^T \rh \phih  \dt =\sum_{k=0}^{N_T-1} \int_{I^k} \rh \Pit{\phi} \dt = \sum_{k=0}^{N_T-1} \int_{I^k} \rh \phi \dt = 
 \int_0^T \rh \phi \dt. 
\]
Thanks to this equality, hereafter, we will directly use smooth (in time) test functions to show the consistency of our numerical scheme. 

The construction of the $\Psie$ is such that we may multiply \eqref{SS_M} with the couple $(\Psie, \psi)$. 
As we are dealing with continuous in space functional spaces, we only need to treat the consistency error of the time derivative terms. 

First, for the time derivative term of the shell displacement, we have 
\begin{equation}\label{consistency_shell}
\begin{aligned} 
& \int_0^T  \intS{\delta_t \zh \psih } \dt
= \int_0^{T} \intS{ \frac{ \zh(t) - \zh(t-\TS)}{\TS} \psi(t)}  \dt
 =
 \frac1{\TS} \int_0^T \intS{\zh(t) \psi(t) } \dt
 -   \frac1{\TS}  \int_{-\TS}^{T-\TS}   \intS{\zh(t) \psi(t+\TS) }  \dt
\\& =
\int_0^{T} \intS{ \zh(t) \frac{ \psi(t) - \psi(t+\TS)}{\TS} }  \dt
 +  \frac1{\TS} \int_T^{T+\TS} \intS{\zh(t) \underbrace{\psi(t+\TS)}_{=0}}  \dt
  - \frac1{\TS} \int_{-\TS}^{0}  \intS{\zh(t) \psi(t+\TS)} \dt
\\& =
- \int_0^T \intS{\zh \big(\pdt \psi + \frac{\TS}{2} \pdt^2 \psi|_{t^*} \big) }  \dt
  - \intS{ \psi^{0} \zh^0 } 
 =   \int_0^T \intS{\zh \pdt \psi }  \dt
  - \intS{ \psi^{0} \pdt \eta(0) } 
  +c(\norm{\psi}_{C^2}, \norm{\zh}_{L^\infty(0,T;L^2(\Sigma))}) \TS
  ,
\end{aligned}
\end{equation}
where $t^* \in(t,t+\TS)$ comes from Taylor's expansion in the last second equality. 
In the following we use $\rh$ as a substitute for either $\vrh$ or $\vrh\vuh$.
We begin by the observation, that
\begin{align*}
 &   \int_{t_{k}}^{t^{k+1}} \int_{\Oh^k}\rh^{k-1}\circ\bfX_{k-1}^k\mathcal{J}_{k-1}^k\Psie(t)\, dx\, dt
    = \int_{t_{k-1}}^{t^k} \int_{\Oh^k}\rh^{k-1}\circ\bfX_{k-1}^k\mathcal{J}_{k-1}^k\Psie(t+\tau)\, dx\, dt
\\&
    = \int_{t_{k-1}}^{t^k} \int_{\Oh^{k-1}}\rh^{k-1}\Psie\Big(t+\tau,\bfr, x_d\frac{\eta^{k}+H}{\eta^{k-1}+H}\Big)\, d\bfr\, d x_d\, dt
\end{align*}
 We find
\begin{equation*}
\begin{split}
& \int_0^T \intOtau{t}{D_t \rh \Psie }  \dt=\sum_{k=1}^{N_T} \int_{t_{k}}^{t^{k+1}}\int_{\Oh^k}\frac{\rh^k-\rh^{k-1}\circ\bfX_{k-1}^k\mathcal{J}_{k-1}^k}{\tau}\Psie(t)\, dx\, dt
\\
&=\sum_{k=2}^{N_T}\int_{t^k}^{t^{k+1}}\int_{\Oh^k}\rh^k\frac{\Psie(t,\bfr,x_d)-\Psie\Big(t+\tau,\bfr, x_d\frac{\eta^{k+1}+H}{\eta^{k}+H}\Big)}{\tau}\, dx\, dt - \frac{1}{\tau}\int_0^{\tau}\int_{\Oref}\rh^0\Psie(t)\, dx\, dt =I_1+I_2.
\end{split}
\end{equation*}
Next observe that
\begin{align*}
    &\Psie\Big(t,\bfr, x_d \frac{\eta^{k+1}+H}{\eta^{k}+H}\Big)-\Psie\Big(t,\bfr, x_d\Big)=\int_0^1\partial_d\Psie\Big(t,\theta x_d\frac{\eta^{k+1}+H}{\eta^{k}+H} +(1-\theta)x_d\Big)\, d\theta x_d\frac{\eta^{k+1}-\eta^{k}}{H+\eta^{k}}
    \\
    &= -\tau \vwh^k\cdot \nabla\Psie(t,\bfr,x_d) + \int_0^1\partial_d\Big(\Psie\Big(t,\theta x_d\frac{\eta^{k+1}+H}{\eta^{k}+H} +(1-\theta)x_d\Big)-\Psie(t,\bfr,x_d)\bigg)\, d\theta x_d\frac{\eta^{k+1}-\eta^{k}}{H+\eta^{k}}
    \\
     &=:-\tau \vwh^{k} \cdot \nabla\Psie(t,\bfr,x_d) +\mathcal{R}^k.
\end{align*}
By Taylor expansion (and the bounds on $\etah$) we find $\abs{\mathcal{R}^k}\leq c\norm{\nabla^2\Psi}_\infty\abs{\eta^{k+1}-\eta^k}^2$ which implies in particular that for $\alpha$, such that $\alpha=\frac{2}{\beta'}$ (where $\beta$ is defined via $\gamma$ in \eqref{eq:aprir}) and by Lemma~\ref{lem:apri} that
\[
\frac{1}{\tau}\int_{t^k}^{t^{k+1}}\int_{\Oh^k}\abs{\rh^k}\abs{\mathcal{R}^k}\, dx\leq c\norm{\nabla^2\Psi}_\infty\norm{\eta^{k+1}-\eta^k}^{2-\alpha}_\infty\tau^\alpha\norm{\zh^{k+1}}_{L^{\alpha\beta'}(\Sigma)}^\alpha\norm{\rh^k}_{L^\beta(\Oh^k)}
\leq C\tau^{\alpha+(2-\alpha)\theta}\norm{\rh^k}_{L^\beta(\Oh^k)}.
\]
Which implies by \eqref{eq:aprir2} and Lemma~\ref{lem:apri} that for $\frac{2}{\beta'} + \frac{2-\frac{2}{\beta'}}{3}-\frac{\beta}{2}>\vartheta>0$ we have 
$
\sum_{k=1}^{N_T}\frac{1}{\tau}\int_{t^k}^{t^{k+1}}\int_{\Oh^k}\abs{\rh^k}\abs{\mathcal{R}^k}\, dx=O(\tau^\vartheta)$. 
Actually $
\frac{2-\frac{2}{\beta'}}{3}-\frac{\beta}{2}>0\Longleftrightarrow \alpha +\frac{2-\alpha}{3}-\frac{1}{2-\alpha}>0$ 
which is true for all $\alpha<3/2$. Hence for all $\gamma>\frac{6}{5}$ there is a $\vartheta>0$. The maximum is achieved for $\alpha=2-\sqrt{3/2}$, then $\vartheta=2-2\sqrt{2/3}$ and $\beta=\sqrt{6}$ which is admissible for $\gamma>\frac{6}{5}(1+\sqrt{6})$ in 3d and all $\gamma>1$ in 2D.

%

Now we calculate using Taylor's expansion (using the uniform bounds on $\norm{\partial_t^2\Psie}_\infty \tau \norm{\rh}_{L^1(0,T;L^1(\Oh)} $) we find for a suitable $\theta>0$ that
\begin{align*}
    I_1&= \sum_{k=2}^{N_T}\int_{t^k}^{t^{k+1}}\int_{\Oh^k}\rh^k\frac{\Psie(t,\bfr,x_d)-\Psie\Big(t+\tau,\bfr, x_d\Big)}{\tau} + \int_{\Oh^k}\rh^k\frac{\Psie(t+\tau,\bfr,x_d)-\Psie\Big(t+\tau,\bfr, x_d\frac{\eta^{k+1}+H}{\eta^{k}+H}\Big)}{\tau}\, dx\, dt
    \\
    &= -\sum_{k=2}^{N_T}\int_{t^k}^{t^{k+1}}\int_{\Oh^k}\rh^k\partial_t\Psie\, dx\, dt+O(\tau) 
 - \sum_{k=2}^{N_T}\int_{t^k}^{t^{k+1}}\int_{\Oh^k}\rh^k\vwh^{k} \cdot \nabla\Psie(t+\tau)\, dx\, dt + O(\tau^\theta).
\end{align*}

Consequently we derive 
\begin{equation}\label{CT}
\begin{split}
 \int_0^T \intOtau{t}{D_t \rh \Psie }  \dt   + \intOtau{t=0}{\rh^0\Psi^0}\dt
 +  \int_0^T  \intOtau{t}{ \rh(t) \pdt \Psie(t)  } \dt 
\\+ \int_0^T \intOtau{t}{\rh(t) \vwh(t) \cdot \Grad \Psie (t) }
= \order(\TS^\vartheta), \; \vartheta >0, 
\end{split}
\end{equation}
for $\rh$ being $\vrh$ we may take $\Psie\equiv \varphi$  in case $\rh$ = $\vrh\vuh$ we have to take the $\epsilon$-approximation.

Finally, substituting  \eqref{CT} into the continuity method, \eqref{CT} and \eqref{consistency_shell} into the coupled momentum and structure method \eqref{SS_M}, we finish the proof. 
\end{proof}
%
%
%
%

\begin{Remark}
In variable domain analysis (in particular in fluid structure interaction involving elastic solids) it is unavoidable to approximate the space of test functions at some point. In our case we do this by introducing the parameter $\epsilon$. 
We wish to indicate what are the next steps in order to prove that a subsequence converges to a weak solution, which will be the content of a second paper (relaying on this work). The energy estimate allows to take weakly converging sub-sequences (in $\tau$). In order to pass with $\tau\to 0$ one has to prove that the various non-linearities as the pressure and the convective terms do indeed decouple in the limit. This is a sophisticated analysis introduced in~\cite{BS} and goes back to seminal works of Lions~\cite{Li2}. The last step is then to pass with $\epsilon \to 0$. This limit passage is how ever not as dramatic (essentially it uses Taylor expansion); but it depends sensitively on the regularity of $\partial_t\eta$ and in particular on the fact that $\gamma>\frac{12}{7}$. 
\qed
\end{Remark}

\section{fully discrete scheme}
In this section, we propose a fully discrete scheme for the FSI problem \eqref{pde}. For the time discretization, we take the method introduced in Section~\ref{sec:SS}. Further, for the space discretization, we take a mixed finite volume-finite element method proposed by Karper~\cite{Karper} for the compressible Navier-Stoks part \eqref{pde_D}--\eqref{pde_M} and a standard finite element method for the shell part \eqref{pde_S}. As in the last section we keep $\eps$ as the time discretization parameter. For the space discretization we introduce the value $h$ which is assumed to be coupled to $\eps$ in a convenient manner\footnote{For the consistency actually we will assume that $h\sim \tau$.
}. In the following we will use $\vrht,\vuht,\vwht,\etaht,\zht$ as functions which are discrete in space and piecewise constant in time. Similar notations with the same subscripts will be applied to all functions that will appear in this section. 
We shall write $a\aleq b$ if $a\leq c b$ for some positive constant $c$ (independent of $h$ and $\TS$), and $a \approx b$ if $a \aleq b$ and $b \aleq a$.  

\subsection{Discretization}\label{ssec:discret}
For the discretization in time, we follow the previous section and approximate all unknowns including the mesh and test functions by piecewise constant in time functions. 
For the space discretization, we start with the notations on the fixed reference domain. 
\paragraph{Mesh for the fluid part} 
Let $\Omega^0=\Oref$ (the reference domain) be a closed polygonal domain, and $\Thz$ be a triangulation of $\Omega^0$: $ \overline{\Omega^0} = \cup_{K \in \Thz} K $. 
The time evolution of the domain (or mesh) is described by the  ALE mapping for $k\in \{0,..,N_T\}$ that 
\[
\Oht^k= \Oref \circ (\ALEhtk)^{-1} \mbox{ and } \Thk=\Thz \circ (\ALEhtk)^{-1}
 \] 
 where the ALE mapping $\ALEhtk$ will be given in \eqref{proj_ALE} below. 
%
Further, we adopt the following notations and assumptions for the mesh \seb{of the fluid part}. 
\begin{itemize}
\item By $\edges(K)$ we denote the set of the  edges  $\sigma$ of an element $K \in \Tht$. The set of all edges is denoted by $\edges$. We distinguish exterior and interior edges: 
$
\edges = \edgesint \cup \edgesext,\
\edgesext = \left\{ \sigma \in \edges \ \Big| \ \sigma \in \partial \Oht \right\},\
\edgesint = \edges \setminus \edgesext.
$
\item We denote the set of all faces on the top boundary by $\edgesextS$ $(\subset \edgesext)$. 
\item For each $\sigma \in \faces$ we denote $\bfn$ as the outer normal. Moreover, for any $\sigma=K|L$, we write $\bfn_{\sigma, K}$ as the normal vector  that is oriented from $K$ to $L$ (so that $\bfn_{\sigma, K}=-\bfn_{\sigma,L}$), where $K|L$ denotes a common edge.
\item We denote by $ |K|$ and $ |\sigma|$ the Lebesgue measure of the element $K$ and edge $\sigma$  respectively. Further, we remark $h_K$ as the diameter of $K$ and  $h=\max_{K \in \Thz} h_K$ as the size of the triangulation. The mesh is regular and quasi-uniform in the sense of \cite{Ciarlet_fem}, 
 i.e. there exist positive real numbers $\theta_0$ and $c_0$ independent of $h$ such that
$ 	\theta_0 \leq   \inf \left\{ \frac{\xi_K}{h_K}, K \in \Thz \right\}$  and $c_0 h \leq h_K   $,  
where $\xi_K$ stands for the diameter of the largest ball included in $K$. 
\item The mesh is built by an extension of the $(d-1)$-dimensional bottom surface mesh in the $d^{\rm th}$ direction, i.e., the projection of any element in the $d^{\rm th}$ direction must coincide with an edge $\sigma \in \edgesext$ on the bottom surface. We give an example in two dimensions for illustrating such kind of mesh, see Figure~\ref{fig_mesh}. \seb{In particular $\Th^k$ is assumed to be a conformal triangulation uniform in $k,h,\tau$.}
\end{itemize}
\paragraph{Mesh for the structure part} 
 The mesh discretization of the time independent domain $\Sigma$ coincides with that of the initial mesh of the top boundary $\Sh=\edgesextS^0$.  

\begin{figure}\centering
\includegraphics[width=0.3\linewidth]{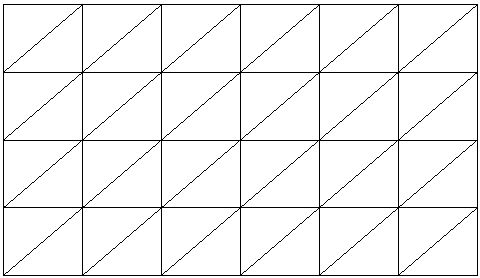}
\includegraphics[width=0.3\linewidth]{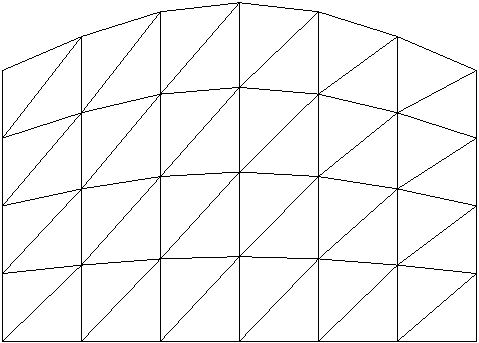}
 \caption{An example of mesh in two dimensions: left is the reference mesh and right is the deformed current mesh.}\label{fig_mesh}
\end{figure}
\begin{Remark}
On one hand, the mesh is constructed by the extension of the mesh of the $(d-1)$-dimensional bottom boundary. On the other hand, we will define a linear function for the discrete ALE mapping $\ALEht$, see \eqref{linear_ALE} below.   
As a consequence, any triangle on the reference mesh is be kept to be a triangle in the current mesh, see a two dimensional mesh discretization in Figure~\ref{fig_mesh}. 
\end{Remark}

\paragraph{Functional spaces}
Our scheme utilizes spaces of piecewise smooth functions, for which we define the traces
\[\begin{aligned}
v^{\rm out} = \lim_{\delta \to 0} v(x + \delta \vc{n}), \;
v^{\rm in} = \lim_{\delta \to 0} v(x - \delta \vc{n}), \quad x \in \sigma, \ \sigma \in \edges.
\end{aligned}\]
Note that, $v^{\rm out}_\sigma$ is set according to the boundary condition  for an exterior edges $\sigma \in \edgesext$. 
We also define 
\[
\jump{v}_\sigma = v^{\rm out} - v^{\rm in}, \ \Ovg{ v } = \frac{ v^{\rm out} + v^{\rm in} }{2}, \ \avg{v} = \frac{1}{|\sigma|} \intsh{v}.
\]
Next, we introduce on the reference mesh 
the space of piecewise constant functions
\[
\hXspace  = \left\{ \varphi \in L^1(\Oref) \ \Big|\ \varphi|_K = {\rm const} \in \R \ \mbox{for any} \ K \in \Thz \right\},
\]
and the space of the linear Crouzeix--Raviart finite element 
\[
\hYspace = \left\{ v \in L^2(\Oref) \ \middle| \ v|_K = \mbox{affine function on} \ K \in \Thz, \
\intsh{ \jump{v}_\sigma }=0 \ \mbox{for}\ \sigma \in \edgesint^0\right\},
\]
and the space of piecewise quadratic functions on the shell $\Sigma$ 
\[ \hWspace = \left\{ q \in  C^1(\Sigma)
 \ \Big| q|_\sigma \in \mathcal{P}^2(\sigma) 
 \ \mbox{for}\ \sigma \in \Sh \right\} .
\] 
In accordance with \eqref{two_domain}  we denote $\vht = \hvht \circ \ALEht^{-1}$ 
for all $t\in(0,T)$ and 
for all unknowns including the test functions $v\in\{ \vr, \vu, p, \eta, z , \vw, \varphi, \bfPsi, \psi\}$ as well as the function spaces  
\[ \Xspace(\Oht(t)) = \hXspace \circ \ALEht^{-1}(t),\quad   \Yspace(\Oht(t)) = \hYspace \circ \ALEht ^{-1}(t), \quad   \Wspace(\Sigma) \equiv \hWspace(\Sigma).\] 
Thus it is clearly that 
\begin{equation*}
\varphiht \in \Xspace \Longleftrightarrow \hvarphiht  \in \hXspace, \quad 
     \bfPsiht \in \Yspace \Longleftrightarrow \hbfPsiht \in \hYspace, \quad
            \psiht \in \Wspace \Longleftrightarrow \hpsiht \in \hWspace \mbox{ and } \psiht \equiv \hpsiht.
\end{equation*}
The associated projections of the functional spaces are 
\[
\Pi_\grid : L^1 (\Oht) \to \Xspace , \ \Pim{v} =  \frac{1}{|K|} \int_K v \ \dx, \ K \in \Th.
\]
and \seb{(the uniquely defined interpolation operator~\cite{EG})}
\[
\Pi_\faces : W^{1,1}(\Oht) \to \Yspace,\text{\seb{such that} } \ \ints{ \Pie{v} } = \ints{v} \ \mbox{for any}\ \sigma \in \edges.
\]
We shall frequently write $
\avg{v} =  \frac{1}{|\sigma|}\ints{ \Pie{v} } $.  
Finally\footnote{Please observe the difference between $\widehat{\nabla}_h \hbfPsiht$ and $\reallywidehat{\nabla \bfPsiht}$ which is defined in \eqref{div_twomesh}.} we use the definition of $\widehat{\nabla}_h \hbfPsiht$ as the discrete gradient for $\hbfPsiht \in \hYspace$.

\paragraph{Coupling at the boundary and ALE mapping}
Following the above notation we denote by $\Gamma_S^k=\Gamma_S(t^k)$ the piecewise quadratic boundary produced by $\eta_h^k$. Namely
$ \Gamma_S^k=\{(\bfr,x_d)\in[0,L]^{d-1}\times [0,\infty),\:\ x_d=H+\etaht^k(\bfr)\}$. 
As in the last section we define $\etaht(t)=\etaht^k$ on $[k\eps,(k+1)\eps)$.
For the fluid domain we require the upper boundary of the mesh $\Tht$ to be of piecewise linear geometry. Therefore, we take a piecewise linear projection operator by keeping the values at the vertices of elements on $\Sh$. Let $\phi_j$ be the standard piecewise linear basis function of the $(d-1)$-dimensional mesh $\Sh$ defined by the vertices $\{r_i\}_{i=1}^{N_p}\subset \Sh$, where $N_p$ is the total number of vertices on $\Sh$. Then such a projection reads
\[ \Pi_p:\mathcal{P}^2(\Sigma_h)\mapsto\mathcal{P}^1(\Sigma_h). \quad \Pi(\etaht(r)) = \sum_{j=1}^{Np} \phi_j \etaht(r_j). \]
%
We also illustrate such a projection in Figure \ref{fig_proj} in the case of $d=2$, where the red line is $\Gamma_S$ while the dashed blue line is $\edgesextS$ that determines the ALE mapping. 
\begin{figure}[!h]\centering
\vspace{-0.4cm}
%
%
\begin{tikzpicture}[scale=1.0]

\draw[thick,red]   (0, 2) .. controls (1.5, 2.7) and (2.5, 2.7) .. (4, 2);
\draw[very thick, dashed,blue](0,2)--(2,2.52)--(4,2);
\draw[very thick] (2,0.4)--(0,2);
\draw[very thick] (2,0.4)--(2,2.52);
\draw[very thick] (2,0.4)--(4,2);
\draw[very thick] (2,0.4)--(2,2.52);
\path node at (-0.07,2){\begin{large} $\cred  \star$ \end{large}};
\path node at (1.93,2.53) {\begin{large} $\cred  \star$ \end{large}};
\path node at (3.93,2) {\begin{large} $\cred  \star$ \end{large}};
\end{tikzpicture}
\vspace{-0.4cm}
\caption{Part of  $\Sh$ and $\Th$ near top boundary: the red line is $ \Gamma_S=\{\bfx\;|\; x_d =\etaht +H\}$; the dashed blue line is $\edgesextS =\{\bfx\;|\; x_d= H+\Pi_p(\etaht)\}$.}\label{fig_proj}
\end{figure}
Accordingly, the discrete ALE mapping~\eqref{ALE_mapping} is redefined due to the space discretization precedingly introduced  
\begin{align}\label{linear_ALE}
\bfx:=\ALEht(\bfxref) \equiv \left(\widehat{\bfr}, \frac{ \Pi_p(\etaht) +H}{H}\xref_d \right), \text{ with } \bfr = \widehat{\bfr}, \;  x_d=  \frac{\Pi_p(\etaht)  +H}{H}\xref_d.
\end{align}
Moreover, we need to update the definitions in \eqref{mesh_velocity} and \eqref{jacob} due to the ALE mapping. 
\begin{equation}\label{proj_ALE}
\begin{aligned}
& \hvwht^{k}= \frac{\ALEhtk -\ALEhtkm }{\TS} 
= \left( {\bf0}_{d-1}, \frac{\Pi_p(\etaht^{k} -\etaht^{k-1})}{\TS }\frac{\xref_d}{H } \right)^T, \quad 
 \vwht^{k} = \left( {\bf0}_{d-1}, \frac{\Pi_p(\etaht^{k} -\etaht^{k-1})\; x_d }{\TS \; \left(\Pi_p(\etaht^{k}) +H \right) }  \right)^T, 
\\
& \mJ_i^j =\det \left( \frac{\pd \bfX_i^j(\bfx^i)}{\pd \bfx^i}  \right) = \frac{\Pi_p(\etaht^{j}) +H }{\Pi_p(\etaht^i) +H}, \quad \mbox{and} \quad 
\TS \;\Div \vwht^{k} =  \frac{\Pi_p(\etaht^{k} - \etaht^{k-1})}{\Pi_p(\etaht^{k}) + H}  =   1- \mJkm.
\end{aligned}
\end{equation}

Please observe, that here the domain $\Oht$ is defined via $\etaht$ and its triangulation $\Th$ is defined by $\Pip\etaht$. Moreover, the Dirichlet boundary values of $\vuht$ will be defined by $\Pip\zht$. 


\subsection*{ Upwind divergence}
To approximate the convective terms, we apply a dissipative upwind operator
\begin{equation}\label{def_divup}
\Divup(\rht, \vvht)(\bfx) :=\sumK \frac{1_K}{|K|} \sum_{\sigma \in \facesK} \intsh{ {\rm Up}[\rht, \vvht]}, 
\end{equation}
where
${\rm Up}[r_{h,\TS}, \vvht]  
= \underbrace{ r_{h,\TS}^{\rm up}  \avg{\vvht \cdot \vc{n}} }_{\rm standard \ upwind} - \underbrace{\hheps \jump{r_h}}_{\rm artificial\ diffusion}
= \underbrace{\Ov{r _{h,\TS} } \avg{ \vvht \cdot \vc{n}}}_{\rm convective \ part} - 
\underbrace{\big( \hheps + \frac{1}{2} | \avg{ \vvht \cdot \vc{n}} | \big) \jump{ r_h }}_{\rm dissipative \ part}, \; \varepsilon > -1.$ 
Here 
$ r_{h,\TS}^{\rm up}|_{\sigma} = 
\rht^{\rm in} \avg{\vvht \cdot \vc{n}}^+ + \rht^{\rm out} \avg{\vvht \cdot \vc{n}}^- \quad \mbox{and} \quad f^{\pm} = \frac{f\pm \abs{f}}{2}. 
$ 
As pointed out in \cite{FLMS_NS}, the additional artificial diffusion included in the above flux function is $h^{\varepsilon +1}$ which indicates $\varepsilon >-1$. 
For the consistency  we will require
\[ \varepsilon \in (0, 2(\gamma-1)). \] 
Observe, that the artificial diffusion has been introduced in order to allow to show that a weakly converging subsequence converges to a weak solution with $h,\tau\to 0$. Actually, up to today it is an unavoidable regularization with respect to the analysis. There are, however no numerical experiments that show the necessity of the artificial diffusion in this context. We wish to emphasize that the existence of the scheme, its stability, mass conservation and positivity of the density {\em do not depend} on the additional artificial diffusion term. 
However, it is important for deriving the unconditional consistency of our numerical scheme without any assumption on the regularity of the numerical solution, see Theorem~\ref{Thm_C2}.  

In accordance with the relation \eqref{Piola}, we introduce the upwind divergence on the reference domain as
\[
\reallywidehat{ \Divup(\rht, \vvht)}:  = \Divup(\rht, \vvht) \circ \ALEht =
\sum_{K \in \Thz}  \frac{1_K}{|K|} \sum_{\sigma \in \facesK}  \intshref { {\rm Up} [ \hrh, \hvuht] \abs{\mJ \Jacob^{-1} \widehat{\bfn} } }.
\]


\subsection*{Preliminary inequalities}
We assume the readers are familiar with the techniques in finite element method. For the sake of completeness, we report a few necessary inequalities. As is common the constant depend on the regularity of the mesh; i.e.\ on the constants $\theta_0$ and $c_0$ above. As follows from our estimates~\eqref{est_J} the numbers $c_0$ and $\theta_0$ can be chosen independently of $\heps$ and $\eps$. 
Meaning that
 $	\theta_0 \leq   \inf \left\{ \frac{\xi_K}{h_K}, K \in \Tht \right\}   \quad\text{ and }\quad c_0 h \leq h_K,   
$
where $\xi_K$ stands for the diameter of the largest ball included in $K$. 
Moreover, it follows from~\eqref{est_J} it follows that 
 $ h_K\leq c_1 h \mbox{ for all }  K \in \Tht$, 
with a constant uniform in $\heps$ and $\eps$.

 Since these constants can assumed to be uniform w.r.t the change of variables (due to bounds on $\etaht$) the below estimates hold both on the reference domain, as well as on the current domain. For that reason we take $\Oht=\bigcup_{K\in \Th} K$ as a (regular) polygonal domain. We define for the discrete in space function $\vuht$ the following norms:
\[ \norm{\vuht}_{1,\Th} :=  \sum_{K\in \Th} \intK{|\Grad \vuht|^2}, \quad 
\norm{\vuht}_{H^1_Y}  := \intEh{\frac1h\jump{\vuht}^2}.
\]
Next we would like to introduce from Brenner the  Korn inequality~\cite[equation (1.19)]{Brenner_MC}.
\[ 
 \norm{\vuht}_{1,\Th} 
\aleq \norm{\bfD(\vuht)}_{L^2(\Oht)} 
 + \norm{\vuht}_{L^2(\Gamma)} 
+
\norm{\vuht}_{H^1_Y} ,
\]
and the Poincar\'{e}--Friedrichs inequality  \cite[equation(1.5)]{Brenner_SINUM} 
\[
\norm{\vuht}_{L^2(\Oht)} \aleq \norm{\vuht}_{1,\Th} + \norm{\vuht}_{L^1(\Gamma)},
\]
respectively for all $\vuht \in \Yspace$. 
Thus we deduce the following modified Korn inequality
\begin{equation}\label{Korn}
  \norm{\vuht}_{1,\Th}  +  \norm{\vuht}_{L^2(\Oht)} 
  \aleq C\left( \norm{\bfD(\vuht)}_{L^2(\Oht)}  
 + \norm{\vuht}_{L^2(\Gamma)} +  \norm{\vuht}_{H^1_Y} \right) 
\end{equation}
Further, we need the following version of Sobolev's inequality~\cite[Lemma 2.3]{FLNNS}
\begin{equation}\label{Sobolev}
\norm{\vuht}_{L^6(\Oht)} \aleq  \norm{\vuht}_{1,\Th} 
 + \norm{\vuht}_{L^1(\Oht)}^2, \; \forall\; \vuht \in \Yspace.
\end{equation}

Next, due to scaling argument we introduce the trace theorem~\cite[equation (2.26)]{FLNNS}
\begin{equation}
\label{eq:invtrace}
\norm{\vu}_{L^p(\pd K)} \aleq h^{-1/p} \norm{\vu}_{L^p(K)} ,  \; \vu \in \mathcal{P}^m(K), \; 1 \leq p \leq \infty, \; \forall\; K \in \Th, 
\end{equation}
where $\mathcal{P}^m(K)$ denotes the space of polynomial degree not grater than $m$.
It is worth mentioning the inverse estimate, see\cite{Ciarlet_fem} and \cite[Remark 2.1]{HS_MAC}
\begin{equation}
\label{inv_est}
\norm{\vu}_{L^{p_1}(0,T; L^{q_1}(\Oht))} \aleq  
\TS^{\frac{1}{p_1} - \frac{1}{p_2} }  h^{d(\frac{1}{q_1} - \frac{1}{q_2}) }
\norm{\vu}_{L^{p_2}(0,T; L^{q_2}(\Oht))}, \; \forall \; 1 \leq p_2 \leq p_1 \leq \infty, \; 1 \leq q_2 \leq q_1 \leq \infty .
\end{equation}
Finally, we recall the standard interpolation error estimates for $\phi \in C^1(\Oht)$~\cite{boffi} (see also \cite[Appendix]{GallouetIMA}).
\begin{equation}\label{IP}
\begin{aligned}
\jump{\Pim \phi} \aleq h, \; 
\norm{\Pim \phi - \phi}_{L^p} \aleq h, \; 
\norm{\Pie \phi - \phi}_{L^p} \aleq h, \; 
\norm{\Pim {\Pie \phi} - \phi}_{L^p} \aleq h ,\; 
 1\leq p \leq \infty.
\end{aligned}
\end{equation}
and also from~\cite[Lemma 2.7]{Karper} and  \cite{crouzeix_raviart} that
\begin{equation}\label{interpolation}
\begin{aligned}
   \norm{\vv - \Pim\vv}_{L^2(\Oht)}  \aleq  h \norm{\Grad \vv}_{L^2(\Oht)},\; \forall \vv \in \Yspace \mbox{ or } C^1(\Oht) \\
    \norm{\vv - \Pie\vv}_{L^2(\Oht)} + h \norm{\Grad(\vv - \Pie\vv)}_{L^2(\Oht)} 
    \aleq h^2 \norm{\vv}_{W^{2,2}(\Oht)}, \; \forall \vv \in W^{2,2}(\Oht).
\end{aligned}
\end{equation}

\subsection{The  scheme}
With the above notations, we are ready to present a mixed finite volume--finite element method for the FSI problem \eqref{pde}. First we present the scheme in the current domain. 
\begin{Definition}[Fully discrete scheme on the current domain]\label{def_FS} 
We seek $(\etaht^{k}, \vrht^{k}, \vuht^{k})\in (\Wspace,\Xspace(\Oht^k),\Yspace(\Oht^k))$ for all $k \in \{ 1,\ldots, N_T \}$, 
such that for all $(\varphiht^k, \bfPsiht^k, \psiht^k) \in (\Xspace(\Oht^k), \Yspace(\Oht^k), \Wspace(\Sigma))$ with $\Pie{\bfPsiht^k\circ\ALEhtk} = \Pie{\Pip{\psiht^k}}\bfe_d$ the following hold:
\begin{subequations}\label{FS}
\begin{equation} \label{FS_D}
\intOk{ D_t \vrht^{k} \varphiht}  + \intOk{\Divup ( \vrht^{k} , \vvht^{k} ) \varphiht}  
= 0;
\end{equation}
\begin{align} \label{FS_M}
\begin{aligned}
&\intOk{ D_t \left( \vrht^{k}  \avc{\vuht^{k}} \right) \cdot \bfPsiht +\Divup (\vrht^{k} \avc{\vuht^{k}}, \vvht^{k} ) \cdot \bfPsiht }
 \\& 
 + 2\mu \intOk{  \bfD(\vuht^{k}) : \Gradh \bfPsiht } 
 + 2\mu  \intEh{ \frac1h \jump{\vuht^k} \cdot \jump{\bfPsiht}}
+  \lambda \intOk{ \Divh \vuht^{k} \Divh \bfPsiht } 
\\
&- \intOk{ p(\vrht^{k}) \Divh \bfPsiht }
+ \intS{ \delta_t \zht^{k}  \psiht }  +  \intS{  \Lap \etaht^{k} \Lap  \psiht } 
=
  \intOk{\vrht^k \bfh^{k}\cdot \bfPsiht}  + \intS{\gh^{k} \psiht} ;
  \end{aligned}
\end{align}
where 
\[ \zht^k = \delta_t \etaht^k, \;  \vvht^{k} = \vuht^{k}-\vwht^{k}, \; \vwht^{k}(\bfx) = \left( {\bf0}_{d-1}, \frac{\Pi_p[\zht^k] x_d}{\Pi_p[\etaht^{k}] +H } \right)^T,\;  
\gh^k:=\frac{1}{\eps}\int_{I^k} \Pie{g}\, \dt \text{ and }\bfh^k:=\frac{1}{\eps}\int_{I^k} \Pie{\bff}\, \dt.
\]
The scheme is supplemented with the initial data and boundary conditions 
\begin{equation*}
 \vr^0_h = \Pim{\vr_0}, \; \vuht^0 \in \Pim{\vu_0}, \; \etaht^0=0, \; \zht^{0}= 0, \quad 
 \avg{\vvht}  = 0 ,\; \jump{\vrht^{k}}_{\sigma}=0, \; \forall\; \sigma \in \edgesext. 
\end{equation*}  
\end{subequations}
\end{Definition}
Pleas observe that by construction 
\[
\avg{\vuht^k}=\avg{\vwh^k}=\avg{\Pip{\zht^k}}\text{ for all }\sigma\in \Sh.
\]
Analogously to the semi-discrete case, the fully discrete scheme \eqref{FS} preserves the total mass as well
\begin{equation}\label{MCFS}
 \intOhk{\vrht^k} = \intOhkm{\vrht^{k-1}} = \cdots = \intOref{\vrht^0} =:M_0, \; \mbox{ for all } k=1,\ldots,N_T.
\end{equation}
Actually, by setting $\varphiht \equiv 1$ in \eqref{FS_D} and noticing the conservativity of the upwind flux lead to the expected result. 
\begin{Remark}
\label{rem:implicit}
Please observe that the above scheme is fully implicit and nonlinear. This means that both velocity $\vuh^k$ and density $\vrh^k$ are coupled to {\em their domain of definition} $\Oh^k$, which is determined by the unknown $\etah^k$ for each time step $k=1,2,\ldots, N_T$. This is in striking contrast to the  approaches from incompressible flows~\cite{Boris1} where the velocity and pressure can be solved for each time step in the domain of the previous step. Here a common problem for compressible fluids reveals itself: Due to the fact that the {\em renormalized} density equation is necessary to derive an energy inequality out of the (discrete) scheme seems to enforce an implicit and non-linear scheme. Indeed, until today there is no time discretization scheme for compressible fluids that provides energy estimates which is not both {\em nonlinear} and {\em fully implicit}  even for fixed domains. Unfortunately, in our investigations it turned out that also for fluid-structure interactions there is no space to allow explicit in time parts of the solutions. 
 Nevertheless, we can solve the scheme \eqref{FS} by rewriting its equivalent formulation on the reference domain $\Oref$ to avoid the problem of solving unknowns on an unknown domain, see scheme~\eqref{FSR} given below in Definition~\ref{def_Sref}. Though the scheme \eqref{FSR} is also a fully implicit and nonlinear scheme, we can solve the nonlinear system iteratively on the given reference domain. 
Furthermore, we will show that a full discretization in time and space actually possesses a solution, see Theorem~\ref{thm:exist}. In addition, we can assure that for a positive time interval that the fully discrete scheme is well-defined, see Corollary~\ref{cor:maxtimeh}. 
\end{Remark}

Recalling \eqref{two_domain} allows to transfer the scheme in the following way:
\begin{Definition}[Fully discrete scheme on reference domain]\label{def_Sref} 
We seek the solution $\etaht^{k}\in \Wspace(\Sigma)$ and 
$(\hvrht^{k}, \hvuht^{k} ) \in \Xspace(\Oref)\times \Yspace(\Oref)$ 
for all $k \in \{ 1,\ldots, N_T \}$, such that for all $(\hvarphiht^k, \hbfPsiht^k, \psiht^k) \in \Xspace(\Oref)\times\Yspace(\Oref)\times \Wspace(\Sigma)$
, with $\Pie{\hbfPsiht^k}|_{\Sigma} = \Pie{\Pi_p[\psiht^k(\bfr)]} \bfe_d$ (for all $\bfr\in \Sigma$ and all $k\in \{ 1,\ldots, N_T \}$) the following holds: 
\begin{subequations}\label{FSR}
\begin{equation} \label{FSR_D}
\intOref{ \frac{\hvrht^k \mJik -\hvrht^{k-1}\mJikm }{\TS}  \hvarphiht }  + \intOref{\reallywidehat{\Divup ( \vrht^{k}, \vvht^{k} )} \hvarphiht}  
= 0;
\end{equation}
\begin{align} \label{FSR_M}
\begin{aligned}
&\intOref{ \frac{\hvrht^k \avc{\hvuht^{k}} \mJik -\hvrht^{k-1}\avc{\hvuht^{k-1}}\mJikm }{\TS} \cdot \hbfPsiht }+ \intOref{\reallywidehat{ \Divup (\vrht^{k} \avc{\vuht^{k}}, \vvht^{k} )} \cdot \hbfPsiht \mJik} 
\\& +2 \mu \intOref{ \reallywidehat{ \bfD( \vuht^k)}  :   \reallywidehat{ \Grad \bfPsiht} \mJik}
 +2 \mu \intEhref{\frac{1}{ h} \jump{\hvuht^k}\cdot \jump{\hbfPsiht} 
 \abs{\mJik(\Jacob_0^k)^{-T} \widehat{\bfn} }
}
+  \lambda \intOref{ \reallywidehat{ \Divh \vuht^{k}}   \reallywidehat{  \Divh \bfPsiht} \mJik  } 
\\& - \intOref{ p(\hvrht^{k})  \reallywidehat{  \Divh \bfPsiht} \mJik }
+ \intS{ \delta_t \zht^{k}  \psi }  +  \intS{  \Lap \etaht^{k} \Lap \psi } 
 \quad =
  \intOref{\hvrht^k \hbfh^{k}\cdot \hbfPsiht \mJik }  + \intS{\gh^{k} \psiht} 
  \end{aligned}
\end{align} 
supplemented with the initial  and boundary conditions 
\begin{equation*}\label{FSR_ic}
 \vr^0_{\heps,\eps} = \Pim{\vr_0}, \; \vuht^0 = \Pim{\vu_0}, \; \etaht^0=0, \; \zht^{0}= 0,
\quad 
\avg{\hvvht^{k}}:=\avg{\hvuht^{k}-\hvwh^{k}}=0, \; \jump{\hvrht^{k}}_{\sigma}=0, \; \forall\; \sigma \in \edgesext . 
\end{equation*} 
\end{subequations}
\end{Definition}
Our numerical scheme  is nonlinear, nevertheless it can be shown that such solutions always exist.
\begin{Theorem}[Existence of a numerical solution and positivity of the density.]\label{thm:exist} 
Let $ 0<\vrht^{k-1}\in \Xspace(\Oht^{k-1})$, $(\vuht^{k-1}, \etaht^{k-1}, \zht^{k-1}) \in  \Yspace(\Oht^{k-1}) \times \Wspace(\Sigma) \times \Wspace(\Sigma)$ be given. For simplicity, we denote $\Oht^k$ as $\Oht$. Then there exists $0<\vrht^k\in \Xspace(\Oht)$ and $ (\vuht^k,\etaht^k,\zht^k:= \frac{\etaht^k -\etaht^{k-1}}{\TS}) \in  \Yspace(\Oht) \times \Wspace(\Sigma)\times \Wspace(\Sigma)$ satisfying the discrete problem \eqref{FS} (or equivalently \eqref{FSR}), where $\etaht^k = \etaht^{k-1} + \TS \zht^k$. 
\end{Theorem}
The proof is an adaption of previous approaches using homotopy arguments (see~\cite{GallouetMAC}). For the sake of completeness a rigorous proof can be found in the Appendix~\ref{Ap:exist}.

\subsection{Stability}\label{S}
Since the differences of the proofs of the renomalization and the stability between the semi and fully discrete scheme are merely of technical nature we put the respective proofs in the appendix.

First, the fully discrete scheme \eqref{FS} satisfies the renormalized continuity equation. 
\begin{Lemma}[Renormalized continuity equation]\label{lem_r1}
Let $(\vrht,\vuht)\in  \Xspace \times \Yspace$ satisfy the discrete continuity equation \eqref{FS_D}. Then for any function $B \in C^1([0,\infty))$ we have 
\begin{equation*}
\begin{aligned}
\frac{ 1}{\TS} \left( \intOhk{  B(\vrht^k)}  - \intOhkm{ B(\vrht^{k-1}) }  \right) +  \intOhk{ \left(\vrht^k B'(\vrht^k) -B(\vrht^k)  \right) \Divh \vuht^k} +D_1 + D_2 =0, 
\end{aligned}
\end{equation*}
where 
\begin{align*}
D_1 &=  \frac{1}{\TS}\intOhk{  \mJkm  \left( B(\vrht^{k-1}\circ \bfXkm)  - B(\vrht^k) - B'(\vrht^{k}) \big(  \vrht^{k-1}\circ \bfXkm-\vrht^k  \big)  \right)  }, 
\\ D_2&= \sumintKs{ \left(B'(\vr_K^{k})  \jump{\vrht^{k}}  - \jump{B(\vrht^{k})} \right) \left(\left[\avg{\vvht^{k} \cdot \bfn} \right]^- -\hheps\right) }. 
\end{align*}
Moreover, $D_1 \geq 0$ and $D_2 \geq 0$ provided $B$ is convex. 
\end{Lemma}
\begin{proof}
The proof is similar to Lemma~\ref{lem_r0} but we need to pay attention to the convective terms, see the proof in Appendix~\ref{Aprc}. 
\end{proof}
Next, we show the strictly positivity of the discrete density. 
\begin{Lemma}[Positivity of density]\label{lem_positive}  
Any solution to \eqref{FS} satisfies $\vrht^{k} > 0$ for all $k=1,\ldots, N_T$ provided $\vrht^0 > 0$.
\end{Lemma}
\begin{proof}
We prove via mathematical induction and start with the assumption $\vrht^{k-1}>0$. First, by exactly the same argument as in Lemma~\ref{lem_non-negative} we know that  $\vrht^{k} \geq 0$. Second, we assume there exists a $K\in \Thk$ such that $\vr_K^{k}=0$. Then a straightforward calculation from the density scheme~\eqref{FS_D} yields 
\begin{equation*} (0- |K'|\vr_{K'}^{k-1} )/ \TS=-\intK{\Divup(\vrht^k, \vvht^k)} \geq  - \sum_{\sigma\in\facesK}\vrht^{k,{\rm out}} [\avg{\vvht^k \cdot \bfn}]^- \geq 0, \mbox{ with } K' = K  \circ \bfX_{k-1}^k
\end{equation*}
which contradicts with the assumption $\vrht^{k-1}>0$. Thus we know $\vrht^k>0$ and finish the proof. 
\end{proof}
Further, setting $B=\Hc(\vrht)$ we get the following  internal energy balance.
\begin{Corollary}[Discrete internal energy balance]\label{lem_r11}   
Let $(\vrht,\vuht)\in  \Xspace \times \Yspace$ satisfy the discrete continuity equation \eqref{FS_D}. Then  there exists $\xi \in \co{\vrht^{k-1}\circ \bfXkm }{\vrht^{k}}$ and  $\zeta \in \co{\vr_K^{k}}{\vr_L^{k}}$ for any $\sigma=K|L \in \facesint^{k}$  such that
\begin{equation}\label{r1}
\frac{ 1}{\TS} \left( \intOhk{  \Hc(\vrht^k)}  - \intOhkm{ \Hc(\vrht^{k-1}) }  \right)
+ \intOhk{  p(\vrht^{k}) \Divh \vuht^{k} } = - D_1 - D_2, \mbox{ where: }
\end{equation} 
\begin{equation}\label{D1D2}
D_1= \intOhk{  \TS \mJ^k  \frac{\Hc''(\xi)}{2} \left| D_t^\ALE \vrht^{k} \right|^2 }  , \quad    
D_2= \intEhk{ \Hc''(\zeta) \jump{  \vrht^{k} } ^2 \left( h^\varepsilon  + \frac12 \left|\avg{\vvht^{k} \cdot \bfn}\right| \right)  } .
\end{equation}
\end{Corollary}

Analogously as the semi-discrete case, the fully discrete scheme \eqref{FS} (or \eqref{FSR}) dissipates the total energy. 
\begin{Theorem}[Energy stability of the fully discrete scheme \eqref{FS}]\label{thm_s2}  
Let $\left(\vrht^k, \vuht^k, \etaht^k\right)_{k=1}^{N_T}$ be a family of numerical solutions obtained by the scheme~\eqref{FS} (or \eqref{FSR}). Then for any $N=1,\ldots, N_T$ the energy is stable in the following sense 
\begin{equation*} 
\begin{aligned}
&   \int_{\Oht^N} E_f^{N}\dx  + \intS{ E_s^N}  + \TS \sum_{k=1}^N \intOhk{  \left(2\mu |\bfD( \vuht^k) |^2  + \lambda |\Divh \vuht^k |^2   \right)  }   + 2 \mu \TS \sum_{k=1}^N \intEh{\frac1h\jump{\vuht^k}^2}
\\&  \quad 
+ \frac{\TS^2 }2  \sum_{k=1}^N  \intS{ \left( |\delta_t \zht^{k} |^2+ \alpha \left| \Lap  \zht^{k} \right|^2 + \beta \left| \Gradh \zht^{k} \right|^2  \right) } 
+ \TS \sum_{k=1}^N \intOhk{\frac{\TS}{2} \vrht^{k-1} \circ \bfXkm \left| D_t \avc{ \vuht^{k}} \right|^2  } 
\\&  \quad 
+ \TS \sum_{k=1}^N (D_1 +D_2)
+ \TS \sum_{k=1}^N  \intEhk{ \left(  \frac12 \vrht^{k,up}   | \vvht^k \cdot \vc{n} |  + \hheps  \Ov{\vrht^{k}} \right) \jump{ \avc{ \vuht^{k}} }^2 } 
\\ &  =  \intOref{  E_f^{0} } + \intS{ E_s^0}  + \TS \sum_{k=1}^N  \intOk{\vrht^k \bfh^{k}\cdot \vuht^{k}} +  \TS \sum_{k=1}^N \intS{\gh^{k} \zht^{k}} 
\end{aligned}
\end{equation*}
where $D_1, \; D_2$ are given in \eqref{D1D2}, 
$ E_f^{k}=  \frac{1}{2} \vrht^k \left|\avc{ \vuht^{k}}\right|^2 + \Hc(\vrht^k)$ and  
$E_s^{k} =\frac12 (|\zht^{k}|^2 +  \alpha |\Lap  \etaht^{k}|^2 +  \beta |\Gradh \etaht^{k}|^2). $
\end{Theorem}
\begin{proof}
The proof is similar to the energy stability of the semi-discrete scheme, see Theorem~\ref{thm_s1}. We leave it to the appendix~\ref{A_thm_s2}. 
\end{proof}
\subsection{A-priori estimates}
Before proving the consistency of the scheme \eqref{FS} (or equivalently \eqref{FSR}) we derive some useful estimates. Due to the coherence of the argument we use the notation of Subsection~\ref{ssec:consistency}. In particular we use the same definition of the piecewise constant in time functions (as defined in~\eqref{Tspace}) and the piecewise constant domain $\Oht$.

Applying the modified Korn inequality~\eqref{Korn} and the Sobolev inequality~\eqref{Sobolev} to the energy estimates (Theorem~\ref{thm_s2}) and the definition of $D_1$ and $D_2$ (see \eqref{D1D2}) directly imply the following uniform bounds on the numerical solutions:
\begin{equation}\label{uniform_bounds_h}
\begin{split}
\norm{ \vrht }_{ L^{\infty}(0,T; L^\gamma(\Oht(\cdot)))} \leq c, \quad \norm{ \vrht |\vuht|^2}_{ L^{\infty}(0,T; L^1(\Oht(\cdot)))} \leq c, \\
 \norm{\bfD( \vuht) }_{  L^2(0,T; L^2(\Oht(\cdot)))}\leq c,  \quad \norm{\Divh \vuht }_{  L^2(0,T; L^2(\Oht(\cdot)))}\leq c, \quad 
 \int_0^T \intEh{\frac1h \jump{\vuht}^2} \leq c
 \\
 \norm{\Gradh \vuht }_{  L^2(0,T; L^2(\Oht(\cdot)))}\leq c,  \quad \norm{\vuht }_{  L^2(0,T; L^6(\Oht(\cdot)))}\leq c, \quad 
\norm{\zht }_{  L^\infty(0,T; L^2(\Sigma))}\leq c, \quad  
\norm{\Lap  \etaht }_{  L^\infty(0,T; L^2(\Sigma))} \leq c,
\\
\int_0^T \! \intEh{ \left(  \frac12 \vrht^{up}   | \vvht \cdot \vc{n} |  + \hheps  \Ov{\vrht}  \right) \jump{ \avc{ \vuht} }^2 } \leq c,
\;  
\int_0^T\! \intEh{ \Hc''(\zeta) \jump{  \vrht } ^2 \left( \hheps + \left|\avg{\vvht \cdot \bfn}\right| \right)  } \leq c, 
\\
\norm{ \vrht \Pim\vuht}_{ L^{\infty}(0,T; L^{\frac{2\gamma}{\gamma+1}}(\Oht(\cdot)))} \leq c,\quad 
\norm{ \vrht \vuht}_{ L^{2}(0,T; L^{\frac{6\gamma}{\gamma+6}}(\Oht(\cdot)))} \leq c.
 \end{split}
\end{equation}
 where $c$ depends on the external force $\bfh$ and $\gh$ as well as the initial data. Further, since the discretization of the displacement $\etaht$ is conformal we find for $\etaht$ that 
$     \norm{\Gradh  \etaht }_{  L^\infty(0,T; L^2(\Sigma))}\leq c$ and $\norm{  \etaht }_{  L^\infty(0,T; L^\infty(\Sigma))}\leq c.$ 
Moreover, by precisely the same argument as in Lemma~\ref{lem:apri} we find 
for all $\theta\in [0,\frac13)$ there exists a constant $C$ depending on the energy estimates and $\theta$, such that 
 $   \max_{k}\norm{\etaht^k(\bfr)-\etaht^{k-1}(\bfr)}_{L^\infty(\Sigma)}\leq C\tau^\theta$, 
which implies the following corollary by the very same argument as in the semi-discrete case.
\begin{Corollary}[Exclusion of self-touching]
\label{cor:maxtimeh}
Let $\tau^\theta\leq \frac{\delta_0}{C}$ and $\delta_1\geq 2\delta_0$. Then, if for some $k\in \{0,...,N_T\}$ we find that
$\inf_\sigma \etaht^k(\bfr)\geq \delta_1-H$, the $\etaht^{k+1}$ satisfies
$\inf_\sigma \etaht^k(\bfr)\geq \delta_1-\delta_0-H$.
Moreover, for every $\delta_0\in (0,H/2)$ there exists a $T_0$ just depending on the bounds of the energy inequality $H$, such that 
$\inf_{[0,T_0]\time \Sigma}\eta(t,\bfr)\geq \delta_0-H$.
\end{Corollary}
From the above and the $L^\infty$ bound of $\etaht$, we may assume in the following that there exist two positive constants $\delta_2> \delta_1 >0$ such that  
\begin{equation}\label{est_J}
  0<  \delta_1 \leq  \mJik = \frac{\etaht + H}{H} \leq \delta_2  .
\end{equation}
\begin{Remark}
Note the uniform upper and lower bounds on the Jacobian~\eqref{est_J} imply that all uniform bounds  in Lebesgue spaces appeared in this paper hold both on the reference domain and the current time-dependent domain. 
We emphasize this fact in the following by denoting  $L^p L^{q}$, $L^q$ for the norms  $L^p(0,T;L^q(\Oht(\cdot))$.

Moreover, by the same reasoning all estimates on integrals over the jumps, as well as on area-integrals that have been shown on the reference mesh  are also valid on the push forwarded mesh. 
\end{Remark}

In the following we collect some estimates that we need for the consistency proof. 
First, due to the fact that the trace theorem for $\vuht$ is not available we have to use a different estimate on $\vwht$.
First by the $L^2$-Stability of $\Pi_p$ we find $\norm{\vwht}_{L^\infty L^{2} }$ uniformly bounded. Second, since the projection $\Pi_p$ is conformal we actually do have a proper gradient of $\vwht$ and may interpolate to find
\begin{align*}
& \norm{\nabla \vwh^k}_{L^2(\Oh)}^2 \approx \norm{\widehat{\nabla} \hvwh^k}_{L^2(\Ohref)}^2 \leq c \int_{\Sigma}\abs{\nabla \Pip {\zht^k}}^2 \, d\bfr+ c\int_{\Sigma}\abs{ \Pip{\zht^k} }^2 d\bfr
\leq  \frac{c}{\tau^2} \norm{\nabla(\etaht^k-\etaht^{k-1})}_{L^2}^2 + c\norm{\zht^k}_{L^2}^2
\\
&= -\frac{c}{\tau^2}\int_{\Sigma}\Delta(\etaht^k-\etaht^{k-1})\cdot (\etaht^k-\etaht^{k-1}) + c\norm{\zht^k}_{L^2}^2
\leq \frac{c}{\tau} (\norm{\Delta \etaht^k}_{L^2}+\norm{\Delta \etaht^{k-1}}_{L^2})\norm{\zht^k}_{L^2}+ c\norm{\zht^k}_{L^2}^2
\end{align*}
But this implies by Sobolev embedding (using the fact that $\hvwht\equiv 0$ on $\Gamma_D$) that
\begin{align}
\label{eq:west}
    \norm{\vwht^k}_{L^6(\Oht^k)}\sim \norm{\hvwht^k}_{L^6(\Ohref)}\leq c \norm{\widehat{\nabla} \hvwht^k}_{L^2(\Ohref)}\aleq \tau^{-\frac{1}{2}}.
\end{align}
Interpolation (H\"older's inequality) implies for $\beta\in[0,1]$ and $q=6\beta+2(1-\beta)$
that
\[
 \norm{\vwht}_{L^\infty(0,T;L^q(\Oht))}\sim \norm{\hvwht^k}_{L^\infty (0,T;L^q(\Ohref))}
 \aleq \tau^\frac{-\beta}{2}.
\]
Further, from the definition of the $\vwht$ \eqref{proj_ALE} and the uniform lower and upper bounds of $\etaht$, we notice
\begin{equation}
\norm{\Divh \vwht }_{  L^\infty L^2 } 
= \norm{\frac{\zht}{H+\etaht} }_{  L^\infty L^2 } 
= H \norm{\frac{\zht}{H+\etaht} }_{  L^\infty (0,T; L^2(\Sigma))} 
\aleq c 
\end{equation}

Seeing $\vvht=\vuht -\vwht$, we discover the bound on $\vvht$  that 
\begin{equation}\label{est_v}
    \norm{\Divh \vvht }_{  L^2 L^2}\leq c.  
\end{equation}

The next lemma collects some estimates related to the errors that appear due to the convective terms. The proof of these estimates goes along the techniques developed by the community of the numerics for compressible fluids. Please consider the appendix~\ref{Apue} for a complete proof.
\begin{Lemma}[Useful estimates]\label{lemBs} Let $c>0$ be a constant independent of the parameters $\TS$ and $h$ (may depend on the initial data, the external force $\bff$ and $g$ and the mesh regularity).
\begin{enumerate}
\item Let $\vrht, \vvht$ be a solution of \eqref{FS_D} with $h \in(0,1)$ and satisfy the estimates 
\[ \norm{\vrht}_{L^\infty L^\gamma} \leq c, \quad \quad \norm{ \vrht |\Pim \vuht|^2}_{ L^{\infty}L^1 } \leq c, \quad 
\hheps \int_0^T \intEhk{ \Hc''(\zeta) \jump{  \vrht^{k} } ^2    } \leq c.
\]
Then the following holds
\begin{equation}\label{lemB0}
\norm{\vrht}_{L^2 L^2 } \leq c h^{-\frac{\varepsilon+2}{2\gamma}} \mbox{ and } \norm{\vrht \Pim \vuht}_{L^2 L^2 } \leq c h^{-\frac{\varepsilon+2}{2\gamma}} .
\end{equation}
\item Let $\vrht, \vvht$ be a solution of \eqref{FS_D} with $\gamma \geq 2$ and satisfies the estimates 
\[ \norm{\vrht}_{L^\infty  L^\gamma} \leq c, \quad 
\norm{\Divh \vvht}_{L^2 L^2 } \leq c.
\]
Then the following holds
\begin{equation}\label{lemB1}
\int_0^T  \sumEKh{ \frac{ \jump{\vrht}^2}{\max{\{\vrht^{\rm in}, \vrht^{\rm out}\}}}   \abs{\avg{\vvht \cdot \vc{n} }} }  \dt  \leq c .
\end{equation}
\item Let $\vrht, \vuht, \vwht$ satisfy the estimates in \eqref{uniform_bounds_h}. Then the following hold
\begin{subequations}\label{con_rm}
\begin{equation} \label{con_r}
    \int_0^T  \sumEKh{  \abs{ \jump{\vrht}  \avg{\vvht \cdot \vc{n} }^-} }  \dt  \leq 
c \TS^{-\frac14} h^{\theta} , 
\end{equation}
\begin{equation}\label{con_m}
\int_0^T  \sumEKh{  \abs{ \jump{\vrht}\Pim \vuht  \avg{\vvht \cdot \vc{n} }^-} }  \dt  \leq 
c  \TS ^{-1/4} \; h^\zeta,
\end{equation}
\end{subequations}
where 
\[ 
\theta = \begin{cases} -\frac12 & \mbox{if } \gamma \geq \frac65, \\
\frac{3\gamma-6}{4\gamma}& \mbox{if } \gamma \in (1, \frac65), \\
\end{cases}
\quad 
\zeta= \begin{cases} -\frac12 & \mbox{if } \gamma \geq \frac43, \\
\frac{7\gamma-12}{4\gamma}  & \mbox{if } \gamma \in (1,\frac43) .
\end{cases}
\]
\item Let $r, F \in \Qh(\Tht), \vv \in \Vh(\Tht)$ and $\phi \in C^1(\Tht)$. Then it holds
\begin{equation}\label{Lem_up}
\begin{aligned}
&\intO{r \vv \cdot \Grad \phi} 
=- \sumK \intK{ F  \Divup[r, \vv ]  }  + \sum_{i=1}^4 E_i(r), \quad  \mbox{where: }\\
& E_1(r)= \sumEKh{ (F-\phi) \jump{r}  \avg{\vv \cdot \vc{n} }^- } ,
\; 
E_2 (r)= \sumEKh{ \phi r  \big( \vv \cdot \vc{n}  - \avg{\vv \cdot \vc{n} } \big)    },
\\& 
E_3 (r)= \intO{r (F- \phi) \Divh \vv } , \quad 
E_4(r) =  \hheps \intEh{ \jump{r} \jump{F}  } .
\end{aligned}
\end{equation}
\end{enumerate}
\end{Lemma}

\subsection{Consistency}
With the a-priori estimates derived in the last subsection, we are ready to show the consistency of the fully discrete scheme \eqref{FS} (or equivalently \eqref{FSR}). For the momentum equation we have to introduce the $\epsilon$-layer again. 
\begin{Theorem}[Consistency of the fully discrete scheme \eqref{FS}]\label{Thm_C2}    
Let $(\vrht, \vuht, \etaht)$ be the numerical solution of the scheme \eqref{FS} with $\TS \approx h$, $\gamma>\frac65$ and $\varepsilon \in(0, 2(\gamma-1))$. Then for any 
$\varphi \in C^2_0(0,T;\mathbb{R}^d)$ we have 
\begin{equation}\label{consistency_den}
- \intO{\vrht^0 \varphi^0}  - \int_0^T \intOt{t}{ \left( \vrht  \pdt \varphi +  \vrht \vuht \cdot \Grad \varphi \right) }  \aleq 
\seb{\order(\heps)}.
\end{equation}
If moreover, $\etaht\rightarrow \eta$ in $C^\alpha([0,T]\times \Sigma)$ (for some $\alpha\in (0,1)$), then there exists a positive $\vartheta$ such that for all
pairs $(\Psi,\psi)  \in C^2_0(0,T\times \mathbb{R}^d) \times C^2_0([0,T]\times\Sigma)$ as constructed in \eqref{eq:test function} we have uniformly in $\epsilon$ that for all $\tau\leq \frac12\epsilon$ and $\Psie$ satisfying \eqref{eq:epsilonlayer} that
\begin{multline} \label{consistency_mom}
-\intO{\vrht^0 \vuht^0 \cdot \Psie^0}  - 
\int_0^T \intOt{t}{ \left( \vrht  \vuht \cdot \pdt \Psie + \vrht \vuht\otimes \vuht : \Grad \Psie \right) } \\
+ \mu \int_0^T \intOt{t}{ \Grad \vuht^{k} : \Grad \Psie }   
+ \left( \mu + \lambda \right) \intOt{t}{ \Div \vuht^{k} \Div \Psie } 
-  \intOt{t}{ p(\vrht) \Div \Psie } \dt 
\\
-\intS{\pdt \eta(0) \psi^0} - \int_0^T \intS{\delta_t \etaht \pdt \psi } +  \intS{K'(\etaht) \psi} \dt
 - \int_0^T \intS{ \gh \psi }\dt  -   \int_0^T \intOt{t}{\bfh \cdot \Psie}\dt 
 \aleq  \order(h^\vartheta).
\end{multline}
\end{Theorem}
\begin{proof}
To show the consistency of the numerical scheme, we take $\Psieh = \Pim \Psie$ and the pair $( \Psieh ,  \psiht )= (\Pie \Psie , \Pie \psi )$ as the test functions in the discrete density and momentum equation, respectively.
As mentioned already before due to the uniform conformity of the mesh with respect to time change we have bounds on the projection error independent of the time-step. And as before we will use below all quantities that are related to the triangulation like $\Tht,K,\sigma, \mathcal{E}$ as quantities that change from time-step to time-step. We deal with each term separately:
\paragraph{Step 1 -- time derivative terms} 
The consistency of the time derivative terms have been done in Theorem~\ref{Thm_C1}. 
Indeed, by recalling \eqref{CT} and \eqref{consistency_shell} and using the uniform in $\tau$ bounds on the spatial projection \eqref{IP}, we find that
\begin{subequations}
\begin{equation}\label{CTD}
\begin{split}
  \int_0^T  \intOt{t}{D_t \vrht \varphi }  \dt   + \intOt{t=0}{\vrht^0\varphi^0}\dt
 +  \int_0^T  \intOt{t}{ \vrht(t) \pdt \varphi(t)  } \dt 
\\+ \int_0^T \intOt{t}{\vrht(t) \vwht(t) \cdot \Grad \varphi (t) }
= \order(\TS^\theta)+\order(\heps), \; \theta >0, 
\end{split}
\end{equation}
\begin{equation}\label{CTM}
\begin{split}
  \int_0^T  \intOt{t}{D_t (\vrht \vuht) \cdot \Psie }  \dt   + \intOt{t=0}{\vrht^0\vuht^0 \cdot \Psie^0}\dt
 +  \int_0^T  \intOt{t}{ \vrht \vuht \cdot  \pdt \Psie(t)  } \dt 
\\+ \int_0^T \intOt{t}{(\vrht\vuht\otimes \vwht) : \Grad \Psie (t) }
= \order(\TS^\theta), \; \theta >0, 
\end{split}
\end{equation}

\begin{equation}\label{CTS}
\begin{aligned} 
 \int_0^T  \intS{\delta_t \zht \psi } \dt
 =   \int_0^T \intS{\zht \pdt \psi }  \dt
  - \intS{ \psi^{0} \pdt \eta(0) } 
  +\order(\TS)+\order(h)
  ,
\end{aligned}
\end{equation}
\end{subequations}

\paragraph{Step 2 -- convective terms} 
We first deal with convective terms of the discrete density problem by setting $r=\vrht$, $\vv = \vvht$, $\phi= \varphi$, and $F=\Pim \varphi$ in \eqref{Lem_up}
\[
\begin{aligned}
 \int_0^T \intO{ \vrht \vvht \cdot \Grad  \varphi} \dt  &=- \int_0^T \sumK \intK{ \Pim \varphi  \Divup[\vrht, \vvht ]  } \dt
+ \sum_{i=1}^4 E_i
 \\& =-  \intTO{  \varphi  \Divup[\vrht, \vvht ]  } 
+ \sum_{i=1}^4 E_i
\end{aligned}
\]
where 
\[
\begin{aligned}
& E_1(\vrht) = 
\int_0^T  \sumEKh{ (\Pim \varphi -\varphi) \jump{\vrht}  \avg{\vvht \cdot \vc{n} }^- }  \dt, \; E_3 (\vrht)= \int_0^T \intO{\vrht ( \Pim \varphi - \varphi) \Divh \vvht } \dt ,  
\\& 
E_2(\vrht) = \int_0^T \sumEKh{ \varphi \vrht  \big( \vvht \cdot \vc{n}  - \avg{\vvht \cdot \vc{n} } \big)    } \dt ,
 \; 
E_4 (\vrht)= \hheps \int_0^T  \intEh{ \jump{\vrht} \jump{\Pim \varphi }  } \dt.
\end{aligned}
\]
Next, we estimate the terms $ \sum_{i=1}^4 E_i$. 
\paragraph{Term $E_1(\vrht)$} Applying the estimate~\eqref{con_r} we get 
\[
 \abs{ E_1(\vrht) } \leq h \norm{\varphi}_{C^1}
\int_0^T  \sumEKh{  \abs{ \jump{\vrht}  \avg{\vvht \cdot \vc{n} }^-} }  \dt 
\leq c h^{\zeta_1}, 
\]
where 
\begin{equation}\label{zeta1}
\zeta_1 = \begin{cases} \frac14 & \mbox{if } \gamma \geq \frac65, \\
\frac{3(\gamma-1)}{2\gamma}& \mbox{if } \gamma \in (1, \frac65), \\
\end{cases}
\end{equation}

Obviously $\zeta_1 > 0$ for all $\gamma >  1$. 

\paragraph{Term $E_2(\vrht)$}It is easy to get from H\"older's inequality, the estimates \eqref{uniform_bounds_h}, the fact that $\vrht$ is piece wise constant, Gauss theorem and \eqref{lemB0} that 
\begin{equation}
 \label{PI}
\begin{aligned}
& \abs{ E_2(\vrht) } =  \Abs{
\int_0^T  \sumEKh{ \vrht ( \varphi - \avg{ \varphi} )   \big( \vvht \cdot \vc{n}  - \avg{\vvht \cdot \vc{n} } \big)  }  \dt  }
\\ &=\Abs{
\int_0^T  \sum_{K\in \Tht}\vr_K \int_{\partial K}  ( \varphi - \langle \varphi\rangle_{\partial K} )   \big( \vvht \cdot \vc{n}  - \langle\vvht\rangle_{K} \cdot \vc{n}  \big) }  \dSx \dt 
\\
 &=\Abs{
\int_0^T  \sum_{K\in \Tht} \vr_K \int_{K}  ( \varphi - \langle \varphi\rangle_{\partial K} ) \Divh \vvht  + \nabla \varphi \cdot (  \vvht- \langle \vvht\rangle_K )\dx }  \dt 
\\
 &=\Abs{
\int_0^T  \sum_{K\in \Tht} \vr_K \int_{K}  ( \varphi - \langle \varphi\rangle_{\partial K} ) \Divh \vvht  + (\nabla \varphi-\langle\nabla \varphi\rangle_K )\cdot \vvht\dx }  \dt 
\\&
\aleq h \norm{\varphi}_{C^2}\norm{\vrht}_{L^2L^2}(\norm{\Divh \vvht}_{L^2L^2}+\norm{ \vvht}_{L^2L^2}) \leq h^{\zeta_2},
\end{aligned}
\end{equation}
where $\zeta_2$ reads  
\begin{equation}\label{zeta2}
\zeta_2= \begin{cases}
1-\frac{\varepsilon+2}{2 \gamma} & \mbox{ if } \gamma\in (1,2), \\
1 & \mbox{ if } \gamma \geq 2. 
\end{cases}
\end{equation}
Obviously $\zeta_2>0$ as $\varepsilon<2(\gamma-1)$.

\paragraph{Term $E_3(\vrht)$} Applying H\"older's inequality and inverse estimate~\eqref{inv_est} we get
 \[
 \abs{ E_3(\vrht) } \leq h \norm{\varphi}_{C^1}  \norm{\vrht}_{L^2L^2}  \norm{\Divh \vvht}_{L^2L^2} 
\leq h  \norm{\vrht}_{L^2L^2}
\leq h^{\zeta_2},
\]
where $\zeta_2>0$ is the same as in \eqref{zeta2}.

\paragraph{Term $E_4(\vrht)$} Applying H\"older's inequality, the interpolation estimate~\eqref{IP}, the uniform bounds \eqref{uniform_bounds_h} and the fact
the fact $(\abs{a-b}\leq {a} +{b})$ for $a,b \geq 0$ we get 
\[
\begin{aligned}
& \abs{E_4 (\vrht)}= \hheps \abs{ \int_0^T \intEh{ \jump{\vrht} \jump{\Pim \varphi } } \dt }
 \leq h^{\varepsilon+1}  \int_0^T \intEh{ \abs{ \jump{\vrht} } } \dt 
\leq h^{\varepsilon+1}  \int_0^T \intEh{ 2 \Ovg{\vr} } \dt 
\\& \leq \hheps \norm{\vrht}_{L^1(0,T;\Oht)} \leq \hheps.
\end{aligned}
\]
Consequently, we derive 
 \begin{equation}
 \int_0^T \intOtau{\cdot}{ \vrht \vvht \cdot \Grad  \varphi} \dt   + \int_0^T \sum_{K}\intK{ \Pim \varphi  \Divup[\vrht, \vvht ]  } \dt
\leq h^\theta, \quad \theta=\min\{\zeta_1 , \zeta_2, \varepsilon\}. 
 \end{equation}
Clearly, $\theta>0$ for $\varepsilon \in (0, 2(\gamma-1))$ and $\gamma >1$.

Next, we deal with convective terms in the discrete momentum problem. We recall \eqref{Lem_up} with $r=\vrht \Pim{\vuht}$, $\vv = \vvht$, $\phi= \Psie$, $F= \Pim{\Pie\varphi}$
\[
\begin{aligned}
&\int_0^T \intO{ \vrht \Pim{\vuht} \otimes \vvht : \Grad  \Psie} \dt  
= - \int_0^T \sumK \intK{   \Divup(\vrht \Pim{\vuht}, \vvht) \cdot \Pim{\Pie \Psie} } \dt
\\&\phantom{ \int_0^T \intO{ \vrht \Pim{\vuht} \otimes \vvht : \Grad  \Psie} \dt  
=} 
+  \sum_{i=1}^4 E_i(\vrht \Pim{\vuht})
\\&\quad  =-  \intTO{    \Divup(\vrht\Pim{\vuht}, \vvht) \cdot \Pie \Psie } + \sum_{i=1}^4 E_i(\vrht \Pim{\vuht})
\end{aligned}
\]
where 
\[
\begin{aligned}
& E_1(\vrht  \Pim{\vuht}) = 
\int_0^T  \sumEKh{ ( \avc{\Pie \Psie} -\Psie) \jump{\vrht \Pim{\vuht}}  \avg{\vvht \cdot \vc{n} }^- }  \dt, 
\\& 
E_2(\vrht  \Pim{\vuht}) = \int_0^T \sumEKh{ \Psie \cdot (\vrht  \Pim{\vuht}) \big( \vvht \cdot \vc{n}  - \avg{\vvht \cdot \vc{n} } \big)    } \dt 
\\& 
E_3 (\vrht  \Pim{\vuht})= \int_0^T \intO{\vrht  \Pim{\vuht} \cdot ( \avc{\Pie \Psie} - \Psie) \Divh \vvht } \dt ,  
\\& 
E_4 (\vrht \Pim{\vuht})= \hheps \frac{1}{2}\int_0^T  \intEh{ \jump{\vrht \Pim{\vuht}} \cdot \jump{ \avc{\Pie \Psie} } } \dt
\end{aligned}
\]
Next, we estimate the terms $ \sum_{i=1}^4 E_i(\vrht \Pim{\vuht})$. 
\paragraph{Term $E_1(\vrht\Pim \vuht )$}By H\"older's inequality and the interpolation estimate~\eqref{IP} we get 
\[
\begin{aligned}
& \abs{ E_1(\vrht \Pim{\vuht}) } \leq h \norm{\Psie}_{C^1}
\int_0^T  \sumEKh{  \abs{ \jump{\vrht\Pim{\vuht}}  \avg{\vvht \cdot \vc{n} }^-} }  \dt 
\\& \leq h 
\int_0^T  \sumEKh{  \abs{ \big(\jump{\vrht} \Pim{\vuht} + {\vrht}^{\rm out}\jump{\Pim{\vuht}} \big) \avg{\vvht \cdot \vc{n} }^- } }  \dt =: I_1 +I_2,
\end{aligned}
\]
where we have also applied the chain rule
$ \jump{uv}_\sigma = {u}_\sigma^{\rm in} \jump{v}_\sigma +\jump{u}_\sigma {v}_\sigma^{\rm out} \mbox{ for all } u,v\in \Qh.$ 
Applying the estimate \eqref{con_m} we get the estimates of the first term 
\[I_1 =h \int_0^T\sumEKh{  \abs{ \jump{\vrht}\Pim{\vuht}  \avg{\vvht \cdot \vc{n} }} }  \dt \leq  h^{\zeta_3}   \]
where  
$\zeta_3= \begin{cases}1/4 & \mbox{if } \gamma \in [\frac43, \infty), \\
{(5\gamma-6)}/(2\gamma)  & \mbox{if }\gamma \in (1,\frac43), 
\end{cases}\quad  \zeta_3>0 \mbox{ provided } \gamma >\frac65.$  
The second term $I_2$ can be estimates by 
\[\begin{aligned}
& I_2 =h\int_0^T \sumEKh{  \abs{ \vrht^{\rm out} \jump{\Pim{\vuht}}  \avg{\vvht \cdot \vc{n} }^- } }  \dt 
\\ & \leq 
 h \left( \int_0^T \intEh{   \vrht^{\rm out} \abs{\avg{\vvht \cdot \vc{n} }^- } \jump{\Pim{\vuht}}^2    }  \dt \right)^{1/2}
\left( \int_0^T \intEh{    \vrht^{\rm out} \abs{\avg{\vvht \cdot \vc{n} }^- } }  \dt \right)^{1/2}
\\& \leq 
h^{1/2} \norm{\vrht}_{L^1L^{6/5}}^{1/2} \norm{\vvht}_{L^\infty L^6}^{1/2} 
\aleq 
h^{1/4} \norm{\vrht}_{L^\infty L^{6/5}}^{1/2}, 
\end{aligned} \]
where we have used the estimate \eqref{uniform_bounds_h}$_4$. 
It is obvious that 
$I_2 \aleq h^{\frac14}$ for $\gamma \geq \frac65$.  Further by the inverse estimate we derive for $\gamma \in(1, \frac65)$ that 
\[ I_2 \aleq h^{\frac14} h^{\frac32(\frac56 -\frac{1}{\gamma})}  \norm{\vrht}_{L^\infty L^\gamma}^{1/2} \aleq  h^{\frac{3(\gamma-1)}{2\gamma}}. 
\]
Consequently, $I_2 \aleq h^{\zeta_1}$, and  $\zeta_1>0$ for all $\gamma>1$, see \eqref{zeta1}. 

Combining the estimates of the terms $I_1$ and $I_2$ we get
\[
E_1(\vrht\Pim \vuht ) \leq h^{\zeta_3} + h^{\zeta_1}.
\]
\paragraph{Term $E_2(\vrht\Pim \vuht )$} We proceed as in \eqref{PI} using the fact that $\vrht\Pim{\vuht}$ is constant on each $K$. Hence we find analogously
 \[
\begin{aligned}
& \abs{ E_2(\vrht \Pim{\vuht}) } =  \abs{
\int_0^T  \sumEKh{ \vrht\Pim{\vuht} \cdot ( \Psie - \avg{ \Psie} )   \big( \vvht \cdot \vc{n}  - \avg{\vvht \cdot \vc{n} } \big)  }  \dt  }
\\ &=\Abs{
\int_0^T  \sum_{K\in \Tht} \sum_{j=1}^d\vr_K \Pim{(u_{K})_j} \int_{\partial K}  ( \Psie^j - \langle \Psie^j\rangle_{\partial K} )   \big( \vvht \cdot \vc{n}  - \langle\vvht\rangle_{K} \cdot \vc{n}  \big) \dSx}   \dt 
\\
 &=\Abs{
\int_0^T  \sum_{K\in \Tht} \sum_{j=1}^d\vr_K \Pim{(u_{K})_j} \int_{K}  ( \Psie^j - \langle \Psie^j\rangle_{\partial K} ) \Divh \vvht  + (\nabla \Psie^j-\langle\nabla\Psie^j\rangle_K )\cdot \vvht \dx}  \dt 
\\& \leq h \norm{\Psie}_{C^2} \norm{\vrht\Pim{\vuht}}_{L^2L^2}(\norm{\Divh \vvht}_{L^2L^2} +\norm{\vvht}_{L^2L^2} )\aleq h^{\zeta_2},
\end{aligned}
\]
where $\zeta_2 >0$ is given in \eqref{zeta2}.
\paragraph{Term $E_3(\vrht\Pim \vuht )$} Employing H\"older's inequality, the interpolation estimate~\eqref{IP} and the estimate \eqref{uniform_bounds_h} we derive 
 \[
 \abs{ E_3(\vrht \Pim{\vuht}) } \leq h \norm{\Psie}_{C^1}  \norm{\vrht \Pim \vuht}_{L^2L^2}  \norm{\Divh \vvht}_{L^2L^2} 
\leq h^{\zeta_2}
\]
where $\zeta_2 >0$ is given in \eqref{zeta2}.  
\paragraph{Term $E_4(\vrht\Pim \vuht )$} Using H\"older's inequality, the interpolation estimate~\eqref{IP}, and the estimate \eqref{uniform_bounds_h} we derive 
\[
\begin{aligned}
&  \abs{E_4 (\vrht \Pim{\vuht})}= \hheps \abs{ \int_0^T \intEh{ \jump{\vrht \Pim{\vuht} } \jump{\Pim \Psie }  } \dt }
\\& \leq h^{\varepsilon+1}  \int_0^T \intEh{ \abs{ \jump{\vrht \Pim{\vuht} } } } \dt 
\leq  h^{\varepsilon+1}  \int_0^T \intEh{ 2 \Ov{\vrht \abs{\Pim{\vuht}}}  }\dt
\\& \leq \hheps \norm{\vrht \Pim{\vuht} }_{L^1(0,T;\Oht)} \leq \hheps.
\end{aligned}
\]
Consequently, we derive 
 \begin{equation}
 \int_0^T \intO{ (\vrht \Pim{\vuht} \otimes \vvht) : \Grad  \Psie} \dt   + \int_0^T \sumK \intK{  \Divup[\vrht\Pim{\vuht}, \vvht ]  \cdot \Pie \Psie } \dt
\leq h^\theta
 \end{equation}
where $\theta= \min \{ \zeta_1, \zeta_2, \zeta_3, \varepsilon \} >0$ provided $\gamma >\frac65$ and  $\varepsilon \in (0, 2(\gamma-1))$.  
  
\paragraph{Step 3 -- pressure and diffusion terms} 
First, it is easy to calculate 
\begin{equation}
\intO{\pht \Divh \Pie \Psie} = \sumK p_K \intK {\Divh \Pie \Psie}
=\sumK p_K \int_{\facesK} { \Psie \cdot \vc{n} }
= \intO{\pht \Divh \Psie} .
\end{equation}
Similarly for the physical diffusion term we have
\begin{equation}
\intO{ \Divh  \vuht  \Divh \Pie \Psie} =\intO{ \Divh  \vuht  \Divh  \Psie} , \quad 
\intO{ \bfD(\vuht) : \Gradh\Pie \Psie} =\intO{ \bfD(\vuht) : \Gradh \Psie}.
\end{equation}
Concerning the penalty diffusion term, we control it as follows
\begin{align*}
\intEh{\frac1h \jump{\vuht} \cdot \jump{\Pie\Psie}} 
&= \intEh{\frac1h \jump{\vuht} \cdot \jump{\Pie\Psie - \Psie}} 
\aleq \norm{\vuht}_{H^1_Y} \norm{\Pie\Psie - \Psie}_{H^1_Y}
\\
&\aleq h \norm{\vuht}_{H^1_Y} \norm{\Psie}_{W^{2,2}}
\end{align*}
where we have used H\"older's inequality, the interpolation error~\eqref{interpolation} and the fact $\jump{\Psie}\equiv 0$. 

\paragraph{Step -- 4 rest of the structure part and external forces}
By the standard interpolation error, we have 
\begin{equation} 
\begin{aligned}
\int_0^T \intS{  \Lap \etaht  \Lap  \psi_{\heps,\eps } } &= \int_0^T \intS{  \Lap \etaht  \Lap  \psi }  + \order(h) ,
\\
\int_0^T \intO{\vrht \bfht \cdot \Psie} + \int_0^T \intS{ \ght  \psi_{\heps,\eps}} 
&=\int_0^T \intO{\vrht \bfht \cdot \Psie}  + \int_0^T \intS{\gh  \psi}  + \order(h) .
\end{aligned}
\end{equation}
Finally, collecting all the above terms we finish the proof. 
\end{proof}

\subsection*{Conclusion}
We have studied the fluid--structure interaction problem involving compressible viscous fluids. 
We have firstly proposed an energy stable time discretization scheme \eqref{SS}, see Theorem~\ref{thm_s1}. Our discretization fulfills the geometric conservation law, see \eqref{GC}. 
Moreover, we have shown that the numerical solutions satisfy the renormalized equation and they are consistent with respect to the weak solutions, see Lemma~\ref{lem_r0} and Theorem~\ref{Thm_C1}, respectively. 

Further, we have developed a fully discrete mixed finite volume--finite element method \eqref{FS}. We have proven the existence of a numerical solution to the scheme \eqref{FS} in Theorem~\ref{thm:exist}. 
We have shown that the numerical solutions of \eqref{FS} satisfy the renormalized equations (see Lemma~\ref{lem_r1}), mass conservation (see \eqref{MCFS}), positivity of density (see Lemma~\ref{lem_positive}),  energy dissipation (see Theorem~\ref{thm_s2}) and they are consistent to the weak solutions as well (see Theorem~\ref{Thm_C2}).

%
%
%

\appendix
\section{}

\subsection{Proof of Theorem~\ref{thm:exist}: existence of a numerical solution}\label{Ap:exist}
We aim to prove Theorem~\ref{thm:exist} for the existence of a numerical solution. Before that let us first introduce an abstract theorem, see \cite[Theorem A.1]{GallouetMAC}.
\begin{Theorem}(\cite[Theorem A.1]{GallouetMAC})\label{thm_A1} 
Let $M$ and $N$ be positive integers. Let $  C_1>\epsilon>0$  and $C_2>0$ be real numbers. Let $V$ and $W$ be defined as follows:
\begin{equation*}
V = \{ (x,y) \in R ^M \times R ^N, \ x>0 \}, \quad 
W = \{ (x,y) \in R ^M \times R ^N, \, \epsilon < x < C_1   \text{ and }  \| y\| \le C_2 \},
\end{equation*}
where the notation $x > c$ means that each component of $y$ is greater than $c$,  and $\| \cdot \| $ is a norm defined over $R ^N$. Let  $ F $ be a continuous function from  $V \times [0,1]$ to $R ^M \times R ^N$ satisfying:
\begin{enumerate}
\item $\forall\, \zeta \in [0,1] $, if $ v \in V $ is such that $ F(v,\zeta)=0 $ then $ v \in W $;
\item The equation $F (v, 0)=0$ is a linear system on $v$ and  has a solution in $W$.
\end{enumerate}
Then there exists at least a solution $ v \in W$ such that $F(v,1) = 0$.
\end{Theorem}
Now we are ready to show Theorem~\ref{thm:exist}. 
\begin{proof}
Let us denote $\Uht^k =(\vuht^k, \zht^k)$, $\Q(\Oht)=\left\{ (\bfPsi, \phi) \in \Yspace(\Oht) \times \Wspace(\Sigma)\middle| \bfPsi|_{\Sigma} =\psi \bfe_d \right\}$,  and define
\begin{equation*}
V =\{ ( \vrht^k, \Uht^k  ) \in \Xspace(\Oht)   \times \Q, \ \vrht^k > 0  \}.
\end{equation*}
It is obvious that the degrees of freedom of the spaces $\Xspace(\Oht)$ and $\Q(\Oht)$ are finite. Indeed, the space $\Xspace(\Oht)$ can be identified by the set of values $\vr_K$ for all $K \in \Thk$, therefore $\Xspace(\Oht) \subset \R^M$, where $M$ is the total number of elements of $\Thk$. Analogously, $\Q(\Oht)\subset \R^N$, where $N$ is the sum of $d$ times degrees of freedom of $\edges^k$ and the degrees of freedom of $\Sigma$. 
Let us consider the mapping
  \begin{align*}
   F : \ V  \times [0,1]\longrightarrow  \Xspace  \times \Q.  \qquad
        (\vrht^k,\Uht^k , \zeta)\longmapsto ( \vr^\star, U^\star) =F(\vrht^k,\Uht^k ,\zeta),
  \end{align*}
  where $( \vr^\star, U^\star) \in   \Xspace \times \Q $ is such that
\begin{subequations}
\begin{equation} \label{q1}
\intO{ \vr^\star \varphiht } = \intO{ \frac{\vrht^k  -\vrht^{k-1}\circ \bfXkm \mJkm }{\TS} \varphiht }  + \zeta \intO{\Divup ( \vrht^{k}, \vvht^{k} )\varphiht }  
;
\end{equation}
\begin{multline} \label{q2}
\intO{  U^\star \cdot \bfPsiht } =
\intO{ \frac{\vrht^k \avc{\vuht^{k}}  - (\vrht^{k-1}\avc{\vuht^{k-1}}) \circ \bfXkm \mJkm }{\TS}  \cdot \bfPsiht }
\\
+ \intS{ \frac{ \zht^{k} - \zht^{k-1}  }{\TS}\psiht }
+  \intS{  \Lap  \etaht^{k} \Lap  \psiht } 
-  \intO{\vrht^k \bfh^{k}\cdot \bfPsiht }  + \intS{\gh^{k} \psiht } 
\\
+\zeta \intO{\Divup (\vrht^{k} \avc{\vuht^{k}}, \vvht^{k} ) \cdot \bfPsiht }
- \zeta \intO{ p(\vrht^{k}) \Divh \bfPsiht } 
+ \zeta \lambda  \intO{ \Divh \vuht^{k} \Divh \bfPsiht }   
\\
+ \zeta 2\mu \intO{ \bfD( \vuht^{k} ) : \Gradh \bfPsiht }  
+ (1-\zeta) 2\mu \intO{ \left( (\Jacob_0^k)^{-1} \bfD( \vuht^{k} ) \right) : \left((\Jacob_0^k)^{-1} \Gradh \bfPsiht\right) }  
\\
+ \zeta 2\mu \intEh{\frac{1}{h} \jump{\vuht^k} \cdot \jump{\bfPsiht}}
+ (1-\zeta) 2\mu \intEh{\frac{1}{h}  \jump{\vuht^k} \cdot \jump{\bfPsiht} \frac{\abs{\sigma \circ \ALE}}{\abs{\sigma}}}
;
\end{multline} 
\end{subequations}
where 
$
\bfPsiht = (\bfPsiht, \psiht), \quad 
\etaht^k = \etaht^{k-1} + \TS \zht^k, \quad
\vuht^k|_{\Sigma} = \zht^k \bfe_d , \quad 
\mJkm = \left(H + \etaht^{k-1}\right)/\left( H + \etaht^{k}\right).
$
  
It is easy to check that $F$ is continuous. Indeed, it is a one to one mapping, since the values of $\vr^\star$ and $U^\star$ can be determined by setting $\varphiht = 1_{K}$ in \eqref{q1}, and $(\Phih)_i=1_{D_\sigma}, (\Phih)_j=0$ for $j\neq i$ in \eqref{q2}.

Let $(\vrht^k,\Uht^k ) \in \Xspace \times \Q$ and $\zeta \in [0,1]$ such that $F(\vrht^k,\Uht^k , \zeta)=(0,0)$ (in particular $\vrht^k>0$). Then for any $\big(\varphiht, \bfPhiht = (\bfPsiht, \psiht) \big)\in \Xspace \times \Q$  

\begin{subequations}
\begin{equation} \label{q3}
 \intO{ \frac{\vrht^k  -\vrht^{k-1} \circ \bfXkm \mJkm }{\TS} \varphiht }  + \zeta \intO{\Divup ( \vrht^{k}, \vvht^{k} )\varphiht}  =0 
;
\end{equation}
\begin{multline} \label{q4}
\intO{ \frac{\vrht^k \avc{\vuht^{k}} -(\vrht^{k-1}\avc{\vuht^{k-1}}) \circ \bfXkm \mJkm}{\TS}  \cdot \bfPsiht }
+ \intS{ \frac{ \zht^{k} - \zht^{k-1}  }{\TS}\psiht }
\\+  \intS{  \Lap  \etaht^{k} \Lap  \psiht } 
-  \intO{\vrht^k \bfh^{k}\cdot \bfPsiht }  + \intS{\gh^{k} \psiht} 
\\
+\zeta \intO{\Divup (\vrht^{k} \avc{\vuht^{k}}, \vvht^{k} ) \cdot \bfPsiht }
- \zeta \intO{ p(\vrht^{k}) \Divh \bfPsiht } 
+ \zeta \lambda  \intO{ \Divh \vuht^{k} \Divh \bfPsiht }   
\\
+ \zeta 2\mu \intO{ \bfD( \vuht^{k} ) : \Gradh \bfPsiht }  
+ (1-\zeta) 2\mu \intO{  \frac1{\mJik}\left((\Jacob_0^k)^{T} \bfD( \vuht^{k} ) \right) : \left((\Jacob_0^k)^{T} \Gradh \bfPsiht\right) }  
\\
+ \zeta 2\mu \intEh{\frac{1}{h} \jump{\vuht^k} \cdot \jump{\bfPsiht}}
+ (1-\zeta) 2\mu \intEh{\frac{1}{h}  \jump{\vuht^k} \cdot \jump{\bfPsiht} 
\frac{1}{ \abs{\mJik(\Jacob_0^k)^{-T} \widehat{\bfn} } }
}
.
\end{multline} 
\end{subequations}
Taking $\varphiht=1$ as a test function in \eqref{q3} we obtain
\begin{equation}\label{q5}
\norm{\vrht^k}_{L^1(\Oht)}= \intOk{\vrht^k} =  \intOkm{\vrht^{k-1}} >0,
\end{equation}
which indicates the boundedness of $\vrht^k$ in the $L^1$ norm, and thus in all norms as the problem is of finite dimension. Following the same argument as Lemma~\ref{lem_non-negative} we know that $\vrht^k \geq 0$  provided $\vrht^{k-1} \geq 0$. 

Taking $\bfPhiht=(\vuht^k, \zht^k)$ as the test function in \eqref{q4} and follow the proof of Theorem \eqref{thm_s1} gives 
\begin{equation}\label{q6}
\norm{\Uht^k } := \norm{\Gradh \vuht^k}_{L^2(\Oht)} + \norm{\zht^k}_{L^2(\Sigma)} \leq C_1
\end{equation} 
where $C_1$ depends on the data of the problem. 

Further, let $K\in \Thk$ be such that $\vr_K^k$ is the smallest, i.e., 
$ \vr_K^k \leq \vr_L^k$ for all $L\in \Thk$. We denote $K'= \ALEhtkm\circ (\ALEhtk)^{-1}(K)$. Then a straightforward computation gives 
\begin{align*}
&\frac{ \vr_K^k \abs{K} - \vr^{k-1}_{K'}\abs{K'} }{\TS \zeta}  = - \int_K \Divup (\vrht^k, \vvht^k)
=  - \sum_{\sigma \in \facesK} \abs{\sigma} \vrht^{k,up} \avg{\vvht^k \cdot \vc{n}} + 
 \sum_{\sigma \in \facesK} \abs{\sigma} \underbrace{ \hheps \jump{\vrht^k}}_{\geq 0} 
\\
&\quad  \geq  - \sum_{\sigma \in \facesK} \abs{\sigma}  \vr_K^k \avg{\vvht^k \cdot \vc{n}} 
+ \sum_{\sigma \in \facesK} \abs{\sigma}  (\vr_K^k - \vrht^{k,up} )\avg{\vvht^k \cdot \vc{n}}  
\\
&\quad  =
- \abs{K} \vr_K^k  (\Divh \vvht^k)_K - \sum_{\sigma \in \facesK} \abs{\sigma}  \jump{\vrht^k} \avg{\vvht^k \cdot \vc{n}}^-
 \geq - \abs{K} \vr_K^k  (\Divh \vvht^k)_K  \geq - \abs{K} \vr_K^k  \abs{(\Divh \vvht^k)_K} .
\end{align*}
Thus 
$
\vrht^k \geq  \vr_K^k \geq \frac{\abs{K'}}{\abs{K}} \frac{\vr^{k-1}_{K'} }{1 + \TS \zeta  \abs{(\Divh \vvht^k)_K} }  >0.
$
Consequently, by virtue of \eqref{q6}
$ \vrht^k > \epsilon $,
where $\epsilon$ depends only on the data of the problem. 
Further, we get from \eqref{q5} that
$ \vrht^k \leq \frac{ \intOkm{\vrht^{k-1}}}{\min_{K\in \Thk} |K|}$, 
which indicates the  existence of $C_2>0$ such that 
$\vrht^k <C_2$.
Therefore, the Hypothesis 1 of Theorem \ref{thm_A1} is satisfied. 

Next, we proceed to show that the Hypothesis 2 of Theorem \ref{thm_A1} is satisfied. 
Let $\zeta=0$ then the system $F(\vrht^k,\Uht^k )=0$ reads
\begin{subequations}\label{q78}
\begin{equation} \label{q7}
\vrht^k = \vrht^{k-1} \circ \bfXkm  \mJkm;
\end{equation}
\begin{multline}\label{q8}
\intO{ \frac{\vrht^k \avc{\vuht^{k}}  -(\vrht^{k-1} \avc{\vuht^{k-1}} ) \circ \bfXkm \mJkm}{\TS}  \cdot \bfPsiht }
+ \intS{ \frac{ \zht^{k} - \zht^{k-1}  }{\TS}\psiht }
\\
+  2\mu \intO{ \frac1{\mJik} \left( (\Jacob_0^k)^{T} \bfD( \vuht^{k} ) \right) : \left((\Jacob_0^k)^{T} \Gradh \bfPsiht\right) }  
+ 2\mu \intEh{\frac{1}{h}  \jump{\vuht^k} \cdot \jump{\bfPsiht} 
\frac{1}{ \abs{\mJik(\Jacob_0^k)^{-T} \widehat{\bfn} } }
}
\\
+  \intS{\left( \alpha \Lap  \etaht^{k} \Lap  \psiht + \beta \Grad  \etaht^{k} \Grad \psiht \right)} 
-  \intO{\vrht^k \bfht^{k}\cdot \bfPsiht}  + \intS{\gh^{k} \psiht} 
=0.
\end{multline} 
\end{subequations}
To solve the above system \eqref{q78}, we further reformulate it on the reference domain according to \eqref{Piola} 
\begin{subequations}
\begin{equation} \label{q9}
\hvrht^k \mJik= \hvrht^{k-1}\mJikm;
\end{equation}
\begin{multline}\label{q10}
\intOref{  \hvrht^{k-1}\mJikm \frac{\avc{\hvuht^{k}} -\avc{\hvuht^{k-1}}}{\TS}  \cdot \hbfPsiht }
+ \intS{ \frac{ \zht^{k} - \zht^{k-1}  }{\TS}\psiht }
+ 2\mu \intEhref{\frac{1}{h}  \jump{\hvuht^k} \cdot \jump{\hbfPsiht}}
\\+  2\mu \intOref{  \widehat{\bfD}( \hvuht^{k} ) : \widehat{ \Grad} \hbfPsiht}  
+  \intS{\left( \alpha \Lap  \etaht^{k} \Lap  \psiht + \beta \Grad  \etaht^{k} \Grad \psiht \right)} 
-  \intOref{\hbfh^{k}\cdot \hbfPsiht \hvrht^{k-1} \mJikm }  + \intS{g^{k} \psiht} 
=0,
\end{multline} 
where $\mJikm = 1+ \etaht^{k-1}/H$ and $\mJik = 1+ \etaht^k/H$.
\end{subequations}
Realizing that \eqref{q8} is a linear system with a matrix being block-wise symmetric positive definite, we know that there exists exactly one solution $\widehat{U}_{\heps,\eps}^k =(\hvuht^k, \zht^k)$. 
Then using the fact $\etaht^k = \etaht^{k-1} + \TS \zht^k$ we get $\etaht^k$ and $\ALEhtk$. Further, it is straightforward that $\vuht^k=\hvuht^k\circ \ALEhtk(\bfxref)$. Finally, substituting $\etaht^k$ into \eqref{q7} we obtain the solution for $\vrht^k$. Obviously, $\vrht^k>0$ as long as no self touching. Thus the solution $(\vrht^k,\Uht^k )$ belongs to $V$. 

We have shown that both Hypothesis of Theorem~\ref{thm_A1} hold. Applying Theorem~\ref{thm_A1} finishes the proof.
\end{proof}

\subsection{Proof of Lemma~\ref{lem_r1}: renormalization}\label{Aprc}
Here we show the validity of the discrete renormalized equation stated in Lemma~\ref{lem_r1} for the discrete continuity problem~\eqref{FS_D}.  
\begin{proof}
Firstly, we set $\varphiht=B'(\vr)$ in  \eqref{FS_D} and obtain
$\intOhk{ D_t \vrht^k B'(\vrht^{k}) }  + \intOhk{\Divup \left( \vrht^{k}, \vvht^{k}  \right) B'(\vrht^{k}) } 
 = 0 $.
Next, recalling \eqref{N1}, we know there exist $\xi \in \co{\vrht^{k-1}\circ \bfXkm}{\vrht^{k}}$ such that 
\begin{align*}
& \intOhk{ D_t \vrht^{k} B'(\vrht^{k})} =\intOhk{ \frac{\vrht^{k} -\vrht^{k-1}\circ \bfXkm \mJ^k }{\TS} B'(\vrht^{k}) }  
\\& = 
\frac{ 1}{\TS} \left( \intOhk{  B(\vrht^k)}  - \intOhkm{ B(\vrht^{k-1}) }  \right)
+  \intOhk{ (\vrht^k B'(\vrht^k) -B(\vrht^k)  ) \Divh \vwht^k}  + D_1
\end{align*}
where  
$D_1=  \frac{1}{\TS}\intOhk{  \mJkm  \left( B(\vrht^{k-1}\circ \bfXkm)  - B(\vrht^k) - B'(\vrht^{k}) \big(  \vrht^{k-1}\circ \bfXkm-\vrht^k  \big)  \right)  }. $

Further, by recalling the definition of the upwind flux~\eqref{def_divup}, and using again the Taylor expansion, we reformulate the convective term as 
\begin{align*}
&\intOhk{  \Divup (\vrht^{k},\vvht^{k}) B'(\vrht^{k}) } 
= \sumintK{ B'(\vrht^{k}) \sum_{\sigma\in \pd K}\frac{|\sigma|}{|K|}  \left( \vrht^{k,up} \avg{\vvht^{k} \cdot \bfn}  - \hheps \jump{\vrht^{k}}  \right) } 
\\&=
\intOhk{ \vrht^{k} B'(\vrht^{k})  \Divh\vvht^{k}}
+ \sumintK{ B'(\vr_K^{k})  \sum_{\sigma\in \pd K}\frac{|\sigma|}{|K|} (\vrht^{k,up}-\vr_K^{k}) \avg{\vvht^{k} \cdot \bfn} \!\! }
- \hheps \sum_{K\in \Thk}B'(\vr_K^k) \sum_{\sigma\in \pd K} |\sigma| \jump{\vrht^k}  
\\& =
\intOhk{ \vrht^{k} B'(\vrht^{k})  \Divh\vvht^{k}}
+ \sumintKs{B'(\vr_K^{k})  \jump{\vrht^{k}}  \left(\left[\avg{\vvht^{k} \cdot \bfn} \right]^- -\hheps\right) }
\\& =
\intOhk{ \vrht^{k} B'(\vrht^{k})  \Divh\vvht^{k}}
+ \sumintKs{ \jump{B(\vrht^{k})}  \left(\left[\avg{\vvht^{k} \cdot \bfn} \right]^- -\hheps\right) } +D_2 
\end{align*}
where $D_2= \sumintKs{ \left(B'(\vr_K^{k})  \jump{\vrht^{k}}  - \jump{B(\vrht^{k})} \right) \left(\left[\avg{\vvht^{k} \cdot \bfn} \right]^- -\hheps\right) }$. 
Moreover, using the facts
\[ \left[\avg{\vvht^{k} \cdot \bfn} \right]^- = \frac12 \left( \avg{\vvht^{k} \cdot \bfn} - \left|\avg{\vvht^{k} \cdot \bfn }\right| \right)
\mbox{ and } 
\sum_{K \in \Thk} \sum_{\sigma\in \pd K}\intsh{ \jump{B(\vrht^k)} \left(\left|\avg{\vvht^{k} \cdot \bfn} \right| + \hheps\right)  } =0,
\]
we obtain 
\begin{align*}
& \sumintKs{ \jump{B(\vrht^{k})}  \left(\left[\avg{\vvht^{k} \cdot \bfn} \right]^- -\hheps\right) } 
 =  \frac12 \sumintKs{ \jump{B(\vrht^{k})}  \avg{\vvht^{k} \cdot \bfn}} 
\\ & =
\sum_{K \in \Thk} B(\vr_K^{k}) \sum_{\sigma\in \pd K}\intsh{ \avg{\vvht^{k} \cdot \bfn} }
= \intOhk{B(\vrht^{k})  \Div \vvht^{k} }. 
\end{align*}
Consequently,  we derive 
$ \intOhk{  \Divup (\vrht^{k},\vvht^{k}) B'(\vrht^{k}) }  = 
\intOhk{ \left(\vrht^k B'(\vrht^k) -B(\vrht^k)  \right) \Divh \vvht^k} + D_2.$
  
Finally, collecting the above terms and seeing $\vvht^k + \vwht^k= \vuht^k$, we complete the proof, i.e., 
\begin{equation*}
\begin{aligned}
\frac{ 1}{\TS} \left( \intOhk{  B(\vrht^k)}  - \intOhkm{ B(\vrht^{k-1}) }  \right) +  \intOhk{ \left(\vrht^k B'(\vrht^k) -B(\vrht^k)  \right) \Divh \vuht^k} +D_1 + D_2 =0. 
\end{aligned}
\end{equation*}
\end{proof}

\subsection{Proof of Theorem~\ref{thm_s2}: energy stability}\label{A_thm_s2}
Here we prove the energy stability stated in Theorem~\ref{thm_s2} for the discrete scheme~\eqref{FS}.
\begin{proof}
Setting $\varphiht = - \frac{\left|\avc{ \vuht^{k}}\right|^2}{2}$ in \eqref{FS_D} and $(\bfPsiht,\psiht) = (\vuht^{k}, \zht^{k})$ in \eqref{FS_M} we get 
$\sum_{i=1}^2 I_i=0$ and $\sum_{i=3}^{9} I_i=0$ respectively, where 
\begin{equation*}
\begin{aligned}
& I_1 = -\intOhk{ D_t \vrht^k \frac{\left|\avc{ \vuht^{k}}\right|^2}{2} } , \; 
I_2= -\intOhk{\Divup \left( \vrht^{k}, \vvht^{k}  \right) \frac{\left|\avc{ \vuht^{k}}\right|^2}{2} } , 
\\&
I_3 = \intOhk{ D_t \left( \vrht^{k} \avc{ \vuht^{k}} \right) \cdot \vuht^{k} },  \; 
I_4 = \intOhk{\Divup \left(\vrht^k \avc{ \bfu}^{k}_h , \vvht^{k} \right) \cdot \vuht^{k} },
\\
&I_5 =- \intOhk{ p(\vrht^k) \Divh \vuht^{k} },
I_6 =  \intOhk{  \left(2\mu |\bfD(\vuht^k) |^2 + \lambda|\Divh \vuht^k |^2   \right)  } + 2 \mu \intEh{\frac1h\jump{\vuht^k}^2}
\\
& I_7=  \intOk{\vrht^k \bfh^{k}\cdot \vuht^{k}} + \intS{\gh^{k}  \zht^k} , \;
I_{8}= \intS{ \frac{\zht^k- \zht^{k-1}}{\TS} \zht^k },\; 
I_{9}= \intS{ \left(\alpha \Lap  \etaht^{k} \Lap  \zht^{k}  + \beta \Grad \etaht^{k} \cdot \Grad  \zht^{k} \right) }.
\end{aligned}
\end{equation*}
Now we proceed with the summation of all the $I_i$ terms for  $i=1,\ldots,9$. 

{\bf Term $(I_1+I_3+I_8)+(I_6+I_7)+I_9$.} 
Firstly, analogously as in the proof of Theorem~\ref{thm_s1} we have 
\begin{equation*}
\begin{aligned}
(I_1&+I_3+I_8)+(I_6+I_7)+I_9 
\\ =&  
\frac{1}{\TS} \left( \intOhk{ \frac12\vrht^{k} \left|\avc{ \vuht^{k}}\right|^2}   -  \intOhkm{ \frac12\vrht^{k-1} \left|\avc{ \vuht^{k-1}}\right|^2}   \right)
 + \frac{\TS}{2}  \intOhk{ \vrht^{k-1}\circ \bfXkm \left| D_t^\ALE \avc{ \vuht^{k}}  \right|^2  } 
 \\ &  + \intSh{ \left(\delta_t \left(\frac{|\zht^{k} |^2}{2} \right)  + \frac{\TS}{2}|\delta_t \zht^{k} |^2\right)  } 
+ \intOhk{  \left(2\mu |\bfD(\vuht^k) |^2 + \lambda|\Divh \vuht^k |^2   \right)  } + 2 \mu \intEh{\frac1h\jump{\vuht^k}^2}
 \\&  + \intOk{\vrht^k \bfh^{k}\cdot \vuht^{k}} 
 + \intS{\gh^{k} \zht^{k}} + \intS{ \frac12  \delta_t \left( \alpha |\Lap  \etaht^{k}|^2 + \beta |\Gradh \etaht^{k}|^2 \right)}
+ \intS{ \left( \frac{\TS\alpha}2 \left|\Lap  \zht^{k}\right|^2 +\frac{\TS\beta}2 \left|\Gradh \zht^{k}\right|^2 \right) }.
\end{aligned}
\end{equation*}

{\bf Term $I_2+I_4$.} For the convective terms, we have using the fact that $\Pim{\vuht}$ and $\Divup \left(\vrht^k \avc{ \bfu}^{k}_h , \vvht^{k} \right)$ are constant on each $K\in \Tht$ and the upwind divergence
\begin{equation*}
\begin{aligned}
& I_2+I_4= \intOhk{ -\Divup \left( \vrht^{k}, \vvht^{k}  \right) \frac{\left|\avc{ \vuht^{k}}\right|^2}{2} }+  \intOhk{ \Divup \left(\vrht^k \avc{ \bfu}^{k}_h , \vvht^{k} \right) \cdot \vuht^{k} }
\\&= \sumintKs{ \left( \vrht^{k,up} \avc{ \vuht}^{k,up} \cdot \avc{ \vuht^{k}} - \vrht^{k,up} \frac12 \left|\avc{ \vuht^{k}}\right|^2 \right) \vvht^{k}\cdot \bfn }
\\& \quad - \hheps \sumintKs{ \left(  \jump{\vrht^{k} \avc{ \vuht^{k}} } \cdot \avc{ \vuht^{k}} - \jump{\vrht^{k}} \frac12\left|\avc{ \vuht^{k}}\right|^2 \right)    }
\\&= \intEhkKL{\frac12 \jump{\avc{ \vuht^{k}}}^2 \left( \vr_K^{k}[\vvht^{k}\cdot \bfn_{\sigma,K} ]^+ +\vr_L^{k}[\vvht^{k}\cdot \bfn_{\sigma,L} ]^+  \right)} 
+ \hheps \intEhk{ \Ov{\vrht^{k}} \jump{\avc{ \vuht^{k}}}^2  }
\\& =  \intEhk{ \left(  \frac12 \vrht^{k,up}   | \vvht^k \cdot \vc{n} |  + \hheps \Ov{\vrht^{k}} \right) \jump{ \avc{ \vuht^{k}} }^2 }.
\end{aligned}
\end{equation*}

{\bf Pressure term $I_{5} $. }
Recalling the discrete internal energy equation \eqref{r1}, we can rewrite the pressure term as 
\begin{equation*}
I_5= -\intOhk{  p(\vrht^{k}) \Divh \vuht^{k} } =
\frac{ 1}{\TS} \left( \intOhk{  \Hc(\vrht^k)}  - \intOhkm{ \Hc(\vrht^{k-1}) }  \right)
+ D_1 + D_2,
\end{equation*}
where $D_1$ and $D_2$ are given in~\eqref{D1D2}. 
Collecting all the above terms, we get 
\begin{equation*} 
\begin{aligned}
& \frac{1}{\TS} \left(  \intOhk{  E_f^{k} }   - \intOhkm{ E_f^{k-1} }  \right) +  \intSh{  \delta_t \left( \frac{|\zht^{k}|^2}{2}   + \alpha \frac{|\Lap  \etaht^k |^2}{2}  + \beta \frac{|\Gradh \etaht^k |^2}{2} \right)}
\\& \quad 
+ \frac{\TS}2  \intS{ \left( |\delta_t \zht^{k} |^2+ \alpha \left|\Lap  \zht^{k}\right|^2 + \beta \left| \Grad \zht^{k}\right|^2  \right) }
  + \intOhk{  \left(2\mu |\bfD(\vuht^k) |^2 + \lambda|\Divh \vuht^k |^2   \right)  } 
  + 2 \mu \intEh{\frac1h\jump{\vuht^k}^2}
\\&= 
-D_1 -D_2- \intOhk{\frac{\TS}{2} \vr^{k-1}_h \circ \bfXkm
\left| D_t \avc{ \vuht^{k}} \right|^2  }
+  \intOk{\vrht^k \bfh^{k}\cdot \vuht^{k}} + \intS{\gh^{k} \zht^{k}}
\\ & \quad
 - \intEhk{ \left(  \frac12 \vrht^{k,up}   | \vvht^k \cdot \vc{n} |  + \Ov{\vrht^{k}}\hheps  \right) \jump{ \avc{ \vuht^{k}} }^2 }.
\end{aligned}
\end{equation*}
We finish the proof by summing up the above equation for $k=1,\ldots,N$ and multiplying with $\TS$.  
\end{proof}

\subsection{Proof of Lemma~\ref{lemBs}: useful estimates}\label{Apue}
\begin{proof}
Item~1 has been reported by \cite[Lemma 3.5]{FLMS_NS}. 
Item~2 has been reported by   \cite[Lemma 4.3]{GallouetIMA}. 
Item~4 has been reported by ~\cite[Chaper 9, Lemma 7]{FeiKaPok}.
We are only left with the proof of Item~3. 
We start the proof with the a-priori estimates on $\vvht$ 
\[
\norm{\vvht}_{L^\infty(L^6)}\leq \norm{\vuht}_{L^\infty(L^6)}+\norm{\vwht}_{L^\infty(L^6)}\aleq \tau^{-\frac12},
\]
where we used \eqref{inv_est} for $\vuht$ and \eqref{eq:west} for $\vwht$.
On one hand, for $\gamma \geq 2$, we employ \eqref{lemB1} to get
\[
\begin{aligned}
& 
\int_0^T  \sumEKh{  \abs{ \jump{\vrht}  \avg{\vvht \cdot \vc{n} }^-} }  \dt \\& \aleq  
\left( \int_0^T  \intEh{ \frac{ \jump{\vrht}^2}{\max\{\vrht^{\rm in}, \vrht^{\rm out}\}} \abs{\avg{\vvht \cdot \vc{n} }}} \right)^{1/2} \left( \int_0^T  \intEh{ \max\{\vrht^{\rm in}, \vrht^{\rm out}\} \abs{\avg{\vvht \cdot \vc{n} }}  }\right)^{1/2} 
\\& \aleq 
 h^{-1/2} \left( \norm{\vrht}_{L^2L^2} \norm{\vvht}_{L^2L^2}\right)^{1/2} 
\aleq  h^{-1/2} .
\end{aligned}
\]
On the other hand, it is easy to check for $\gamma\in(1,2)$ that 
$\Hc''(r)=a r^{\gamma-2} \geq a$ if $r\leq 1$ and $r \Hc''(r)= ar^{\gamma-1} \geq a$ if $r\geq 1$. Therefore 
\[ \Hc''(r)(1+r) \geq a \mbox{ for all } r \in(0,\infty)
\]
Applying these inequalities together with H\"older's inequality, and the estimate \eqref{uniform_bounds_h} we derive (by choosing $\vrht^\dagger$ conveniently and \eqref{eq:invtrace}) that
\[
\begin{aligned}
& 
\int_0^T  \sumEKh{  \abs{ \jump{\vrht}  \avg{\vvht \cdot \vc{n} }^-} }  \dt 
\\& \leq \frac2{\sqrt{a}}
\int_0^T  \intEh{ \sqrt{ \Hc''(\vrht^\dagger)} \abs{ \jump{\vrht} } \sqrt{\abs{\avg{\vvht \cdot \vc{n} }}}  \sqrt{(1+\vrht^\dagger) \abs{\avg{\vvht \cdot \vc{n} }}}  }
\\& \leq  \frac2{\sqrt{a}} 
\left( \int_0^T  \intEh{ \Hc''(\vrht^\dagger)  \jump{\vrht}^2 \abs{\avg{\vvht \cdot \vc{n} }}} \right)^{1/2} 
\left( \int_0^T  \intEh{ \abs{\avg{\vvht \cdot \vc{n} }} + \abs{\vrht^\dagger \avg{\vvht \cdot \vc{n} }}    }\right)^{1/2} 
\\& \aleq 
 h^{-1/2} \left( \norm{\vvht}_{L^1L^1}^{1/2}  + \norm{\vrht}_{L^2L^{6/5}}\norm{\vvht}_{L^2L^6}  \right)^{1/2} 
\aleq h^{-\frac12} \TS^{-\frac14} \norm{\vrht}_{L^\infty L^{6/5}}^{1/2} 
=:I_1 .
\end{aligned}
\]
Then, for $\gamma \in[6/5, 2)$ we have $I_1 \aleq h^{-\frac12}\TS^{-\frac14}$. 
Concerning $\gamma \in (1,6/5)$ we deduce by inverse estimate~\eqref{inv_est} that 
$ I_1 \aleq h^{-\frac12}\TS^{-\frac14} h^{\frac32(\frac{5}{6}-\frac{1}{\gamma})}
 \norm{\vrht}_{L^\infty L^{\gamma}}^{1/2}   \aleq h^{\frac{5\gamma-6}{4\gamma}-\frac12}\TS^{-\frac14}$, 
which completes the proof of the first estimate \eqref{con_r}.

Similarly, we prove the second estimate \eqref{con_m} in two steps. First for $\gamma \geq 2$ we may derive it due to H\"older's inequality, trace theorem, and the inverse estimate~\eqref{inv_est} that 
\[
\begin{aligned}
& 
\int_0^T  \sumEKh{  \abs{ \jump{\vrht} {\Pim\vuht} \avg{\vvht \cdot \vc{n} }^-} }  \dt 
\\& \leq  
\left( \int_0^T  \intEh{ \frac{ \jump{\vrht}^2}{\max\{\vrht^{\rm in}, \vrht^{\rm out}\}} \abs{\avg{\vvht \cdot \vc{n} }}} \right)^{1/2} 
 \left( \int_0^T  \intEh{ \max\{\vrht^{\rm in}, \vrht^{\rm out}\} \left({\Pim\vuht}\right)^2\abs{ \avg{\vvht \cdot \vc{n} }}  }\right)^{1/2}
\\& \aleq 
 h^{-1/2} \left( \norm{\vrht}_{L^\infty L^{2}} \norm{\vuht}_{L^2L^6}^2 \norm{\vvht}_{L^\infty L^6} \right)^{1/2} 
\aleq  h^{-1/2} \TS^{-\frac14}.
\end{aligned}
\]

Next, we proceed to show the second estimates for $\gamma \in(1, 2)$. 
\[
\begin{aligned}
& 
\int_0^T  \sumEKh{  \abs{ \jump{\vrht}{\Pim\vuht}  \avg{\vvht \cdot \vc{n} }^-} }  \dt 
\\& \leq  
\left( \int_0^T  \intEh{  \jump{\vrht^{\gamma/2}}^2 \abs{\avg{\vvht \cdot \vc{n} }}} \right)^{1/2} \left( \int_0^T  \intEh{  \Ov{\vrht^{1-\gamma/2} }^2\abs{\avg{\vvht \cdot \vc{n} }} \left({\Pim\vuht}\right)^2  }\right)^{1/2} 
\\& \aleq 
h^{-1/2}\left( \int_0^T \norm{\vrht}_{L^{2(2-\gamma)}}^{2-\gamma} \norm{ \vvht}_{L^6} \norm{\vuht}_{L^6}^2 \dt   \right)^{1/2}
\aleq h^{-1/2}\norm{\vrht}_{L^\infty L^{2(2-\gamma)}}^{(2-\gamma)/2} \norm{ \vvht}_{L^\infty L^6}^{1/2} \norm{\vuht}_{L^2L^6}
\\ &\aleq
h^{-1/2}\TS^{-\frac14} \norm{\vrht}_{L^\infty L^{2(2-\gamma)}}^{(2-\gamma)/2}
 =: I_2,
\end{aligned}
\]
where we have used the algebraic inequality for $\gamma \in (1,2)$ that
$ \jump{\vrht}^2 \leq \jump{(\vrht^{\gamma/2})}^2 \left(\Ov{\vrh^{1-\gamma/2}}\right)^2$.  
If $\frac43 \leq \gamma$ it follows (as before) that  
$ I_2  \aleq    h^{-1/2} \TS^{-\frac14}$. 
On the other hand, if $1<\gamma< \frac{4}{3}$  we complete the proof by the inverse estimates~\eqref{inv_est} and find
$
I_2  \leq  
h^{-1/2}\TS^{-\frac14}\norm{\vrht}_{L^\infty(L^{2(2-\gamma)})}^\frac{2-\gamma}2
\leq h^{-1/2}\TS^{-\frac14}\big(h^{\frac{3}{2(2-\gamma)}-\frac{3}{\gamma}}\norm{\vrht}_{L^\infty(L^{\gamma})}\big)^\frac{2-\gamma}2
= h^{-1/2}\TS^{-\frac14} h^\frac{9\gamma-12}{4\gamma}
$
which finishes the estimate.
\end{proof}


\bibliographystyle{siamplain}

\begin{thebibliography}{10}

\bibitem{BF}
M.~Bal\'{a}zsov\'{a}, M. Feistauer and M.~Vlas\'{a}k. 
\newblock Stability of the ALE space--time discontinuous Galerkin method for nonlinear convection-diffusion problems in time--dependent domains.
\newblock
{\em  ESAIM Math. Model. Numer. Anal.}{\bf 52}(6): 2327--2356, 2018.


\bibitem{BelGalKye13}
J.~Bemelmans, G.~P. Galdi, and M.~Kyed.
\newblock On the steady motion of a coupled system solid-liquid.
\newblock {\em Mem. Amer. Math. Soc.}, 226(1060):vi+89, 2013.

\bibitem{Bou07}
Muriel Boulakia.
\newblock Existence of weak solutions for the three-dimensional motion of an
  elastic structure in an incompressible fluid.
\newblock {\em J. Math. Fluid Mech.}, 9(2):262--294, 2007.

\bibitem{BS}
D.~Breit, and S.~Schwarzacher.
\newblock  Compressible fluids interacting with a linear-elastic shell.
\newblock {\em Arch. Rational Mech. Anal.}, {\bf 228}:495--562, 2018.

\bibitem{FSIforBIO}
T.~Bodnar, G.~P. Galdi, and {\v{S}}.~Ne{\v{c}}asov{\'a},
  editors.
\newblock {\em Fluid-Structure Interaction and Biomedical Applications}.
\newblock Birkh{\"a}user/Springer, Basel, 2014.



\bibitem{Brenner_SINUM}
S.c.~Brenner. 
\newblock {P}oincar{\'e}--{F}riedrichs Inequalities for Piecewise $H^1$ Functions.
\newblock {\em SIAM J. Numer. Anal.} {\bf 41}(1): 306--324, 2003.



\bibitem{Brenner_MC}
S.c.~Brenner. 
\newblock {K}orn's inequalities for Piecewise $H^1$ vector fields.
\newblock {\em Math. Comput.}{\bf 73}(247): 1067--1087, 2004.




\bibitem{boffi}
D.~Boffi, F.~Brezzi, M.~Fortin.
\newblock {\em Mixed finite element methods and applications}.
\newblock Springer series in computational mathematics.
\newblock{Springer}, 2013



\bibitem{Boris1}
M.~Buka\v{c}, S.~\v{C}ani\'c, and B.~Muha. 
\newblock  A partitioned scheme for fluid-composite structure interaction problems.
\newblock {\em J.  Comput. Phys.}, {\bf 281}:493--517, 2015.

\bibitem{Canic1} M.~Buka\v{c}, S.~\v{C}ani\'c, J.~Tamba\v{c}ac,  Y.~Wang.
\newblock Fluid–structure interaction between pulsatile blood flow and a curved stented coronary artery on a beating heart: A four stent computational study 
\newblock {\em Computer Methods in Applied Mechanics and Engineering}
{\bf 350}: 679--700

\bibitem{Boris2}
M.~Buka\v{c}, and B.~Muha.
\newblock  Stability and convergence analysis of the kinematically coupled scheme and its extensions for the fluid-structure interaction.
\newblock {\em SIAM J. Numer. Anal.} {\bf 54}(5):3032--3061, 2016.

\bibitem{Boris3}
S.~\v{C}ani\'c, B.~Muha, and M.~Buka\v{c}
\newblock  Stability of the kinematically coupled b-scheme for fluid-structure interaction problems in hemodynamics.
\newblock {\em Int. J. Numer. Anal. Model.} {\bf 12}(1):54--80, 2015.

\bibitem{Ciarlet_elas}
P.~G.~Ciarlet.
\newblock Mathematical Elasticity Volume I: Three-Dimensional Elasticity. 
\newblock{\em Academic Press, Elsevier}, 1988.

\bibitem{Ciarlet_fem}
P.~G.~Ciarlet.
\newblock The Finite Element Method for Elliptic Problems.
\newblock{\em Classics in Applied Mathematics, Society for Industrial and Applied Mathematics}, 2002.

\bibitem{CiaII}
P.~G.~Ciarlet and A.~Roquefort. 
\newblock Justification of a two--dimensional nonlinear shell model of Koiter’s type.
\newblock{\em Chinese Annals of Mathematics}{\bf 22}(2): 129--144, 2001.

\bibitem{CSS2}
Daniel Coutand and Steve Shkoller.
\newblock The interaction between quasilinear elastodynamics and the
  {N}avier-{S}tokes equations.
\newblock {\em Arch. Ration. Mech. Anal.}, 179:303--352, 2006.




\bibitem{crouzeix_raviart}
M.~Crouzeix, and P.-A.~Raviart.  
\newblock  Conforming and nonconforming finite element methods for solving the stationary {S}tokes equations. 
\newblock {\em Rev. Fran\c{c}aise Automat. Informat. Recherche Op\'erationnelle S\'er. Rouge} {\bf 7}(3): 33--75, 1973. 



\bibitem{DFKV}
A.~Dervieux, C.~Farhat, B.~Koobus and M.~V\'{a}zquez.
\newblock Total energy conservation in ALE schemes for compressible flows. 
\newblock {\em Euro. J. Comput. Mech.} {\bf 19}(4): 337--363, 2010. 


\bibitem{DE}
B.~Desjardins and M.~J. Esteban.
\newblock On weak solutions for fluid-rigid structure interaction: compressible
  and incompressible models.
\newblock {\em Comm. Partial Differential Equations}, 25(7-8):1399--1413, 2000.

\bibitem{DEGLT}
B.~Desjardins, M.~J. Esteban, C.~Grandmont, and P.~Le~Tallec.
\newblock Weak solutions for a fluid-elastic structure interaction model.
\newblock {\em Rev. Mat. Complut.}, 14(2):523--538, 2001.



\bibitem{EG}
A.~Ern, and J-L.~Guermond.
\newblock Theory and Practice of Finite Elements. 
\newblock Applied Mathematical Sciences 159, Springer, 2004. 



%

\bibitem{FLMS_NS}
E.~Feireisl M.~Luk\'{a}\v{c}ov\'{a}-Medvid'ov\'{a},  H.~Mizerov\'{a}, and B.~She.
\newblock Convergence of a finite volume scheme for the compressible Navier--Stokes system.
\newblock {\em ESAIM: M2AN.}  {\bf 53}(6): 1957--1979, 2019. 

\bibitem{FLNNS}
E.~Feireisl M.~Luk\'{a}\v{c}ov\'{a}-Medvid'ov\'{a}, \v{S}. Ne\v{c}asov\'{a}, A. Novotn\'{y}, and B.~She.
\newblock Asymptotic Preserving Error Estimates for Numerical Solutions of Compressible Navier--Stokes Equations in the Low Mach Number Regime. 
\newblock {\em Multi.  Model.  Simul.} {\bf 16}(1): 150-183, 2018.




\bibitem{FeiKaPok}
E.~Feireisl, T.~Karper, and M.~Pokorn{\' y}.
\newblock {\em Mathematical theory of compressible viscous fluids: {A}nalysis and numerics}.
\newblock Birkh{\" a}user--Verlag, Basel, 2017.

\bibitem{FeireislNovotny09}
Eduard Feireisl and Anton\'{i}n Novotn\'{y}.
\newblock {\em Singular limits in thermodynamics of viscous fluids}.
\newblock Advances in Mathematical Fluid Mechanics. Birkh\"{a}user Verlag,  Basel, 2009.

\bibitem{feireisl1} E. Feireisl, A. Novotn\'y, H. Petzeltov\'a: On the existence of globally defined weak solutions to the Navier--Stokes equations of compressible isentropic fluids. J. Math. Fluid. Mech. 3, 358--392, 2001. 


\bibitem{FKP}
M. Feistauer, V. Ku\v{c}era, and J. Prokopov\'{a}.
\newblock Discontinuous Galerkin solution of compressible flow in time-dependent domains.
\newblock {\em Mathematics and Computers in Simulation.}, {\bf 80}(8):1612--1623, 2010.


\bibitem{FRW}
S. Frei, T. Richter, and T. Wick. 
\newblock Long-term simulation of large deformation, mechano-chemical fluid-structure interactions in ALE and fully Eulerian coordinates.
\newblock {\em J.  Comput. Phys.}, {\bf 321}: 874--891, 2016.


\bibitem{GallouetMAC}
T.~Gallou{\"e}t, D.~Maltese, and A.~Novotn{\'y}.
\newblock Error estimates for the implicit MAC scheme for the compressible Navier--Stokes equations. 
\newblock{\em Numer. Math.} {\bf 141}(2): 495--567, 2019.

\bibitem{GallouetIMA}
T.~Gallou{\"e}t, R.~Herbin, D.~Maltese, and A.~Novotn{\'y}.
\newblock Error estimates for a numerical approximation to the compressible barotropic Navier--Stokes equations. 
\newblock{\em IMA J. Numer. Anal.} {\bf 36}: 543--592, 2016.


\bibitem{CG}
C{\'e}line Grandmont.
\newblock Existence of weak solutions for the unsteady interaction of a viscous
  fluid with an elastic plate.
\newblock {\em SIAM J. Math. Anal.}, 40:716--737, 2008.

\bibitem{GraHil16}
C\'eline Grandmont and Matthieu Hillairet.
\newblock Existence of global strong solutions to a beam-fluid interaction
  system.
\newblock {\em Arch. Ration. Mech. Anal.}, 220:1283--1333, 2016.



\bibitem{HS_MAC}
R.~Ho\v{s}ek and B.~She. 
\newblock Stability and consistency of a finite difference scheme for compressible viscous isentropic flow in multi-dimension.
\newblock{\em J. Numer. Math.} {\bf 26}(3): 111--140, 2018.


\bibitem{Hundertmark}
A. Hundertmark-Zau\v{s}kov\'{a}, and M. Luk\'{a}\v{c}ov\'{a}-Medvi\v{d}ov\'{a}. 
\newblock Numerical study of shear-dependent non-Newtonian fluids in compliant vessels. 
\newblock {\em Comput. Math. App.} {\bf 60}(3): 572--590, 2010. 





\bibitem{Daniel}
D.~Lengeler and M.~R{\u u}{\v{z}}i{\v{c}}ka.
\newblock Weak solutions for an incompressible {N}ewtonian fluid interacting
  with a {K}oiter type shell.
\newblock {\em Archive for Rational Mechanics and Analysis}, 211(1):205--255,
  2014.

\bibitem{Lukacova}
M. Luk\'{a}\v{c}ov\'{a}-Medvi\v{d}ov\'{a}, G. Rusnakov\'a, A. Hundertmark-Zau\v{s}kov\'{a}.
\newblock Kinematic splitting algorithm for fluid--structure interaction in hemodynamics. \newblock {\em Comput. Methods Appl. Mech. Engrg.} {\bf 265}: 83--106, 2013. 

\bibitem{Karper}
T.~Karper.
\newblock A convergent FEM-DG method for the compressible Navier--Stokes equations.
\newblock {\em Numer. Math.} {\bf 125}(3): 441--510, 2013.

\bibitem{kamakoti2004fluid}
Ramji Kamakoti and Wei Shyy.
\newblock Fluid--structure interaction for aeroelastic applications.
\newblock {\em Progress in Aerospace Sciences}, 40:535--558, 2004.

\bibitem{K1}
W.~T. Koiter.
\newblock On the foundations of the linear theory of thin elastic shells. {I},
  {II}.
\newblock {\em Nederl. Akad. Wetensch. Proc. Ser. B 73 (1970), 169-182; ibid},
  73:183--195, 1970.

\bibitem{Kosik}
   A.~Kos\'{i}k, M.~Feistauer, M.~Hadrava, and J.~Hor\'{a}\v{c}ek.
\newblock Numerical simulation of the interaction between a nonlinear elastic structure and compressible flow by the discontinuous Galerkin method. 
\newblock {\em Appl. Math. Comput.}, {\bf 267}: 382--396, 2015.

\bibitem{Li2} P.\,L. Lions: Mathematical topics in fluid mechanics. Vol. 2. Compressible models. Oxford Lecture Series in Mathematics and its Applications, 10. Oxford Science Publications, The Clarendon Press, Oxford University Press, New York, 1998.

\bibitem{BorSun13}
Boris Muha and Sun{\v c}ica {\v C}ani{\'c}.
\newblock Existence of a {W}eak {S}olution to a {N}onlinear
  {F}luid--{S}tructure {I}nteraction {P}roblem {M}odeling the {F}low of an
  {I}ncompressible, {V}iscous {F}luid in a {C}ylinder with {D}eformable
  {W}alls.
\newblock {\em Arch. Ration. Mech. Anal.}, 207(3):919--968, 2013.

\bibitem{SunBorMulti}
Boris Muha and Sun{\v{c}}ica {\v{C}}ani{\'c}.
\newblock Existence of a solution to a fluid--multi-layered-structure
  interaction problem.
\newblock {\em J. Differential Equations}, 256:658--706, 2014.

\bibitem{MuhSch19}
B.\ Muha and S.\ Schwarzacher.
\newblock Existence and regularity for weak solutions for a fluid interacting
  with a non-linear shell in 3d.
\newblock {\em arXiv preprint arXiv:1906.01962}, 2019.

\bibitem{Ric17}
T.~Richter.
\newblock {\em Fluid-structure interactions: models, analysis and finite
  elements}, volume 118.
\newblock Springer, 2017.

\bibitem{TW}
S.~Trifunovi\'{c}a and Y.-G.~Wang. 
\newblock Existence of a weak solution to the fluid-structure interaction problem in 3D.
\newblock {\em J. Differential Equations}, {\bf 268}(4):1495--1531, 2020.


\end{thebibliography}

\end{document}